\documentclass{amsart}
\usepackage[utf8]{inputenc}

\usepackage{amsmath,amsthm,amsfonts,amssymb,bm,bbm,enumerate,
graphicx,appendix,pgfplots,subfig}

\pgfplotsset{width=10cm,compat=1.9}
\DeclareGraphicsExtensions{.pdf,.png,.jpg}
\usepackage[font=small,skip=0pt]{caption}

\parskip \medskipamount
\RequirePackage[numbers]{natbib}
\RequirePackage[colorlinks,citecolor=blue,urlcolor=blue]{hyperref}
\RequirePackage{graphicx}

\usepackage{tikz}
\usetikzlibrary{shapes}

\usetikzlibrary{decorations.markings}

\tikzset{middlearrow/.style={
        decoration={markings,
            mark= at position 0.5 with {\arrow{#1}} ,
        },
        postaction={decorate}
    }
}

\newtheorem{theorem}{Theorem}[section]
\newtheorem{lemma}{Lemma}[section]
\newtheorem{proposition}{Proposition}[section] 
\newtheorem{corollary}{Corollary}[section]
\newtheorem{algorithm}{Algorithm}[section]
\newtheorem{assumption}{Assumption}[section]
\newtheorem{definition}{Definition}[section]
\newtheorem{remark}{Remark}[section] 
\DeclareMathOperator{\sgn}{sgn}
\DeclareMathOperator*{\argmin}{arg\,min}
\DeclareMathOperator*{\argmax}{arg\,max}

\begin{document}

\title{Speed Up Zig-Zag}
\author{G. Vasdekis and G. O. Roberts}

\begin{abstract}
{The Zig-Zag process is a Piecewise Deterministic Markov Process, efficiently used for simulation in an MCMC setting. As we show in this article, it fails to be exponentially ergodic on heavy tailed target distributions. We introduce an extension of the Zig-Zag process by allowing the process to move with a non-constant speed function $s$, depending on the current state of the process. We call this process Speed Up Zig-Zag (SUZZ). We provide conditions that guarantee stability properties for the SUZZ process, including non-explosivity, exponential ergodicity in heavy tailed targets and central limit theorem. 
Interestingly, we find that using speed functions that induce explosive deterministic dynamics may lead to stable algorithms that can even mix faster. 
We further discuss the choice of an efficient speed function by providing an efficiency criterion for the one-dimensional process and we support our findings with simulation results.

\noindent \textbf{MSC2020 subject classifications:} Primary 60J25 ; Secondary 65C05 , 60F05.
}
\end{abstract}

\maketitle

%%%%%%%%%%%%%%%%%%%%%%%%%%%%%%%%%%%%%%%%%%%%%%
%% Please use \tableofcontents for articles %%
%% with 50 pages and more                   %%
%%%%%%%%%%%%%%%%%%%%%%%%%%%%%%%%%%%%%%%%%%%%%%
%\tableofcontents

\section{Introduction}
Piecewise deterministic Markov processes (PDMP) have recently emerged as a new way to construct MCMC algorithms. Traditional MCMC algorithms employ discrete time Markov Chains to generate samples from a target distribution which is invariant for the chain, and subsequently use these samples to numerically estimate intractable expectations of functions of interest.
 By construction, standard MCMC algorithms like Random Walk Metropolis \cite{metropolis:53}, MALA  \cite{besag.mala:94}, etc. are time-reversible with respect to their target distribution.
 However, there is by now substantial evidence that reversible MCMC methods can be  significantly outperformed (in terms of mixing times and variances of estimators) by non-reversible ones 
  (see for example \cite{hwang.hwang.sheu:93, diaconis.holmes.neal:00, sun.gomez.schmidhuber:12, chen.hwang:13, lelievre.nier.pavliotis:13, bierkens:14, duncan.lelievre.pavliotis:15}).
Some PDMPs such as the Bouncy Particle Sampler \cite{cote.vollmer.doucet_bps:2018} and the Zig-Zag sampler \cite{bierkens.roberts_superefficient:2019} can be implemented directly and free from numerical error,
 providing a source of genuinely non-reversible MCMC algorithms. 
 %Given that these processes are intrinsically non-reversible, this gives us a way to construct non-reversible algorithms.

 The one dimensional Zig-Zag algorithm appeared in \cite{bierkens.roberts_scaling:2017} as a scaling limit of the Lifted Metropolis-Hastings (see \cite{turitsun.chertkov.vucelja:11,diaconis.holmes.neal:00}) applied to the Curie-Weiss model (see \cite{levin.luczak.peres:07}), although a simpler version of the process was introduced in \cite{goldstein:51} as the telegraph process (see also \cite{kac:74, fontbona.guerin.malrieu:16, fontbona.guerin.malrieu:12}). 
 The process was later extended in higher dimensions in \cite{bierkens.roberts_superefficient:2019} and has been proposed as a PDMP which can be used as an MCMC algorithm to target posterior distributions (see also \cite{fearnhead.bierkens.pollock.roberts:18, vanetti.cote.deligiannidis.doucet:17}). In \cite{bierkens.roberts_superefficient:2019}, the authors also introduce some variants of the algorithm that use the technique of sub-sampling, improving computational efficiency when the target distribution is obtained from a Bayesian analysis involving a large data set. Further literature on the topic includes \cite{ bierkens.duncan:17, bierkens.cote.doucet.duncan.fearnhead.lienart.roberts.vollmer:18, bierkens.nyquist.schlottke:19, bierkens.lunel:19, bierkens.grazzi.van-der-meulen.chauer:20, bierkens.grazzi.kamatani.roberts:20}.
 
\cite{bierkens.roberts.zitt:2019} proves ergodicity and exponential ergodicity of the Zig-Zag process in arbitrary dimension. A crucial assumption required for exponential ergodicity  in that work is that the target density has exponential or lighter tails. This paper will demonstrate the converse: the Zig-Zag sampler fails to be exponentially ergodic when the target distribution has tails  thicker than any exponential distribution, i.e. it is a heavy tailed. In fact, polynomial rates of convergence have been proven in \cite{andrieu.dobson.wang:21} for the process in arbitrary dimension, while \cite{vasdekis.roberts.2:21} proves tight polynomial rates of convergence in the total variation distance, for the one-dimensional process, when the target has tails that decay like a Student distribution. 

In order to address the problem of slow mixing on heavy tails, we introduce a variant of the Zig-Zag process, called Speed Up Zig-Zag (SUZZ). The idea behind the process has a similar spirit to the work of \cite{livingstone:21} and \cite{neklyudov.bondesan.welling:21}. In our case, instead of only permitting the process to move with unit speed, we allow it to have a positive position-dependant speed. This assists the exploration of the tails and subsequent return to the high density areas of the distribution more rapidly. We note that if the speed function is large enough, the solution to the ODE that governs the behaviour of the SUZZ process may potentially explode in finite time. Large part of the theory in this article focuses on proving that  such dynamics are mathematically acceptable in the context of MCMC. Furthermore, for carefully chosen speed functions, these ODEs and the induced SUZZ process can be numerically simulated exactly.
%In this paper, we will not force the speed function to be small enough to create dynamics that do not explode in finite time. Instead we will allow deterministic, which may explode in finite time, something that, t
 Although explosive deterministic dynamics have been mentioned in the past (see for example, Example 2.1.3 of \cite{riedler:11}), to the best of our knowledge, this is the first use of explosive dynamics within the literature of PDMPs for MCMC.

The rest of this paper is organised as follows. In Section 2 we recall the definition of the Zig-Zag process and we prove its lack of exponential ergodicity on heavy tails. Motivated by this slow convergence result, in Section 3 we define the Speed Up Zig-Zag (SUZZ) process and we establish stability and convergence properties. Theorem \ref{non.explo:1} proves that under certain conditions on the speed function, the process is non-explosive. Theorem \ref{inv.measure.speed.up:1} proves that the process has the distribution of interest as invariant. Theorem \ref{geom.ergo.suzz:1} proves that the process is exponentially ergodic and Theorem \ref{skeleton.clt:0} that it satisfies a Central Limit Theorem. Theorem \ref{geom.ergo.light.tails:1} proves that when the target has light tails, the SUZZ process is exponentially ergodic, essentially under the same conditions as in the original Zig-Zag, while Proposition \ref{special.case.theorem:1} proves exponential ergodicity of the SUZZ process for a family of heavy tailed distributions, with some specific, practical choices of speed functions. Corollary \ref{original.zz.geom.ergo:1} proves exponential ergodicity of the original Zig-Zag on light tailed targets, relaxing the assumptions of \cite{bierkens.roberts.zitt:2019}. Furthermore, focusing on the one-dimensional SUZZ process, Theorem \ref{uniform.ergodicity:1} proves that under explosive deterministic dynamics the process is uniformly ergodic and provides weaker assumptions to prove that the process is exponentially ergodic.
In Section 4 we focus on the one-dimensional process and discuss how the choice of the speed function can improve algorithmic efficiency. Using Proposition \ref{efficiency.proposition:1} we write the asymptotic variance of the one-dimensional process as a function of the speed function, which allows us to introduce a minimisation problem characterising %as the way to find 
the optimal speed function for one-dimensional SUZZ within an MCMC context. Finally in Section 5 we describe some numerical results, comparing the efficiency of different algorithms on one-dimensional and twenty-dimensional distributions. The Appendices contain the proofs of the main results along with some other useful information (e.g. how to formally construct the process or how to solve the deterministic ODE and construct the deterministic paths of the process). 

\section{The original Zig-Zag Process}\label{definition.algorithms:1}
%\subsection{Notation}\label{notation:00}
%PDMP $\Phi$, $E$ $C_c^b(E)$, $\theta_{-i},\mathcal{D}[0,1], O_n$
%\subsection{The Zig-Zag Process and Ergodicity}\label{zig.zag:1}
Here we give a brief introduction to the Zig-Zag process (which in this article we will refer to as original Zig-Zag), recalling some basic properties and  proving it has a sub-exponential convergence  rate for heavy-tailed distributions. The $d$-dimensional original Zig-Zag process $(Z_t)_{t \geq 0}=((X_t,\Theta_t))_{t \geq 0}$ is a PDMP with state space $E=\mathbb{R}^d \times \{ \pm 1 \}^d$. One can think of the process as a particle moving in $\mathbb{R}^d$ along one of $2^d$ possible straight lines. When the process is at point $(x,\theta) \in E$ the particle is at point $x \in \mathbb{R}^d$ and moves with constant velocity $\theta \in \{ \pm 1 \}^d$. This means that the process moves according to the ODE
\begin{equation}\label{zig.zag.ode:01}
\left\{ \begin{array}{l}
\dfrac{dX_t}{dt}=\theta \vspace{1.5 mm}\\ 
\dfrac{d\Theta}{dt}=0.
\end{array} \right.
\end{equation}
To each of the $d$ coordinates, we let $T_i$ denote the first event of a non-homogeneous Poisson Process of rate $m_i(t)=\lambda_i^{ZZ}(x+ \theta t,\theta)$, for $i=1,...,d$ and for some function $\lambda_i^{ZZ} :E \rightarrow [0,+\infty)$. We assume that all $m_i$ are locally integrable. 
%For each of the processes we generate its first arrival time $T_i$, we pick the smallest of them $T=\min \{ T_i, i=1,...,d \}$ and 
Let $T=\min_{i \in \{ 1,...,d \}}T_i$ and $j=\argmin_{i \in \{ 1,...,d \}} \{ T_i \}$. The process moves with velocity $\theta$ until time $T$ at which time its velocity changes to $F_j(\theta)$, where 
\begin{equation}
\begin{cases}\label{multi.zz.flip.ssymbol:1}
F_j(\theta)_j=-\theta_j\\
F_j(\theta)_i=\theta_i \text{ for } i\neq j,
\end{cases}
\end{equation}
and proceeds to move again with constant velocity $F_j(\theta)$ until it switches again, etc.

\begin{figure}[h]
\centering
\subfloat[T=100]{\includegraphics[height=5cm, width=6.4cm , trim= 2in 3in 2in 3.5in]{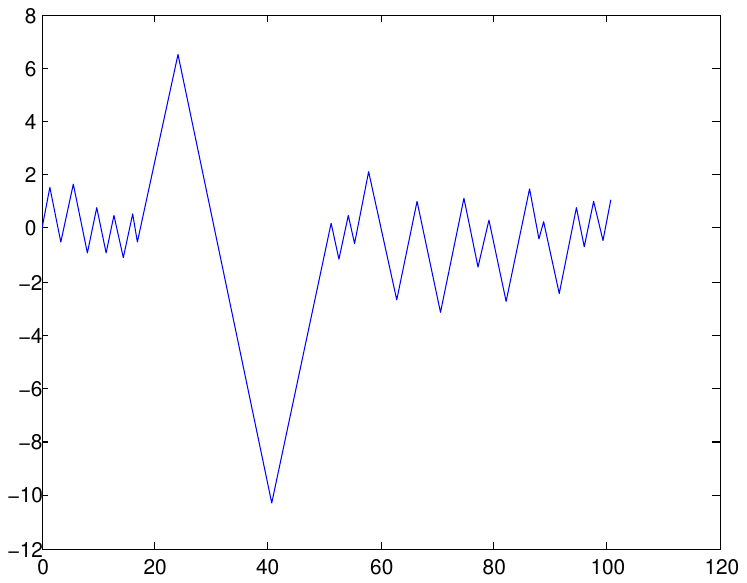}}\label{fig:zz.normal.d1}
\hfill
\subfloat[T=10000]{\includegraphics[height=5cm, width=6.4cm , trim= 2in 3in 2in 3.5in]{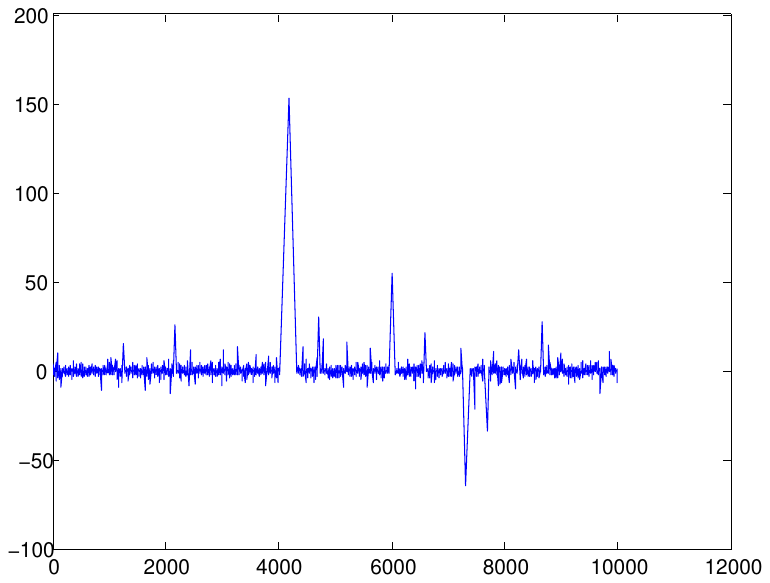}}\label{fig:zz.exponential.d1}
\caption{Trace plots of original Zig-Zag process, ran until time $T$, targeting one-dimensional Cauchy distribution.}\label{fig:zz.normal.exponential.d1}
\end{figure}
In \cite{bierkens.roberts_superefficient:2019} the goal was to target a probability measure on $E$ of the form \begin{equation}\label{zz.inv:3}
\mu(dx,d\theta)=\dfrac{1}{2^dH}\exp\{ -U(x) \}dxd\theta
\end{equation}
for some $U \in C^1$ with
$H=\int_{\mathbb{R}^d}\exp \{ -U(x) \}dx<\infty$. It is proven that the original Zig-Zag process has $\mu$ as invariant distribution when the rate functions are chosen according to 
\begin{equation}\label{zz.inv:2}
\lambda_i^{ZZ}(x,\theta)=[\theta_i \partial_iU(x)]^++\gamma_i(x,\theta_{-i}),
\end{equation}
where we write $\partial_i$ to denote the operator of the partial derivative on the $i$ coordinate, $a^+=\max\{ a,0 \}$ and $\gamma_i$ is a non-negative function that does not depend on the $i$ component of $\theta$. The special case where $\gamma_i(x,\theta_{-i})=0$ for all $x, i$ is known as the {\em canonical Zig-Zag}.

\begin{remark}
We note that
in many MCMC applications the goal is to target a measure 
\begin{equation}\label{zz.inv:4}
\pi(dx)=1/H \exp\{ -U(x) \}dx
\end{equation}
in $\mathbb{R}^d$. Technically, the original Zig-Zag process targets a measure $\mu$ on $E=\mathbb{R}^d \times \{ -1,+1 \}^d$, whose marginal distribution on $\mathbb{R}^d$ is $\pi$ and whose marginal distribution on $\{ -1,+1 \}^d$ is the uniform. One can then use the projection of the process on $\mathbb{R}^d$ to generate samples from the measure of interest $\pi$. Throughout this work we shall denote $\pi$ the measure of interest in $\mathbb{R}^d$ and $\mu$ the measure on $E$ given by (\ref{zz.inv:3}).
\end{remark}

\cite{bierkens.roberts.zitt:2019} demonstrates that assuming that $U$ grows at least linearly in the tails and appropriate smoothness conditions hold,
%and its first two derivatives 
the original Zig-Zag process $(Z_t)_{t \geq 0}$ converges to $\mu$ exponentially fast, i.e.
 there exist $M:E \rightarrow [1,+\infty)$ and $\rho <1$ such that for any $(x,\theta) \in E$
\begin{equation}\label{light.tails.geometric.ergodicity:1}
\| \mathbb{P}_{x,\theta}(Z_t \in \cdot)-\mu(\cdot) \|_{TV}\leq M(x,\theta)\rho^t \ .
\end{equation}
If (\ref{light.tails.geometric.ergodicity:1}) holds, we say that the process is exponentially ergodic.

 However \cite{bierkens.roberts.zitt:2019} does not cover the case where $U$ grows sub-linearly. In this scenario traditional MCMC algorithms based on random walk or Langevin proposals are known to converge at sub-exponential rates, see for example \cite{jarner:00, roberts.tweedie:96, livingstone.betancourt.byrne.girolami:16}. We will observe similar
 behaviour for the original Zig-Zag sampler. Figure \ref{fig:zz.qqplots.students.d1} provides the Q-Q plots of one-dimensional canonical Zig-Zag processes, targeting Student distributions with three different degrees of freedom. Each algorithm runs until $N=10^4$ switches of direction occur. The figure indicates that the process is less stable when targeting a Student distribution with lower degrees of freedom, which has more mass at the tails. 
 This instability is characterised by infrequent and unstable (heavy-tailed) excursions.
 In this article we will be mainly dealing with distributions that assign more mass at the tails than any exponential distribution. We give the following definition.
 
 \begin{definition}\label{def.heavy.tails:1}
 We say that a measure $\pi$ on $\mathbb{R}^d$ is heavy-tailed if for any $a>0$, if $B(0,R)$ is the ball of radius $R$, centered at $0$, then
 \begin{equation}
     \lim_{R  \rightarrow +\infty} \pi(B(0,R)^c) \exp\{ a R \} =+\infty.
 \end{equation}
 \end{definition}
 
 The following simple negative result for the original Zig-Zag sampler on heavy-tailed distributions was made known to us in personal correspondence with Professor Anthony Lee.
\begin{theorem}[Non-Exponential Ergodicity]\label{non.geometric.ergodicity:0}
Suppose that the original Zig-Zag targets a heavy tailed distribution. Then the process is not exponentially ergodic.
\end{theorem}
\begin{proof}[Proof of Theorem \ref{non.geometric.ergodicity:0}]
Suppose that the original Zig-Zag starts from $x=0$, $\theta \in \{ -1,+1 \}^d$. For any $t>0$, let $A_t= \{ x: \| x \|_2 >t \}$ be the complement of the ball of radius $t$ and let us fix a time $t>0$. Note that the process will always move with some velocity $\eta \in \{ -1,+1 \}^d$, and note that for any such $\eta$ we have $\| \eta \|_2 =\sqrt{d}$. Since the process moves with constant speed equal to $\sqrt{d}$, the original Zig-Zag will not have hit $A_{\sqrt{d}t}$ by time $t$. Therefore,
\begin{equation*}
\|\mathbb{P}_{0,\theta}(X_t \in \cdot)-\pi(\cdot)\|_{TV} \geq 
\left | \mathbb{P}_{0,\theta}\left(X_t \in A_{\sqrt{d}t}\right)-\pi\left(A_{\sqrt{d}t}\right) \right |=\pi\left(A_{\sqrt{d}t}\right)=\pi\left(B\left(0,\sqrt{d}t\right)^c\right).
\end{equation*}
So if we were  to have exponential ergodicity, we would have that there exists $M>0$ and $\rho<1$ such that for all $t>0$,
\begin{equation*}
\pi(B(0,t)^c) \leq \|\mathbb{P}_{0,\theta}(X_{\sqrt{d}^{-1}t} \in \cdot)-\pi(\cdot)\|_{TV} \leq M \left( \rho^{\sqrt{d}^{-1}} \right)^t,
\end{equation*}
 which creates a contradiction as $\pi$ has heavy tails.
\end{proof}

\begin{figure}[h]
\centering
\subfloat[Student(1) distribution ]{\includegraphics[scale=1.5, width=4.4cm , trim= 2in 3in 2in 3.5in]{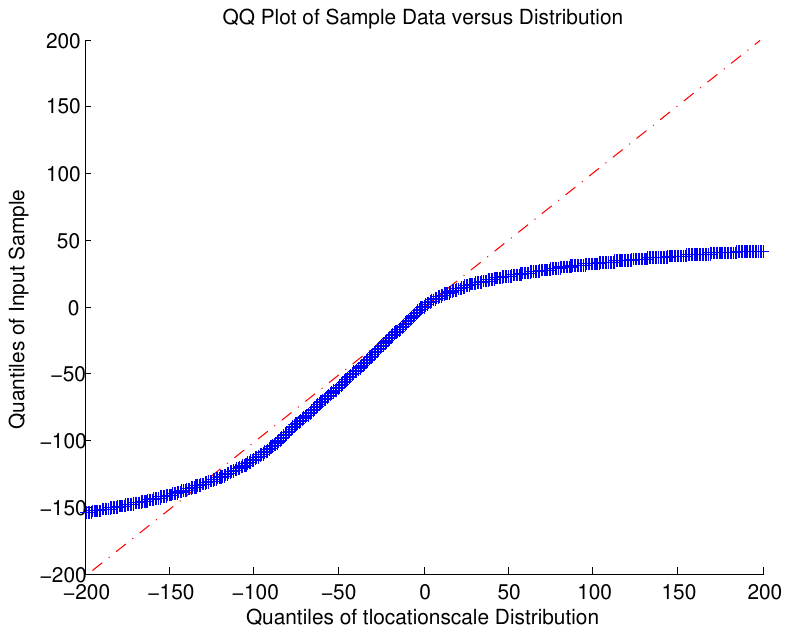}}\label{fig:zz.cauchy.d1}
\hfill
\subfloat[Student(8) distribution]{\includegraphics[scale=1.5, width=4.4cm , trim= 2in 3in 2in 3.5in]{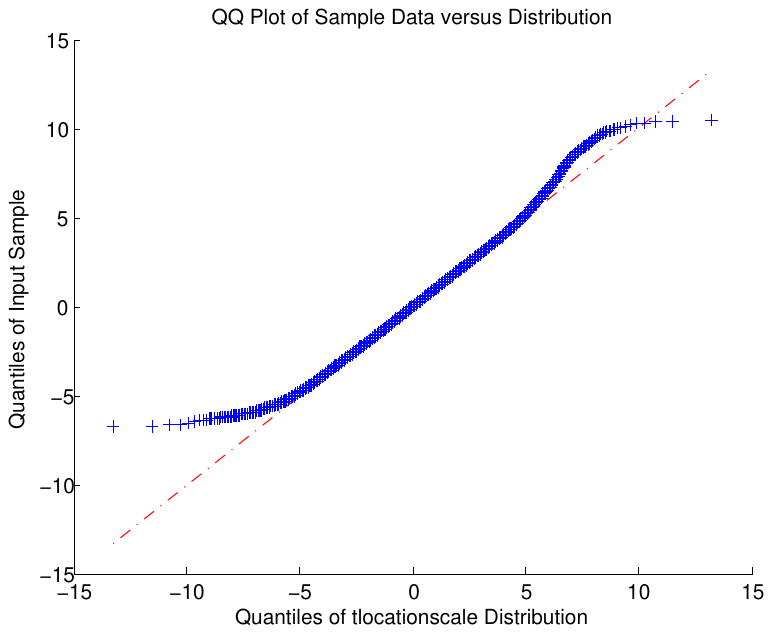}}\label{fig:zz.t8.d1}
\hfill
\subfloat[Student(500) distribution]{\includegraphics[scale=1.5, width=4.4cm , trim= 2in 3in 2in 3.5in]{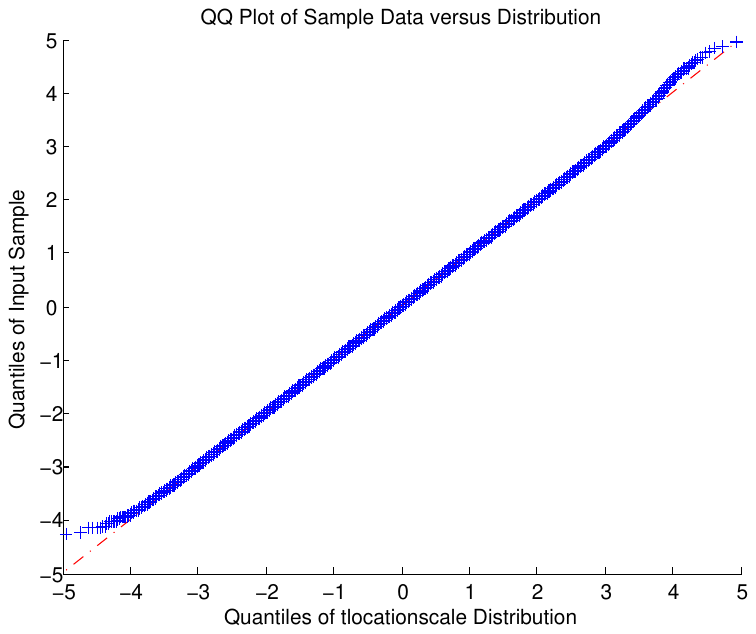}}\label{fig:zz.t500.d1}
\caption{Q-Q plots of canonical Zig-Zag for three Student distributions with increasing degrees of freedom. With Student($\nu$) we denote the Student distribution with $\nu$ degrees of freedom. Each algorithm runs until $N=10^4$ switches of direction occur. The closer the curve is to the diagonal line, the better the algorithm approximates the distribution. Clearly, the approximation is better when the target has higher degrees of freedom, i.e. lighter tails.}\label{fig:zz.qqplots.students.d1}
\end{figure}

Note that the same proof can be used for any other algorithm with constant speed function such as the Bouncy Particle Sampler with refreshing velocities taken from the unit sphere. Since these processes move with constant speed, they will not be able to explore the tails of the distribution sufficiently, which will result in a bad estimation of a target distribution that has heavy tails. Constant velocities are used mainly for the simplicity of the deterministic paths they provide. However, there are other types of deterministic paths that we can simulate exactly which do not move with unit speed. %For example, the one dimensional ODE initial value problem 
%\begin{equation*}
%\left\{ \begin{array}{l}
%\dfrac{dX_t}{dt}=X_t \vspace{1.5mm}\\
%X_0=1 
%\end{array} \right.
%\end{equation*}
%is explicitly solved by $X_t=e^t$.
This raises the question why not allow the original Zig-Zag process to move with non-constant velocities. We introduce this algorithm in the following section. 

\section{Speed Up Zig-Zag}\label{speed.up.zig.zag:00}
\subsection{Definition of the Algorithm}\label{speed.up.zig.zag.definition:00}
In order to address the problem of slow mixing on heavy tails we introduce a variant of the original Zig-Zag process. Instead of allowing the process to move with unit velocity, we allow it to have a positive speed depending on the current position. Since in high dimensions this might create a system of ODEs that is non-implementable, we only allow the process to move in directions $\{ \pm 1 \}^d$ as the original Zig-Zag does.

The state space will, again, be $E=\mathbb{R}^d \times \{ \pm 1 \}^d$. However, when the process is at point $(x,\theta) \in E$, it will move along the path $\{ x+\theta t , t \geq 0 \}$ with speed function $s$ that depends on the current position. Typically, this speed will increase the further the process is from the mode. After a random time that will depend on a Poisson process, as in the original Zig-Zag, the process will stop, one of the coordinates of $\theta$ will switch sign and the process will start moving again in the new direction. This will create excursions that tend to leave the area of high density and visit the tails quite often. At the same time, when the process is at the tails of the distribution, it can speed up and return to the centre fast enough. We shall call this process {\bf Speed Up Zig-Zag (SUZZ)}, with state space $E \cup \{ \partial \}$ (where $\partial$ is a graveyard state that
is needed for technical reasons), with $C^2$ {\bf speed function} $s: \mathbb{R}^d \rightarrow (0,\infty)$ and {\bf rate functions} $\lambda_i : E \rightarrow [0,+\infty)$ for all $i \in \{ 1,...,d \}$.

The SUZZ process  starting from $(x,\theta) \in E$ evolves as follows. Consider the following ODE system
\begin{equation}\label{suzz.heuristic.ode:01}
\left\{ \begin{array}{l}
\dfrac{dX_t}{dt}=\theta s(X_t), \ X_0=x  \vspace{1.5 mm}\\ 
\dfrac{d\Theta_t}{dt}=0, \ \Theta_0=\theta.
\end{array} \right.
\end{equation}
The procedure to solve (\ref{suzz.heuristic.ode:01}) will be given in Appendix \ref{ode.solution:00}. When the speed function is superlinear, the solution to (\ref{suzz.heuristic.ode:01}) explodes in finite time $t^*(x,\theta)$. Let $(X_t,\Theta_t)$ denote the solution to (\ref{suzz.heuristic.ode:01}) until time $t^*(x,\theta)\in(0,+\infty]$. For each coordinate $i \in \{ 1,...,d \}$, we let $T_1^i$ denote the first event of a non-homogeneous Poisson Process of rate $m_i(t)=\lambda_i(X_t,\theta)$, for $i=1,...,d$. 
Let $T_1=\min_{i \in \{ 1,...,d \}}T_1^i$ and $j=\argmin_{i \in \{ 1,...,d \}} \{ T_1^i \}$. The SUZZ process is defined until time $T_1$ to be the solution $(X_t,\Theta_t)_{t \leq T_1}$ to (\ref{suzz.heuristic.ode:01}). At time $T_1$ the direction/velocity $\Theta_{T_1}$ of the process switches from $\theta$ to $F_j(\theta)$, as in (\ref{multi.zz.flip.ssymbol:1}).
 In the case where $T_1>t^*(x,\theta)$, the process is defined as the solution to (\ref{suzz.heuristic.ode:01}) until $t^*(x,\theta)$ and then it moves to the graveyard state $\partial$. If $T_1<t^*(x,\theta)$, the process starts again from the new starting point $(X_{T_1},F_j[\theta])$ and evolves as before until time $T_2$ when the velocity switches again. Then the process starts again from the new position etc. This inductively defines the process until time 
\begin{equation}\label{definition.of.xi:00}
\xi=\lim_{n \rightarrow + \infty}T_n.
\end{equation}
In the case where $\xi < \infty$, $\xi$ is the first time that the process has had infinitely many switches of direction and the process moves to the graveyard state at time $\xi$.

A different way to describe the process is through its generator. We will later see that the class of functions with compact support and continuous first derivative, $C^1_c$, is contained in the domain of the strong generator of the SUZZ process and for any $f \in C^1_c$ the strong generator is given by
\begin{equation*}
    \mathcal{L}f(x,\theta):=\sum_{i=1}^d  \theta_i s(x) \partial_i f(x,\theta) + \lambda_i(x,\theta)(f(x,F_i(\theta))-f(x,\theta)) .
\end{equation*}

The process is therefore defined as a Piecewise Deterministic Process in \cite{davis:84} would be. The difference is that we allow the deterministic dynamics to have a finite explosion time, which Davis in \cite{davis:84} does not. We therefore need to be more careful in the analysis of the process. Let $O_m$ be the ball of radius $m$ centred around the origin $0$. We define 
\begin{equation}\label{definition.of.zeta.m:00}
\zeta_m=\inf \{ t \geq 0 : X_t \notin O_m \}
\end{equation}
 and let 
 \begin{equation}\label{definition.of.zeta:00}
 \zeta=\lim_{m \rightarrow \infty}\zeta_m.
 \end{equation}
  The two random variables $\xi$ and $\zeta$ quantify two types of explosion that can occur for the process. The first is that the process could have infinitely many switches in finite time and the second, that the process might diverge to infinity in finite time. %Note that the way we have constructed the process, we have $Z_t=(X_t,\Theta_t) \in E$ only for $t < \min\{ \xi,\zeta \}=\xi \wedge \zeta$. For formality let us define $Z_t=\partial$ for $t \geq \xi \wedge \zeta$. 
  
  An algorithmic description of the process is given in Supplementary A and a more formal construction of the process is given in Appendix \ref{formal.cosntruction:00}.

\subsection{Stability and Convergence Properties}\label{stability.convergence:00}
In this section we will study the process in more detail and provide some convergence results to the distribution of interest. We will ultimately provide assumptions that ensure exponential ergodicity of the process.

Similarly to the original Zig-Zag case, we will assume that we are trying to target the measure $\mu$ introduced in (\ref{zz.inv:3}) using a SUZZ with speed function $s$. Throughout this article we will assume that the rates are of the form

\begin{equation}\label{rates.formula:2}
\lambda_i(x,\theta)=[\theta_iA_i(x)]^++\gamma_i(x,\theta_{-i}),
\end{equation}
where $\gamma_i$ is a non-negative, locally bounded, integrable function that does not depend on the $i$th component of $\theta$ and
\begin{equation}\label{rates.formula:3}
A_i(x)=s(x)\partial_iU(x)-\partial_is(x).
\end{equation}
%\end{assumption}
 
Note that if we use the constant speed function $s \equiv 1$, we retrieve the original Zig-Zag rates when targeting $\mu$. We will later prove that if the rates satisfy  (\ref{rates.formula:2}) and some extra regularity conditions hold, the SUZZ process leaves the measure $\mu$ in 
(\ref{zz.inv:3}) invariant. 

Before we focus more on whether the SUZZ process targets the right distribution, we first need to consider some explosivity issues the process might have. 
Even in one dimension, picking a large enough speed function $s$ can lead to deterministic dynamics that explode in finite time. %For example, the ODE
%\begin{equation*}
%\begin{cases}
%\dfrac{dX_t}{dt}=X(t)^{1+\epsilon}\\
%X(0)=1
%\end{cases}
%\end{equation*}
%admits solution $X(t)={(1-\epsilon t)^{-1/\epsilon}} , t \leq \epsilon^{-1}$ which explodes at time $t^*=\epsilon^{-1}$. 
On a first glance, following deterministic dynamics that explode in finite time seems to be non-implementable, therefore non-desirable. However, as will be proven in Theorem \ref{non.explo:1}, frequent direction changes will almost surely rule out trajectories actually reaching $\infty$. Moreover, since the deterministic dynamics can reach infinity in finite time, when one reverses the time, the dynamics "come down from infinity" in finite time. This means that the time it takes to return to areas of high density may be independent of where the process starts from. % acts as an exit boundary for the proposed dynamics, it also acts as an entrance boundary (by using reversed velocity){\bf maybe write this not in terms of entry and exit boundaries, but as the comment below. This is to satisfy reviewer 3}.
This paves the way for the SUZZ algorithm to be uniformly ergodic.

One can allow the deterministic dynamics to be explosive, as long as a switching Poisson process is also introduced, having a very large intensity that will switch the direction of the process before it reaches the explosion time. We will provide conditions, the rates should satisfy, for the process to be a.s.\ non-explosive, even if the deterministic dynamics themselves are explosive.

Before that, we need to properly define how can the process explode. 

\begin{definition}
Let $\zeta$ be as in (\ref{definition.of.zeta:00}). 
 The process is called non-explosive if $\zeta=+\infty$ a.s.\
\end{definition}
We begin with the most essential assumption for the speed function.

\begin{assumption}[Speed Growth]\label{non.explo.condition:0}
$\lim_{\| x \| \rightarrow \infty}\exp \{ -U(x) \}s(x)=0$.
\end{assumption}

%\begin{proposition}\label{non.explosion:1}
%If the process satisfies Assumption \ref{non.explo.condition:0} then for any starting point $(x,\theta)$ the process will switch dynamics before the explosion time $t^*(x,\theta)$.
%\end{proposition}

\begin{remark}\label{explain.non.explo.condition:0}
We are imposing Assumption \ref{non.explo.condition:0} in order to ensure that the process will have switched the deterministic dynamics before they reach the explosion time. To see an example of how things could go wrong, consider a one-dimensional SUZZ with speed function $s$ targeting a distribution that has $U$ as minus log-likelihood. Assume that there exists a $x_0$ such that for all $x \geq x_0$, $(U(x)-\log s(x))'>0$, as would typically be the case.% as would be, for example, the case for any Gaussian target Gaussian and $s(x)$ a polynomial. 

Suppose that the process starts from
$(x,+1)$, $x \geq x_0$. The process evolves under the deterministic dynamics given in (\ref{suzz.heuristic.ode:01}) until the explosion time $t^*=t^*(x,+1)$. Consider the Poisson process with rate $\{ m(t)=\lambda(X_t,\theta), t \geq 0 \}$. %that dictates the switching time of the process. 
Then,
\begin{align*}
&\int_0^{t^*}m(t)dt=\int_0^{t^*}\lambda(X_t,+1)dt
=\int_x^{+\infty}\lambda(u,+1)\dfrac{1}{s(u)}du  = \int_x^{+\infty} U'(u)-\dfrac{s'(u)}{s(u)}du\\
&=\lim_{u \rightarrow +\infty}U(u)-\log s(u) - \left( U(x) -\log s(x) \right).
\end{align*}

 Therefore, assuming that Assumption \ref{non.explo.condition:0} does not hold and let's say that $s(u)\exp \{ -U(u) \} \\
 \xrightarrow{u \rightarrow +\infty}a>0$ we get that 
 \begin{equation*}
\lim_{u \rightarrow +\infty} U(u)-\log(s(u))=\log(a^{-1})<\infty,
 \end{equation*} 
 and therefore
 \begin{equation*}
 \mathbb{P}_{x,+1}(\text{no switches until time } t^*(x,+1))=\exp \left\{ -\int_0^{t^*(x,+1)}m(t)dt \right\}>0.
 \end{equation*}
 Therefore, if $t^*(x,+1)<\infty$ then the process has a positive probability to explode. % Even if $t^*(x,+1)=\infty$ the process still has a positive probability of never switching direction towards $-1$. 
 The same situation is experienced in higher dimensions assuming that for all coordinates $i$, $\partial_i (U(x)-\log s(x))>0$ for all $x=(x_1,...,x_d)$ for which $x_i$ is positive and very large. Furthermore, as we will see in Proposition \ref{switch.before.boom:1}, assuming that Assumption \ref{non.explo.condition:0} holds, forces the SUZZ process to a.s.\ switch direction before it reaches the explosion time. This forces us to adopt Assumption \ref{non.explo.condition:0}.
 \end{remark}

We will, also, make the following assumption. 
 
\begin{assumption}\label{ass.lyapunov:1}
Assume that for the refresh rates there exists $\bar{\gamma}$ such that for all $i \in \{ 1,...,d \}, (x,\theta) \in E$ , $\gamma_i(x,\theta_{-i})=\gamma_i(x) \leq \bar{\gamma}$.\\
Furthermore, assume that there exists $R>0$ and $A>0$ so that for all $x \notin B(0,R) $
\begin{equation}\label{non.decay.of.rates:1}
\sum_{i=1}^d|A_i(x)|>A> \max \{ 3d \bar{\gamma} , 4d(d-1)\bar{\gamma} \}.
\end{equation}
\end{assumption}

\begin{remark}
 When all refresh rates are zero, then $\bar{\gamma}=0$ and (\ref{non.decay.of.rates:1}) means that the overall switching rate is bounded away from zero, which seems essential in order to gain exponential ergodicity. More generally, the $A_i$'s describe the intention of the algorithm to switch from a direction leading to lower density areas, while the $\gamma_i$'s describe the intention of the algorithm to switch direction randomly. Large values of $\gamma_i$ would lead the algorithm to a random walk behaviour (see also \cite{bierkens.duncan:17}) and might decrease the convergence rate. Therefore, (\ref{non.decay.of.rates:1}) could be seen as a quantitative upper bound for the refresh rate.
 \end{remark}
 %When at least one refresh rate is positive, the lower bound in (\ref{non.decay.of.rates:1}) is somehow arbitrary and we would like to improve it. However we do expect a lower bound of the form $d\bar{\gamma}$. This is because $d\bar{\gamma}$ is the overall rate of switching because of refresh, whereas $\sum_{i}|A_i|$ describes the overall rate of switching because of the direction the particle moves towards. Heuristically, one should expect that the first rate should be less than the second in order to get good convergence results.{\bf better} 
 
 We also have to assume the following.
 
 \begin{assumption}\label{the.one.that.always.holds:1}
 If we iteratively define the functions 
 %$f_n: [0,+\infty) \rightarrow [0,+\infty)$ such that 
%begin{equation*}
 %   f_0(x)=x
%\end{equation*}
%and for any $n \geq 1$
%\begin{equation}\label{iterative.exponentials:1821}
 %   f_n(x)=\exp\{ f_{n-1}(x) \}-1
%\end{equation}
 $h_n: [0,+\infty) \rightarrow [0,+\infty)$ such that
\begin{equation}
    h_0(x)=x
\end{equation}
and for $n \geq 1$
\begin{equation}\label{iterative.logarithms:1}
    h_n(x)=\log(1+h_{n-1}(x)).
\end{equation}

then there exists an $n \in \mathbb{N}$ such that
\begin{equation}\label{the.arbitrary.second.der.A.assumption:2}
 \lim_{\| x \| \rightarrow \infty}\frac{h_n\left( s(x)\| \nabla \left( U(x)-\log s(x) \right)  \|_1 \right)}{ U(x)-\log s(x) }=0.
 \end{equation}
 \end{assumption}
 
 \begin{remark}
 Recall that, due to Assumption \ref{non.explo.condition:0}, we have that $U(x)-\log s(x) \xrightarrow{|x|\rightarrow \infty}+\infty$. Therefore, in practice Assumption \ref{the.one.that.always.holds:1} will almost always hold.
 \end{remark}
 
Finally, we make one more assumption.
 \begin{assumption}\label{ass.lyapunov:2}
For all $j \in \{ 1,...,d \}$, $A_j \in C^1$ and for any $\delta >0$,
\begin{equation}\label{the.arbitrary.second.der.A.assumption:1}
\lim_{\| x \| \rightarrow \infty} \sum_{i=1}^d\sum_{j=1}^d\dfrac{s(x) }{\sum_{k=1}^d |A_k(x)|} \  \dfrac{ |\partial_i A_j(x) | }{(1+|A_j(x)|)(1+\log(1+\delta|A_j(x)|))}=0.
 \end{equation}
 \end{assumption}

 \begin{remark}\label{weird.assumption.remark:321}
 
 This is a technical assumption used to prove the results of this section and it can be quite difficult to verify in practice for multi-dimensional targets. We believe, however, that it is not necessary for the results to hold. For example, in Section \ref{specific.stability.convergence:00} the desired properties for the SUZZ process are directly proved for a family of targets and with speed functions that do not satisfy Assumption \ref{ass.lyapunov:2}. %Furthermore, in Section \ref{simulations:00}, we present simulation results, indicating that the process can be very efficient in practice, even when Assumption \ref{ass.lyapunov:2} is not satisfied.
 
We note, however, that Assumption \ref{ass.lyapunov:2} generalises one made in \cite{bierkens.roberts.zitt:2019} to prove exponential ergodicity of the original Zig-Zag. Indeed, when $s(x)=1$, Assumption 3.4 writes 
\begin{equation*}
    \lim_{\| x \| \rightarrow \infty} \sum_{i=1}^d\sum_{j=1}^d\dfrac{1 }{\|  \nabla U(x) \|_1} \  \dfrac{ |\partial_i \partial_jU(x) | }{(1+|\partial_jU(x)|)(1+\log(1+\delta |\partial_jU(x)|))}=0.
\end{equation*}
for all $\delta>0$. This is weaker than 
\begin{equation*}
    \lim_{\| x \| \rightarrow \infty} \frac{\| Hess(U)(x) \|}{\|  \nabla U(x) \|_1}=0,
\end{equation*}
assumed in \cite{bierkens.roberts.zitt:2019}. 
The reader can see Example 5.2.9 of \cite{vasdekis:20} for one-dimensional examples where the target has tails asymptotically similar to those of a Student distribution and it is verified that (\ref{the.arbitrary.second.der.A.assumption:1}) holds.

\end{remark}

Our first main result is the following.
\begin{theorem}[Non-Explosion]\label{non.explo:1}
Assume that $s \in C^2$ is strictly positive, the rates satisfy (\ref{rates.formula:2}) and Assumptions \ref{non.explo.condition:0}, \ref{ass.lyapunov:1}, \ref{the.one.that.always.holds:1} and \ref{ass.lyapunov:2} hold.
Then the process is non-explosive, meaning that if $\zeta$ as in (\ref{definition.of.zeta:00}), then $\zeta=+\infty$ a.s. Furthermore, if $\xi$ as in (\ref{definition.of.xi:00}), then $\xi = +\infty$ a.s.
\end{theorem}

Furthermore, if we pick the switching rates according to (\ref{rates.formula:2}), then our non-explosive process leaves the target distribution of interest invariant. For this we need to make the following assumption in the case where the deterministic dynamics of the process are explosive.

\begin{assumption}\label{non.evanescence.assumption:1}
\begin{equation}\label{non.evanescence.assumption:2}
\lim_{\| x \| \rightarrow \infty}\| x \|^{d-1}s(x)\exp \{ -U(x) \}=0.
\end{equation}
\end{assumption}

\begin{remark} This is a stronger assumption than Assumption \ref{non.explo.condition:0}, since it imposes a more strict upper bound on the growth of the speed functions we can use. However, it still allows a lot of flexibility on the growth of $s$. Consider for example a $d$-dimensional Student distribution with $\nu$ degrees of freedom, i.e. $\pi(x)=\exp\{ -U(x) \} \sim 1/|x|^{d+\nu}$,  where we write $a(x) \sim b(x)$ to denote that $\lim_{|x| \rightarrow \infty}\frac{a(x)}{b(x)}=c$ for a constant $c>0$.  Assumption \ref{non.evanescence.assumption:1} implies that $s(x)/|x|^{1+\nu} \xrightarrow{|x| \rightarrow \infty}0$. Therefore, if $s(x) \sim |x|^{1+k}$, we have to impose the condition that $k < \nu$.

As will be seen in the proof of Theorem \ref{inv.measure.speed.up:1}, this assumption is only needed in the case of explosive deterministic dynamics. 
\end{remark}

We then have the following. 
\begin{theorem}[Invariant Measure]\label{inv.measure.speed.up:1}
Assume that the rates satisfy (\ref{rates.formula:2}) and Assumptions \ref{ass.lyapunov:1} , \ref{the.one.that.always.holds:1}, \ref{ass.lyapunov:2} and \ref{non.evanescence.assumption:1} hold. Assume that $s \in C^2$ is strictly positive. 
%\begin{equation}\label{rates.speed.up:1}
%\lambda(x,\theta)=[(s(x)U'(x)-s'(x))\theta]^++\gamma(x),
%\end{equation}
%where $\gamma$ is a non-negative function not depending on $\theta$.
Then, the SUZZ process has the measure $\mu$ in (\ref{zz.inv:3}) as invariant.
\end{theorem}
Crucially, under some further conditions on the speed function $s$, the SUZZ process is exponentially ergodic even when targeting some heavy tailed distributions.

\begin{theorem}[Exponential Ergodicity of SUZZ]\label{geom.ergo.suzz:1}
Let $(Z_t)_{t \geq 0}=(X_t,\Theta_t)_{t \geq 0}$ be a SUZZ process with strictly positive speed function $s \in C^2$. Suppose that the rates satisfy (\ref{rates.formula:2}) and Assumptions  \ref{ass.lyapunov:1},
\ref{the.one.that.always.holds:1}, \ref{ass.lyapunov:2} and \ref{non.evanescence.assumption:1} hold. Assume further that the function $U-\log s \in C^3$ and has a non-degenerate local minimum, i.e. there exists an $x_0 \in \mathbb{R}^d$ local minimum for $U-\log s$ such that the Hessian matrix $Hess(U-\log s)(x_0)$ is strictly positive definite. Finally, assume that $\mu$ introduced in (\ref{zz.inv:3}) is a probability measure. Then the SUZZ process is exponentially ergodic, i.e. there exists some $M:E \rightarrow [1,+\infty)$ and $\rho<1$ such that for any $(x,\theta) \in E$,
\begin{equation}\label{geom.ergo.suzz:11}
\| \mathbb{P}_{x,\theta}(Z_t \in \cdot)-\mu(\cdot) \|_{TV} \leq M(x,\theta)\rho^t.
\end{equation}
\end{theorem}

An immediate result due to Theorem 2 of \cite{chan.geyer:94} is the following CLT.
\begin{theorem}[Central Limit Theorem]\label{skeleton.clt:0}
 Suppose that all the assumptions of Theorem \ref{geom.ergo.suzz:1} hold. Let $\{Y_n , n \geq 0 \}$ be any skeleton of the SUZZ process (i.e. for some $\delta >0$, $Y_n=Z_{n\delta}$ for all $n \in \mathbb{N}$) and let $f:E \rightarrow \mathbb{R}$ such that there exists an $\epsilon>0$ with $\mathbb{E}_{\mu}[f^{2+\epsilon}]<\infty$. Then there exists a $\gamma_f^2 \in [0,\infty)$ such that
\begin{equation}\label{skeleton.clt:1}
\frac{\sum_{k=1}^n\left( f(Y_k)-\mu(f) \right) }{\sqrt{n}}\xrightarrow[D]{n \rightarrow \infty}Z 
\end{equation}
for some $Z \sim \mathcal{N}(0,\gamma_f^2)$.
\end{theorem}

Finally, in the case where the target has lighter tails (such that the gradient of the log-likelihood does not decay to zero) we can prove the convergence results for SUZZ under conditions that can be easily verified.

\begin{assumption}\label{light.tail.assumption:1}
Assume that $U-\log s \in C^2$ and there exists an $\tilde{M}>0$ such that the refresh rate $\gamma(x)$ of the SUZZ process satisfies $\gamma(x) \leq \tilde{M} s(x)$ for all $x \in \mathbb{R}^d$. Assume further that for some $n \in \mathbb{N}$, if $h_n$ as in (\ref{iterative.logarithms:1}),
\begin{equation*}
    \lim_{\| x \| \rightarrow \infty }\frac{h_n \left(\| \nabla (U(x) - \log s(x)) \| \right)}{ U(x)-\log s(x) }=0 , \ \lim_{\| x \|  \rightarrow \infty}\frac{\| Hess(U(x)-\log s(x)) \|}{\| \nabla \left( (U(x)-\log s(x) \right)  \|}=0,
\end{equation*}
and that there exists $R>0$ and $A>0$ so that for all $x \notin B(0,R) $
\begin{equation}\label{non.decay.of.rates.light.tails:1}
\|\nabla (U(x)-\log s(x) ) \|_1>A>\max \{ 3d \tilde{M} , 4d(d-1)\tilde{M} \}.
\end{equation}
\end{assumption}

\begin{theorem}\label{geom.ergo.light.tails:1}
Let $(Z_t)_{t \geq 0}=(X_t,\Theta_t)_{t \geq 0}$ be a SUZZ process with speed function $s \in C^2$ bounded away from $0$. 
\begin{itemize}
    \item Assume that the rates satisfy (\ref{rates.formula:2}) and Assumptions \ref{non.explo.condition:0} and \ref{light.tail.assumption:1} hold. Then the SUZZ process is non-explosive.
\item Assume further that either Assumption \ref{non.evanescence.assumption:1} holds or the deterministic dynamics are non-explosive.
Then, the SUZZ process has the measure $\mu$ in (\ref{zz.inv:3}) as invariant.
\item Assume further that the function $U-\log s \in C^3$ and has a non-degenerate local minimum, in the sense of Theorem \ref{geom.ergo.suzz:1}. Finally, assume that $\mu$ introduced in (\ref{zz.inv:3}) is a probability measure. Then the SUZZ process is exponentially ergodic.
\item Assuming the assumptions of the previous bullet, let $\{Y_n , n \geq 0 \}$ be any skeleton of the SUZZ process (i.e. for some $\delta >0$, $Y_n=Z_{n\delta}$ for all $n \in \mathbb{N}$) and let $f:E \rightarrow \mathbb{R}$ such that there exists an $\epsilon>0$ with $\mathbb{E}_{\mu}[f^{2+\epsilon}]<\infty$. Then, the CLT result of (\ref{skeleton.clt:1}) holds.
\end{itemize}
\end{theorem}

The conditions of Theorem \ref{geom.ergo.light.tails:1} can be seen as direct generalisations of assumptions made in \cite{bierkens.roberts.zitt:2019} for the original Zig-Zag. Therefore, Theorem \ref{geom.ergo.light.tails:1} guarantees that for reasonable speed functions, the convergence properties of the original Zig-Zag
carry over in SUZZ. This allows one to see the speed function as a tuning parameter for the original Zig-Zag, which could potentially increase the efficiency of the algorithm even in cases where the original Zig-Zag works well.

\subsection{Stability and convergence for practical choices of speed functions}\label{specific.stability.convergence:00}

Assumption \ref{ass.lyapunov:2} used in Theorems \ref{non.explo:1}, \ref{inv.measure.speed.up:1}, \ref{geom.ergo.suzz:1} and \ref{skeleton.clt:0} can be difficult or impossible to verify for some practical choices of speed functions. For this reason, in this section we will focus our attention on these particular, practical speed functions and we will establish convergence properties for a class of targets, some of which we will also use in simulations in section 5.

We will consider two speed functions, namely 
\begin{equation}\label{speed.d20:1}
s(x)=\left(1+ \| x  \|_2^2 \right)^{\frac{1+k}{2}}
\end{equation}
for $k=0$ and $k=1$. We will refer to the SUZZ algorithms induced by these two functions as SUZZ($0$) and SUZZ($1$) respectively. Note that SUZZ($0$) has non-explosive deterministic dynamics, while SUZZ($1$) has explosive ones, since the speed function grows super-linearly.

We have the following.

\begin{proposition}\label{special.case.theorem:1}
Assume that the target is of the form 
\begin{equation}\label{pi.subexponential:1}
 \pi(x)=\frac{1}{H}\exp\left\{ -\left( 1+ \| x \|_2^2 \right)^{a/2} \right\}
\end{equation}
for some $a \in (0,1)$ or of the form
\begin{equation}\label{pi.student:1}
    \pi(x)=\frac{1}{H}\left(  1+ \frac{1}{\nu} \| x \|_2^2\right)^{-\frac{\nu+d}{2}}
\end{equation}
for some $\nu$ satisfying
\begin{equation}\label{nu.lower.bound:1}
    \nu >%\frac{1}{2}\left( 1+ \left( 1+4\sqrt{2}\frac{1}{d} \right)^{\frac{1}{2}}  \right)^2
    \frac{27}{2}d^3+2-d.
\end{equation}
Assume also that $s$ is as in (\ref{speed.d20:1}) for $k=0$ or $k=1$. Then the SUZZ process with refresh rate $\gamma \equiv 0$ is non explosive, has $\mu$ as invariant, is geometrically ergodic and satisfies the CLT as in Theorem \ref{skeleton.clt:0}.
\end{proposition}

\begin{remark}\label{specific.remark:1}
Following the proof of Proposition \ref{special.case.theorem:1} in Appendix \ref{proof.of.proposition.special.case:00}, we can more generally have the conclusion of Proposition \ref{special.case.theorem:1} when the speed function is such that there exists a $K>0$ and $M_0>0$ such that for all $\| x \|_2 \geq K$, $\frac{s(x)}{\| x \|_2} \geq M_0$ and the target is such that for all $\| x \|_2 \geq K$, \begin{equation*}
    A_i(x)=c(x)\left( B\cdot x \right)_i,
\end{equation*}
where $B$ is a positive definite matrix such that for all for all $i \in \{ 1,...,d \}$, $b_{i i} - \sum_{j \neq i} |b_{i j}| \geq m >0$ and
if $M=\max\{ \sum_{j=1}^d|b_{ij}|, i=1,...,d \}$, then $c$ satisfies for all $\| x \|_2 \geq K$
\begin{equation*}
    c(x)\frac{\|  x \|_2^2}{s(x)} > \frac{27}{2}\frac{M}{m}d^3.
\end{equation*}
\end{remark}

\subsection{Comparison with results on the Original Zig-Zag}

In this section we will translate the assumptions and the results of Section \ref{stability.convergence:00} in the case of the original Zig-Zag process, which arises when we use the constant speed function $s(x)=1$. In this setting, we will see that all the assumptions made in Section \ref{stability.convergence:00} are weaker versions of assumptions made in \cite{bierkens.roberts.zitt:2019}. This will serve as a way to justify our assumptions and at the same time will allow us to prove exponential ergodicity of the original Zig-Zag process under weaker assumptions than the ones of Theorem 2 of \cite{bierkens.roberts.zitt:2019}.

Our first observation is that in the original Zig-Zag case where $s \equiv 1$, Assumption  \ref{non.evanescence.assumption:1} is implied by the following growth condition. 

\begin{assumption}\label{zz.assumption:101}
There exists $\epsilon>0 , c' \in \mathbb{R}$ such that for all $x \in \mathbb{R}$, $U(x)\geq (d+\epsilon) \log (\| x \|)-c'$. 
\end{assumption}

\begin{remark}
Assumption \ref{zz.assumption:101} is Growth condition 2 of \cite{bierkens.roberts.zitt:2019}, assumed in order to prove non-evanescence of the original Zig-Zag process.
\end{remark}

Secondly, we observe that in the setting of the original Zig-Zag, Assumption \ref{ass.lyapunov:1} is the following. 

\begin{assumption}\label{zz.assumption:102}
Assume that for the refresh rates there exists $\bar{\gamma}$ such that for all $i \in \{ 1,...,d \}, (x,\theta) \in E$ , $\gamma_i(x,\theta_{-i})=\gamma_i(x) \leq \bar{\gamma}$.\\
Furthermore, assume that there exists $R>0$ and $A>0$ so that for all $x \notin B(0,R) $
\begin{equation*}
\|\nabla U(x)\|_1 > A > \max \{ 3d \bar{\gamma} , 4d(d-1)\bar{\gamma} \}.
\end{equation*}
\end{assumption}

\begin{remark}
We observe that this is a weaker version of Growth condition 3 of \cite{bierkens.roberts.zitt:2019}, necessary for proving exponential ergodicity of the original Zig-Zag process. Instead of asking that $\lim_{\| x\| \rightarrow \infty}\| \nabla U(x) \|_1=+\infty$, we only ask that the limit is bounded below by a constant that may depend on the dimension of the space.
\end{remark}

Furthermore, in the case of the original Zig-Zag, Assumption \ref{the.one.that.always.holds:1} is the following.

\begin{assumption}\label{the.second.one.that.always.holds:1}
If $h_0(x)=x$ and for all $n \in \mathbb{N}$, $h_n$ is defined as in (\ref{iterative.logarithms:1}),
then there exists an $n \in \mathbb{N}$ such that
\begin{equation}\label{growth.always.verifiable:102}
 \lim_{\| x \| \rightarrow \infty}\frac{h_n \left(\| \nabla U(x)  \|_1 \right)}{  U(x)  }=0.
\end{equation}
\end{assumption}

\begin{remark}
We note that Assumption \ref{the.second.one.that.always.holds:1} is almost always true in any practical setting where $U \in C^1$. Furthermore, it is a relaxed version of Growth Condition 3 of \cite{bierkens.roberts.zitt:2019}.
\end{remark}

Finally, in the case of the original Zig-Zag, Assumption \ref{ass.lyapunov:2}, is equivalent to the following.

 \begin{assumption}\label{zz.assumption:103}
$U \in C^2$ and for all $\delta>0$
\begin{equation}\label{hess.grwoth.zz:101}
\lim_{\| x \| \rightarrow \infty} \dfrac{1}{\| \nabla U(x) \|_1} \sum_{j=1}^d \dfrac{\sum_{i=1}^d |\partial_i \partial_jU(x) | }{(1+|\partial_jU(x)|)(1+\log(1+\delta|\partial_jU(x)|))}=0.
 \end{equation}
 \end{assumption}

 \begin{remark}
As mentioned in Remark \ref{weird.assumption.remark:321}, Assumption \ref{zz.assumption:103} is weaker than Growth condition 3 of \cite{bierkens.roberts.zitt:2019}.% More specifically, (\ref{hess.grwoth.zz:101}) is always implied by Growth condition 3 of \cite{bierkens.roberts.zitt:2019}. 
\end{remark}

Using these assumptions, we see that an immediate corollary of Theorem \ref{geom.ergo.suzz:1} is the following.

\begin{corollary}[Exponential Ergodicity of original Zig-Zag]\label{original.zz.geom.ergo:1}
Let $(Z_t)_{t \geq 0}=(X_t,\Theta_t)_{t \geq 0}$ be a $d$-dimensional original Zig-Zag  process. Assume that $U \in C^3$, and has a non-degenerate local minimum. Assume further that Assumptions \ref{zz.assumption:101}, \ref{zz.assumption:102}, \ref{the.second.one.that.always.holds:1} and \ref{zz.assumption:103} hold. 
Then the original Zig-Zag process is exponentially ergodic, i.e. there exist $M:E \rightarrow [1,+\infty)$, and $\rho<1$ such that for any $(x,\theta) \in E$,
\begin{equation}\label{geom.ergo.zz:11}
\| \mathbb{P}_{x,\theta}(Z_t \in \cdot)-\mu(\cdot) \|_{TV} \leq M(x,\theta)\rho^t.
\end{equation}
\end{corollary}

\subsection{Space Transformation and Uniform Ergodicity}\label{relationship.suzz.zz:00}
When we focus on the one dimensional process, we can prove that it is a space transformation of an original, one-dimensional Zig-Zag process. We have the following.
\begin{proposition}[One Dimensional SUZZ as Space Transformation]\label{space.transformation:1982}
Consider a one-dimensional SUZZ process $Z_t=(X_t,\Theta_t)_{t \geq 0}$ with strictly positive speed function $s \in C^2$, targeting a measure $\mu$ as in (\ref{zz.inv:3}). Assume that the rates satisfy (\ref{rates.formula:2}) and %Assumptions \ref{non.explo.condition:0}, \ref{ass.lyapunov:1}, \ref{the.one.that.always.holds:1} and \ref{ass.lyapunov:2} hold. 
let 
\begin{equation}\label{space.transform:001}
f(x)=\int_0^x\frac{1}{s(u)}du 
\end{equation}
and 
\begin{equation}
    \pm M^{\pm}=\lim_{x \rightarrow \pm \infty}f(x) \in \mathbb{R} \cup \{ -\infty , + \infty \}.
\end{equation}
 Then, the process $(Y_t,\Theta_t)_{t \geq 0}$, where $Y_t=f(X_t)$, is a one-dimensional original Zig-Zag process, defined on $(-M^-,M^+) \times \{ -1,+1 \}$. If the SUZZ process is non-explosive, then $(Y_t,\Theta_t)_{t \geq 0}$ has invariant measure $\nu$ where
\begin{equation}\label{inv.of.transformed.process:1}
\nu(dy,d\theta)=\frac{1}{\tilde{H}}\exp\{ -\tilde{U}(y) \}dy d\theta
\end{equation}
and
\begin{equation}\label{potential.of.transformed.process:1}
\tilde{U}(y)=U(f^{-1}(y))-\log s(f^{-1}(y)).
\end{equation}
\end{proposition}

Using Proposition \ref{space.transformation:1982}, we can prove that the one-dimensional SUZZ process with explosive deterministic dynamics is %if we use explosive deterministic dynamics we can create algorithms that are 
uniformly ergodic. This means that it is exponentially ergodic and the mixing time can be bounded by a quantity that does not depend on the starting point. This is a consequence of the fact that explosive deterministic dynamics have $\infty$ as entrance boundary. Our current proof, presented in Appendix \ref{proof.of.space.uniform:1}, heavily relies on the fact that the one-dimensional process is a space transformation of an original Zig-Zag.

\begin{theorem}[Exponential and Uniform Ergodicity in One Dimension]\label{uniform.ergodicity:1}
Consider a one dimensional SUZZ process $Z_t=(X_t,\Theta_t)_{t \geq 0}$ with strictly positive speed function $s \in C^2$. Assume that the rates satisfy (\ref{rates.formula:2}) and Assumptions \ref{non.explo.condition:0} and \ref{ass.lyapunov:1} hold. Then the process is non-explosive, it has $\mu$ defined in (\ref{zz.inv:3}) as invariant and is exponentially ergodic.  %\ref{the.one.that.always.holds:1} and \ref{ass.lyapunov:2} hold. %The following two hold.
%\begin{enumerate}
    %\item 
    Assume further, that for some $x \in \mathbb{R}$ and for any $\theta= \pm 1$ the deterministic flow of the process $\{\Phi_t(x,\theta) , t \geq 0\}$ has a finite explosion time $t^*(x,\theta)$. Then the process is uniformly ergodic, i.e. there exists a $M>0$ and $\rho<1$ such that for any $(x,\theta) \in E$ and $t \geq 0$,
\begin{equation*}
\| \mathbb{P}_{x,\theta}(Z_t \in \cdot)-\mu(\cdot) \|_{TV}\leq M \rho^t.
\end{equation*}
\end{theorem}

We emphasise, however, that the SUZZ algorithm cannot necessarily be written as a space transformation of an original Zig-Zag in dimension higher than $1$. In Figure \ref{two.dimensional.suzz.is.not.space.transformation.of.zz:1} we illustrate the contradiction that may occur if such a space transformation were to exist. We consider a $d=2$ case, and assume (to reach a contradiction) that there does exist such a transformation $\phi $.
The left figure represents the movement of a two-dimensional SUZZ process starting from $x_1$ and ending at $x_4$. The right figure represents the movement of the $\phi$-space transformed process, assumed to be an original Zig-Zag, starting from $\phi(x_1)$ and ending at $\phi(x_4)$. There are two paths from $x_1$ to $x_4$, passing through and switching at $x_2$ or $x_3$ respectively. If the speed function $s(x)$ takes smaller values on the path via $x_2$, then the process arrives at $x_4$ faster via the $x_3$ path rather than via the $x_2$ path. The same thing must hold for the $\phi$-transformed process, i.e. the process arrives to $\phi(x_4)$ faster via $\phi(x_3)$ rather than $\phi(x_2)$. However, the transformed process moves with constant unit speed, as it is an original Zig-Zag process. Furthermore, the two paths from $\phi(x_1)$ to $\phi(x_4)$, passing either via $\phi(x_2)$ or via $\phi(x_3)$ have the same length. Therefore the time it takes for the transformed process to traverse either of the two paths from $\phi(x_1)$ to $\phi(x_4)$ is the same. This gives a contradiction and establishes that the SUZZ process cannot be spaced transformed to an original Zig-Zag process in dimension $d \geq 2$.

As a result of this discussion, we cannot rely on the existence of such a transformation between SUZZ and original Zig-Zag, so results for SUZZ cannot easily be obtained from those for original Zig-Zag by simple transformation arguments. This also means that we do not currently have a way to extend the uniform ergodicity result of Theorem \ref{uniform.ergodicity:1} to higher dimensions, since the proof heavily relies on the space transformation property of the one-dimensional SUZZ process. Simulation results, however, seem to suggest that the starting position does not heavily influence the algorithmic performance, and we suspect the uniform ergodicity holds for higher dimensions as well.

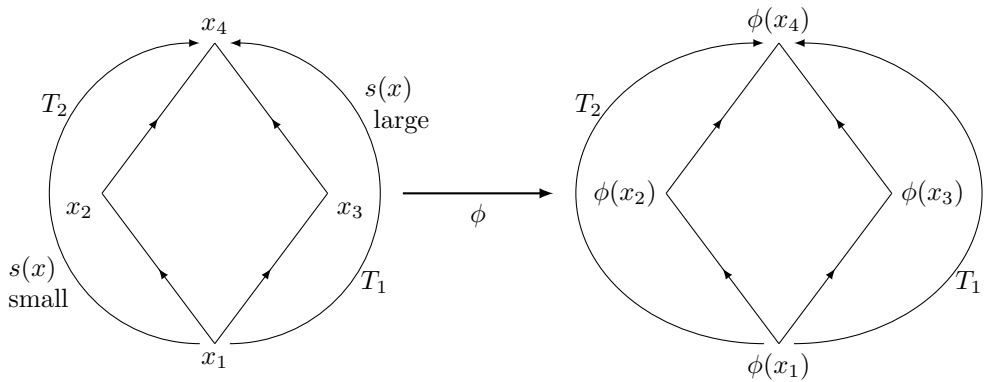
\begin{figure}
\centering
\begin{tikzpicture}
\draw[middlearrow={latex}] (0,-2) -- (1.5,0);
\draw[middlearrow={latex}] (0,-2) -- (-1.5,0);
\draw[middlearrow={latex}] (-1.5,0) -- (0,2);
\draw[middlearrow={latex}] (1.5,0) -- (0,2);
\draw node[below] at (0,-2) {$x_1$};
\draw node[below left] at (-1.5,0) {$x_2$};
\draw node[below right] at (1.5,0) {$x_3$};
\draw node[above] at (0,2) {$x_4$};
\draw[-latex] (-0.2,-2) arc  (270:90:2cm) node[pos=0.3,align=left,left] {$s(x)$\\ small} node[pos=0.7,align=left,left] {$T_2$} ;
\draw[-latex] (0.2,-2) arc  (-90:90:2cm) node[pos=0.7,align=left,right] {\hspace{0.05cm}$s(x)$\\\hspace{0.05cm} large} node[pos=0.3,align=left,right] {$T_1$};
\draw [-latex,thick] (2.5,0) -- node[midway,sloped,below] {$\phi$} (4.5,0);
\draw[middlearrow={latex}] (7.5,-2) -- (9,0);
\draw[middlearrow={latex}] (7.5,-2) -- (6,0);
\draw[middlearrow={latex}] (6,0) -- (7.5,2);
\draw[middlearrow={latex}] (9,0) -- (7.5,2);
\draw node[below] at (7.5,-2) {$\phi(x_1)$};
\draw node[left] at (6,0) {$\phi(x_2)$};
\draw node[right] at (9,0) {$\phi(x_3)$};
\draw node[above] at (7.5,2) {$\phi(x_4)$};
\draw[-latex] (7.3,-2) arc  (270:90:2.5cm and 2cm) node[pos=0.7,align=left,left] {$T_2$} ;
\draw[-latex] (7.7,-2) arc  (-90:90:2.5cm and 2cm) node[pos=0.3,align=left,right] {$T_1$};
\end{tikzpicture}
\caption{{\it Figure explaining why a two dimensional SUZZ is not in general a space transformation $\phi$ of an original Zig-Zag process. The figure on the left shows two possible SUZZ paths from $x_1$ to $x_4$, passing via either $x_2$ or $x_3$ and switching directions there. Assuming that the SUZZ process was a space transformation of an original Zig-Zag, via a function $\phi$, the figure on the right shows the two paths of the space transformed original Zig-Zag process, from $\phi(x_1)$ to $\phi(x_4)$, via either $\phi(x_2)$ or $\phi(x_3)$. The times $T_1$ and $T_2$ to traverse the two paths from $x_1$ to $x_4$ on the left figure, depending on the speed function, do not have to be the same. The same times, $T_1,T_2$, would also be the times to traverse the two paths from $\phi(x_1)$ to $\phi(x_4)$ on the $\phi$-transformed right figure.  However, if the $\phi$-transformed right figure was an original Zig-Zag, moving with constant unit speed, these two times would had to be the same.}
}\label{two.dimensional.suzz.is.not.space.transformation.of.zz:1}
\end{figure}

\section{Choice of Speed Function}\label{choice.of.speed.function:1}
A natural objective is to choose the speed function that generates an algorithm which is as efficient as possible. To get some intuition into how to achieve this, consider the one-dimensional 
SUZZ process $(X_t,\Theta_t)_{t \geq 0}$ with speed function $s$ and $f$ as given in (\ref{space.transform:001}).
From Proposition \ref{relationship.suzz.zz:00}, 
 $(Y_t,\Theta_t)=(f(X_t),\Theta_t)$ is an original Zig-Zag process, targeting a measure with negative log-density given by $\tilde{U}(y)=U(f^{-1}(y))-\log s(f^{-1}(y))$, defined on a subset of $\mathbb{R}$. 
 Therefore, instead of using SUZZ, one could equivalently use the original Zig-Zag $(Y_t,\Theta_t)$, target the potential $\tilde{U}$ and then use the path of $f^{-1}(Y_t)$ as a way to sample from the measure of interest. This is very similar in spirit to the work of \cite{johnson.geyer:12}. In summary, in the one-dimensional case, the goal of choosing the most efficient $s$, boils down to choosing an invertible space transformation $f$, and analysing an original Zig-Zag
algorithm on the transformed potential
\begin{equation*}
\tilde{U}(y)=U(f^{-1}(y))-\log s(f^{-1}(y))=U(f^{-1}(y)) + \log f'(f^{-1}(y)).
\end{equation*}

A natural candidate suggested by this is the choice $s(x)= \exp \{ U(x) \}$ leading to the space transformation
\begin{equation*}
f(x)=\int_0^x \exp \{ -U(x) \} dx=F(x)-\pi((-\infty , 0]),
\end{equation*}
where $F$ is the CDF of $\pi$. Using this SUZZ is equivalent to run an original Zig-Zag on the measure with negative log-density $\tilde{U} \equiv 0$ (i.e. the Lebesgue measure) and then transform the values back according to the function $f^{-1}(y)$. There are similarities here with inverse CDF sampling. However, this choice of $s$ is precluded by  Assumption \ref{non.explo.condition:0} as it leads to an explosive SUZZ, corresponding to the transformed Zig-Zag process eventually hitting the boundary of the transformed space (either $f(+\infty)$ or $f(-\infty)$).

This discussion suggests that we might obtain an efficient method by picking $s$ such that $s(x)\exp\{ -U(x) \}$ decays to zero as $|x| \rightarrow \infty$ slowly.
While the equivalence of the SUZZ to an original Zig-Zag with appropriate transformation is only valid in the one-dimensional case, this strategy for choosing $s$ can be applied quite generally in multi-dimensional settings.

\subsection{A Computational Efficiency Criterion in the One-Dimensional case}\label{efficiency:00}

Computational efficiency of the algorithm goes far beyond qualitative convergence results such as
exponential ergodicity. The actual cost of implementing MCMC algorithms 
  is controlled by the number of computational operations that need to be performed
to obtain a desirable amount of samples from the target distribution.
In our setting the computational cost comes largely from evaluating the gradient of the log-likelihood of the target, which is needed in order to sample the direction switches. Therefore, in order to understand the algorithmic efficiency, we must study the number of the gradient log-likelihood evaluations needed to be performed until we get enough samples from the target. This section will try to answer this question for the one-dimensional SUZZ process. 

In an ideal setting, using Poisson thinning in a perfect way (see \cite{lewis.shedler:79}), and for any choice of speed function, the number of gradient log-likelihood evaluations (and therefore the computational cost) would be equal to the number of switches of direction. In practice, the actual number of gradient log-likelihood evaluations depends on the tightness of the bounds used in this Poisson thinning operation and is therefore difficult to use as
a consistent metric. Therefore we shall instead use
the number of direction switches as a unit for measuring the implementation cost of the algorithm. Our goal now is to define a quantity that depends on the speed function and provides a way to measure the performance of the algorithm per implementation cost.

For the remainder of this section, we focus on dimension one and closely follow \cite{bierkens.duncan:17}.

Let $N_T$ be the expected number of switches until time $T$, i.e. the average implementation cost of the algorithm. Then $N_T=\mathbb{E}[\int_0^T\lambda(X_s,\Theta_s)ds]$. Since the process is Harris recurrent, we have a Law of Large numbers and 
\begin{eqnarray}\label{N_0:1}
&   \hspace{2cm}  N_0:=\lim_{T \rightarrow \infty}\frac{N_T}{T} = \int \lambda(x,\theta) d\mu(x,\theta) = \\ \nonumber &\frac{1}{2H}\sum_{\theta= \pm 1} \int_{\mathbb{R}}\exp \{ -U(x) \}\lambda(x,\theta)dx 
 = \frac{1}{2H}\int_{\mathbb{R}}\exp \{ -U(x) \}|s(x)U'(x)-s'(x)|dx. \\ \nonumber
\end{eqnarray}
Consider a functional of interest $g \in  L^2(\mu)$, whose integral under $\mu$ we are trying to approximate. Let $(Z_t)_{t \geq 0}$ be a SUZZ process targeting $\mu$. Assume without loss of generality that $\mu(g)=0$ and consider the estimator 
\begin{equation}\label{g.t:1}
g_T=\frac{1}{T}\int_0^T g(Z_s)ds. 
\end{equation}
If the process satisfies a CLT then there exists an asymptotic variance $\gamma^2_g \in [0,+\infty)$ such that
\begin{equation}\label{ass.var.definition:1}
\lim_{T \rightarrow +\infty}T \cdot Var(g_T)=\gamma^2_g. 
\end{equation}
A way to measure the efficiency of the algorithm is the Effective Sample Size (ESS) (see \cite{dyk.park:11}) which approximates the number of independent samples the algorithm has generated from the target until time $T$. It is defined as 
\begin{equation}\label{eff.sample.size.definition:1}
ESS(T)=\frac{Var_{\mu}(g)}{Var(g_T)}.
\end{equation}
Since the cost of implementing the algorithm is the average number of switches, it seems natural to consider the quantity of ESS per average number of direction switches in order to evaluate the efficiency of the algorithm. 
Combining (\ref{N_0:1}), (\ref{ass.var.definition:1}) and (\ref{eff.sample.size.definition:1}) we get
\begin{equation}\label{quantity.to.minimize:1}
\lim_{T \rightarrow \infty}\frac{ESS(T)}{N_T} =  \frac{Var_{\mu}(g)}{\gamma_g^2 N_0}.
\end{equation}
Therefore, in order to choose the optimal $s$ that makes the algorithm the most efficient we need to minimize the quantity $\gamma_g^2 N_0$ over different speed functions. $N_0$ is written in terms of $s$ in (\ref{N_0:1}). We will now present a proposition that describes the asymptotic variance $\gamma_g^2$ in terms of $s$. Before that, we need to make an assumption. Let
\begin{equation}\label{V.function:1821}
V(x,\theta)=\exp \{ aU(x)- a\log s(x) + \frac{1}{2} \delta |A_i(x)| h_n(| A_i(x)|) \}
\end{equation}
where $h_n$ as in (\ref{iterative.logarithms:1}), and $a , \delta >0$ small enough so that if $\mathcal{L}$ is the operator defined for all $f \in C^1(E)$ as
\begin{equation}\label{d1.gener:1}
\mathcal{L}f(x,\theta)=\theta s(x) f'(x,\theta)+ \left( [\theta U'(x)]^++\gamma(x) \right)  \left(f(x,-\theta)-f(x,\theta)\right) ,
\end{equation}
then there exist $c,b>0$ and a compact set $C$, such that for all $(x,\theta) \in E$,
\begin{equation}\label{lyapunov:1822}
 \mathcal{L}V(x,\theta) \leq -c V(x,\theta) +b 1_{(x,\theta)\in C}.
\end{equation}
The fact that $(\ref{lyapunov:1822})$ holds for $a, \delta $ small enough will be later proved in Appendix \ref{app.non.explo:00}, in the proof of Theorem \ref{non.explo:1}. We now assume the following.
\begin{assumption}\label{the.not.so.good.assumption:1}
Let $\mathcal{L}$ be the operator in (\ref{d1.gener:1}).
Let $V$ defined as in (\ref{V.function:1821}) for $a,\delta >0$ small enough such that  (\ref{lyapunov:1822}) holds. Assume that there exists a $C>0$ such that for all $g \in L^1(E)$ satisfying $|g(x,\theta)| \leq V(x,\theta)$ for all $(x,\theta) \in E$, there exists a $\phi$ such that
\begin{equation*}
-\mathcal{L}\phi = g
\end{equation*}
and such that for all $(x,\theta) \in E$
\begin{equation*}
|\phi(x,\theta)| \leq C \ V(x,\theta).
\end{equation*}
\end{assumption}
\noindent This assumption is a result proven in \cite{glynn.meyn:96} in the case where $\mathcal{L}$ is the extended generator of a process, in the sense that for any $f \in C^1(E)$ the process
\begin{equation*}
M_t=f(X_t)-f(X_0)-\int_0^t\mathcal{L}f(X_s) \ ds
\end{equation*}
is a martingale. However, since we allow the process to have explosive deterministic dynamics, we can only guarantee that $M_t$ is a local martingale. We note here that in \cite{glynn.meyn:96} the authors claim that Assumption \ref{the.not.so.good.assumption:1} holds in our case as well, i.e. when $\mathcal{L}$ only induces a local martingale. However, to the best of our knowledge this is not something proven in the literature. Therefore, we make this assumption here and we present the following result under Assumption \ref{the.not.so.good.assumption:1}. This result describes the asymptotic variance $\gamma_g^2$ in terms of the speed function $s$.

\begin{proposition}\label{efficiency.proposition:1}
Assume that the rates satisfy (\ref{rates.formula:2}) and Assumptions \ref{non.explo.condition:0}, \ref{ass.lyapunov:1}, \ref{the.one.that.always.holds:1} and \ref{ass.lyapunov:2} hold. Let $g:E \rightarrow \mathbb{R}$ in the domain of $\mathcal{L}$, with  $\mu(g)=0$ and assume that $|g(x,\theta)| \leq V(x,\theta)$ for all $(x,\theta) \in E$, where $V$ is the function defined in (\ref{V.function:1821}) 
 for some $a<1 , \delta >0$ small enough such that (\ref{lyapunov:1822}) holds. Finally, assume that Assumption \ref{the.not.so.good.assumption:1} holds. Then, if $Z_t=(X_t,\Theta_t)$ is the one dimensional SUZZ process with speed function $s$, starting from the invariant measure $\mu$, we have 
\begin{equation*}
\frac{1}{\sqrt{T}}\int_0^T g(Z_s)ds \xrightarrow{T \rightarrow \infty}N(0,\gamma_g^2)
\end{equation*}  
in distribution where 
\begin{equation}\label{efficiency.as.var:1}
\gamma_g^2= \frac{1}{2H} \int_{\mathbb{R}} \left | s(x)U'(x)-s'(x) \right | \frac{1}{s^2(x)\exp\{ -U(x) \}} k^2(x) dx
\end{equation}
and
\begin{equation}\label{efficiency.as.var:1.5}
k(x)= \int_x^{+\infty}(g(y,+1)+g(y,-1))\exp\{ -U(y) \}dy .
\end{equation}
\end{proposition}

Proposition \ref{efficiency.proposition:1} and equations (\ref{N_0:1}) and (\ref{quantity.to.minimize:1}) indicate that for a given function $g: E \rightarrow \mathbb{R}$ with $\mu(g)=0$, such that for all $a \in (0,1)$ we have $|g(x,\theta)| \leq \exp\{ aU(x)-a\log s(x) \}$ for all $(x,\theta) \in E$, satisfying the assumptions of Proposition \ref{efficiency.proposition:1}, we need to pick a speed function $s$ in order to minimize the quantity 
\begin{equation}\label{minimization.problem:1}
J[r]:=\gamma_g^2N_0=\int_{\mathbb{R}}|r'(x)|dx \int_{\mathbb{R}}\dfrac{|r'(x)|}{r^2(x)} k^2(x)dx,
\end{equation}
where 
\begin{equation*}
r(x)=s(x)\exp \{ -U(x) \},
\end{equation*}
and we need to impose the condition
\begin{equation*}
\lim_{|x|\rightarrow \infty}r(x)=0,
\end{equation*}
so that Assumption \ref{non.explo.condition:0} holds.%and $k(x)= \int_x^{+\infty}g(u,+1)+g(u,-1)\mu(du)$.

We will call the functional $J$ the {\bf inverse algorithmic efficiency}. Note that $J$ is invariant under constant scaling of function $s$. This is in accordance to the fact that we do not gain or lose any efficiency by speeding up Zig-Zag with a constant speed, for example by having velocities of the form $\{ \pm 2 \}$. 
\begin{remark}
The result of Proposition \ref{efficiency.proposition:1} can be generalised in the case where $\mu(g)$ is not necessarily zero. In the general case, the function $k$ in (\ref{efficiency.as.var:1.5}) used to define $J[r]$ would be \begin{equation*}\label{efficiency.as.var:1.6}
k(x)= \int_x^{+\infty}\left( g(y,+1)+g(y,-1)-\frac{1}{2}\mu(g) \right) \exp\{ -U(y) \}dy .
\end{equation*}
In practice, $\mu(g)$ is not a known quantity. Then one can use the asymptotically unbiased estimator $g_T$ in (\ref{g.t:1}) instead of $\mu(g)$ to calculate an approximation of the inverse efficiency $J$. 
 \end{remark}
\noindent Ideally, we would like to pick a speed function such that $r$ minimises (\ref{minimization.problem:1}). Note, however, that minimising (\ref{minimization.problem:1}) is not a well-posed problem. Indeed, let $r_0$ be a function such that $J[r_0]<\infty$ and $\lim_{|x|\rightarrow \infty}r_0(x)=0$. For any $n \in \mathbb{N}$ let 

\begin{equation}\label{minimizing.sequence:1}
\left\{\begin{array}{l}
r_n(x)=1 , \ |x| \leq n  \vspace{1.5mm}\\ 
r_n(x)=r_0(x-n), \ x > n\vspace{1.5mm}\\
r_n(x)=r_0(x+n), \ x<-n
\end{array}\right..
\end{equation} 
Then $J[r_n]\xrightarrow{n \rightarrow \infty}0$. At the same time, the only functions that satisfy $J[r]=0$ are the constant ones and since we impose the condition that $\lim_{|x|\rightarrow \infty}r(x)=0$, the only function $r$ that satisfies $J[r]=0$ is the function $r \equiv 0$.%this is not a possible choice for a strictly positive $r$. 

Note, however, that the $n$th term of the minimising sequence $r_n$ is equal to $1$ on $[-n,n]$ and this means that $s(x)=\exp \{ U(x) \}$ for $x \in [-n,n]$. Heuristically, and as discussed in the beginning of Section \ref{choice.of.speed.function:1}, one could expect good performance in the ideal case where $s(x)$ could be set equal to $\exp\{ U(x) \}$ for $x \in [-n,n]$ for some large $n$.

In Table \ref{efficiencies.table:1} we present some examples, comparing the efficiency of different algorithms. As target distribution we consider a Normal with mean zero and variance one, an exponential with parameter one, symmetrically extended to the negative reals, a Student distribution with $3$ degrees of freedom and a distribution of the form 
\begin{equation*}
    \pi(x)=\frac{1}{H}\exp\{ -\left( 1+x^2 \right)^{1/4} \},
\end{equation*}
which we will call sub-exponential, since it has tails heavier than any exponential, but it does not decay polynomially fast. For each of these densities, except for the Student$(3)$, we are estimating the expectation of the distribution, i.e. we set $g(x,\theta)=x$. For the Student($3$) we are estimating the expectation the function $g(x)=\sgn(x)\log(1+|x|)$. This is to guarantee that the function $g$ verifies the growth assumptions of Proposition \ref{efficiency.proposition:1}. Regarding the speed function, we use the original Zig-Zag (i.e. $s(x)=1$), and we also use the speed functions \begin{equation}\label{one.dimensional.speed:1}
    s(x)=\left( 1+x^2\right)^{(1+k)/2}
    \end{equation}
    for $k=0,1,2,3$. These algorithms will be denoted by SUZZ($k$), where $k$  is the parameter in the exponent of the speed function. Choosing $k=0$ induces non-explosive deterministic dynamics, whereas choosing $k>0$ induces explosive ones. We will verify the assumptions used in Theorem \ref{uniform.ergodicity:1} for these speed functions is Appendix \ref{verification:000}. In Table \ref{efficiencies.table:1} we compare the inverse efficiencies of all the algorithms for all four targets. In order to numerically estimate the integrals arising in the definition of $J[r]$ we use the $\mathtt{integrate}$ function and the $\mathtt{polyroot}$ library of $\mathtt{R}$. We should emphasize that since we do not take into account some normalisation constants and since in the case of Student($3$) distribution we are estimating a different observable, the comparison in Table \ref{efficiencies.table:1} should only be made column-wise (i.e. for a given distribution compare different algorithms). 

For any target distribution, the algorithms SUZZ($0$) and SUZZ($1$) provide better results than the original Zig-Zag algorithm. Furthermore, for all targets except for the sub-exponential, the SUZZ($2$) algorithm performs better than the original Zig-Zag. SUZZ($3$) does not seem to perform that well and it only has better efficiency than the original Zig-Zag on the exponential target. Note that we do not present an efficiency value for SUZZ($3$) on the Student$(3)$ target since this algorithm does not satisfy Assumption \ref{non.explo.condition:0} and will in fact explode in finite time a.s. It is also worth noting that for all the targets, with the exception of the sub-exponential one, the inverse efficiency function of the algorithms seems to be "quadratic" with respect to $k$ and seems to be minimised when $k=1$. 

Finally, we should note that the notion of inverse algorithmic efficiency  is so far restricted to one-dimensional setting. Generalising this to higher dimensions would involve solving the Poisson equation of the multi-dimensional SUZZ process and is subject to further research.

\begin{table}
\begin{tabular}{l c c c c}
\multicolumn{5}{c}{Algorithmic Inverse Efficiency in One Dimension}\\
Algorithms  & Normal        & Exponential  & Sub-exponential$(0.5)$      & Student$(3)$      \\ \hline
Original Zig-Zag    &  $16$            & $80$            &  $57044$      & $34.2457$              \\
SUZZ(0)    & $4.9817$       & $26.3397$        & ${\bf 3536}$         & $7.9736$               \\ 
SUZZ(1)  &   ${\bf 4.4259}$       & ${\bf 7.1017}$        & $45948$    & ${\bf 2.4708}$             \\
SUZZ(2) & $14.9568$ & $19.0364$ & $1315827$  & $11.7397$ \\
SUZZ(3) & $45.6342$  & $30.0839$ & $24623975012$ & -
\end{tabular}
\caption{{\it J values, as introduced in (\ref{minimization.problem:1}), for various SUZZ algorithms targeting various distributions; SUZZ($k$) denotes the SUZZ algorithm with speed function of the form $s(x)=\left( 1+x^2\right)^{(1+k)/2}$; Smallest value for every column in bold.}}\label{efficiencies.table:1}
\end{table}

\section{Numerical Simulations}\label{simulations:00}
In this section we will present some computational results that aim to highlight the behaviour of SUZZ and compare it with original Zig-Zag. We will present results for one-dimensional and twenty-dimensional targets. As already suggested in Section \ref{efficiency:00}, the one dimensional SUZZ can vastly outperform the original Zig-Zag. However, it will be seen that there are significant advantages in using a speed function in higher dimensions as well.

First, we present numerical results on a one-dimensional Student target with three degrees of freedom, denoted by Student($3$), i.e. a target with density given by $\pi(x)=\frac{1}{H}\left(1+\frac{1}{3}x^2\right)^{-2}$. We used the Zig-Zag algorithm (ZZ) along with SUZZ(0) and SUZZ(1) algorithms, where SUZZ($k$) indicates the SUZZ algorithm with speed function given by (\ref{one.dimensional.speed:1}). We emphasise here that even though the deterministic dynamics of SUZZ(1) explode in finite time, the process will a.s.\ not explode due to Theorem 3.6. Finally, we also used a Random Walk Metropolis algorithm on a transformed state space, introduced in \cite{johnson.geyer:12} as a method that is geometrically ergodic even on heavy tailed targets. The proposal distribution is a one dimensional Normal(0,1) and the parameters of the space transformation are tuned using the guidance of the discussion in \cite{johnson.geyer:12}. We will be referring to this algorithm as Transformed Random Walk Metropolis (TRWM). For each of the four algorithms presented, we simulated 25 independent realisations of each process, until $N=10^4$ switches of direction occurred for the ZZ or SUZZ algorithms. In the case of TRWM we simulated for $N=10^4$ steps. To construct a sample from the ZZ and the SUZZ algorithms, we used the position of the process every $\delta$ time units ($\delta$-skeletons). Here $\delta$ is different for each algorithm and it is chosen in the following way. For each algorithm, we first run an initial run, which created a path of time length $S(N)$. Then we fixed $\delta=\frac{S(N)}{N}$ so that for this run, the size of the skeleton was equal to $N=10^4$. We used this fixed $\delta$ for all other runs of the algorithm, expecting each future skeleton to have a size roughly equal to $N=10^4$, which was indeed the case. This was done in order to guarantee fairness between the performance evaluation across all algorithms. More precisely, since we use the number of switches ($N=10^4$) as a unit to measure computational cost, it would make sense for all the algorithms that run for the same number of switches to produce roughly the same number of samples. For TRWM, since we run the algorithm for $N=10^4$ steps, the sample generated had a size of $N=10^4$. To analyse the performance of the algorithms we have used the Effective Sample Size (ESS) (see \cite{cowles.carlin:96}), computed using $\mathtt{coda}$ from $\mathtt{R}$. The ESSs were calculated after we transformed the sample via the function 
\begin{equation}\label{simulations.transformation:1}
    f(x)=\sgn(x)\log(1+|x|),
\end{equation} 
so that we can guarantee that the variance of the ESS is finite. All simulations were performed using {\tt MATLAB} in a computer with i7-8550U CPU and 1.80 GHz.

We present our results in Table \ref{one.dimensional.ess.table:1}. We present average and median ESS across 25 realizations (standard deviation in parenthesis). We also report the median ESS per likelihood evaluation and per minute of implementation time. The best performance is highlighted in bold letters. It is clear that both SUZZ algorithms outperform both the original Zig-Zag and the TRWM, in all criteria based on ESS, (ESS per switches, per likelihood evaluations and per implementation time). It is also interesting that the algorithm with the explosive deterministic dynamics seems to perform the best. This is consistent with Table \ref{efficiencies.table:1} where the inverse algorithmic efficiency of SUZZ(1) is the smallest of all algorithms targeting the Student(3).

\begin{table}
\footnotesize
\begin{tabular}{l l l l l l l l l}
\multicolumn{7}{c}{One Dimensional Student($3$), Number of Switches $N=10^4$}\\ \hline
Algorithms  & ESS(SD)   & Median ESS   & ESS/Lik.Eval.  &  ESS/min          \\
ZZ 	        & 5272.9 (1274.0)      & 5675.6            & $1.5 \cdot 10^{-4}$	      &	15765.6            \\
SUZZ(0)	    & 20755.8 (718.1)      & 20779.2      & $3.0 \cdot 10^{-2}$ 	      & 31483.6     \\
SUZZ(1)	        & {\bf 46346.2} (3154.6)      & {\bf 46397.8}	    & ${\bf 3.4 \cdot 10^{-2}}$   & {\bf 154659.3}    \\
TRWM 	        & 29.8 (14.0)      & 22.8	    & $0.2 \cdot 10^{-2}$ & 3257.1   \\
\end{tabular}
\caption{{\it  SUZZ, ZZ and TRWM algorithms targeting a one-dimensional Student($3$) distribution. For SUZZ($k$), we use the speed function as in (\ref{one.dimensional.speed:1}). The algorithms ran until $N=10^4$ switches (or steps for the TRWM) occurred and the average ESS (with Standard deviation in a parenthesis) along with the median ESS are presented. The median ESS per average likelihood evaluations and per average minutes of implementation time is also presented. All ESSs are calculated after we transform the sample via the function $f$ as in (\ref{simulations.transformation:1}). The best performance is highlighted with bold letters.}}\label{one.dimensional.ess.table:1}
\end{table}

As a second example, we used two SUZZ algorithms and ZZ to target a one-dimensional Cauchy distribution (i.e. $\pi(x)=\frac{1}{H}\left( 1+x^2 \right)^{-1}$). For the SUZZ algorithms we used speed functions of the form 
\begin{equation}\label{weird.d1.speed:1}
    s(x)=\max\{ 1, |x|^{1+k} \}
\end{equation}
for $k=0$ and $k=0.5$, denoted by SUZZ(0) and SUZZ(0.5). In Figure \ref{fig:qqplot.t1.d1} we present the Q-Q plots for these three algorithms against the Cauchy target. All the algorithms run for $N=10^4$ number of switches. It is clear that the SUZZ algorithms far outperform the original Zig-Zag process and the SUZZ algorithm with explosive deterministic dynamics ($k=0.5$) seems to have the optimal performance.

\begin{figure}
\centering
\subfloat[ZZ ]{\includegraphics[scale=1.5, width=4.5cm , trim= 2in 3in 2in 3.5in]{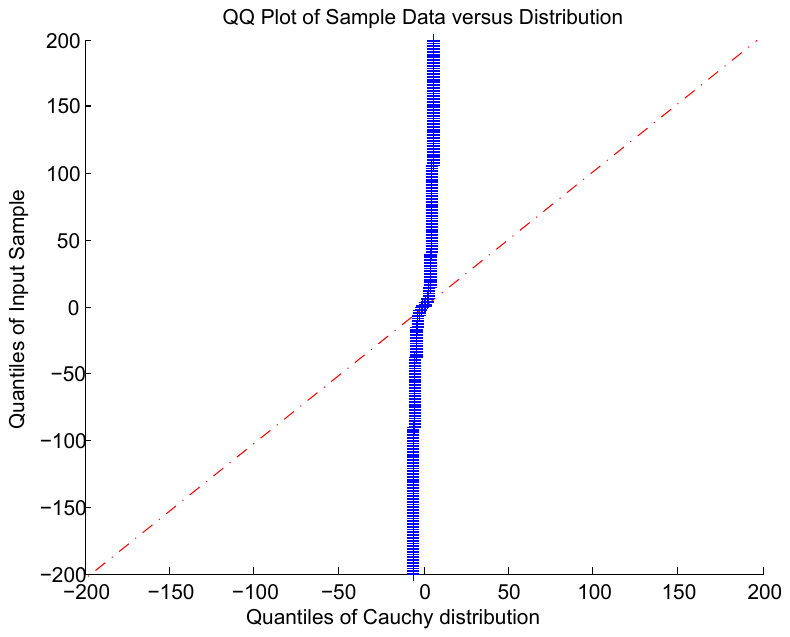}}\label{fig:qqplot.zz.cauchy.t1}
\hfill
\subfloat[SUZZ(0)]{\includegraphics[scale=1.5, width=4.5cm , trim= 2in 3in 2in 3.5in]{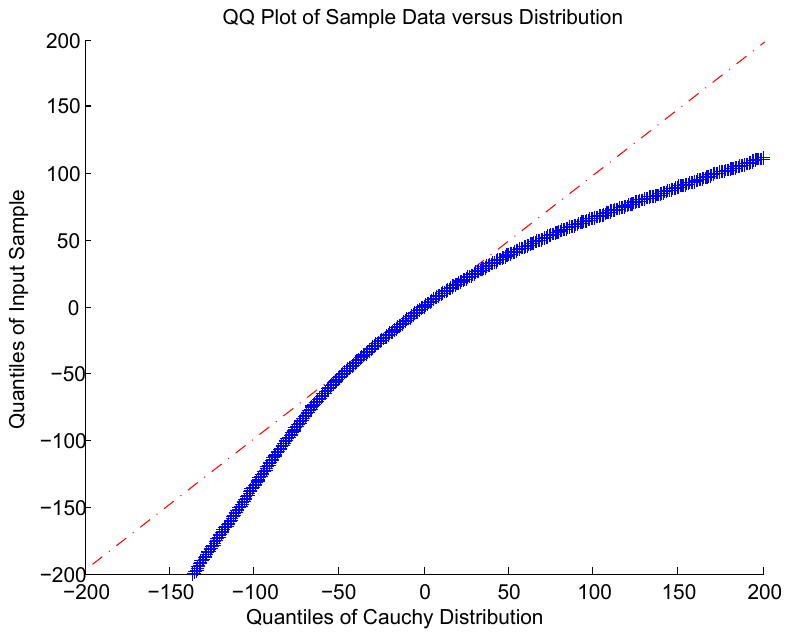}}\label{fig:qqplot.suzz0.t1.d1}
\hfill
\subfloat[SUZZ(0.5)]{\includegraphics[scale=1.5, width=4.5cm , trim= 2in 3in 2in 3.5in]{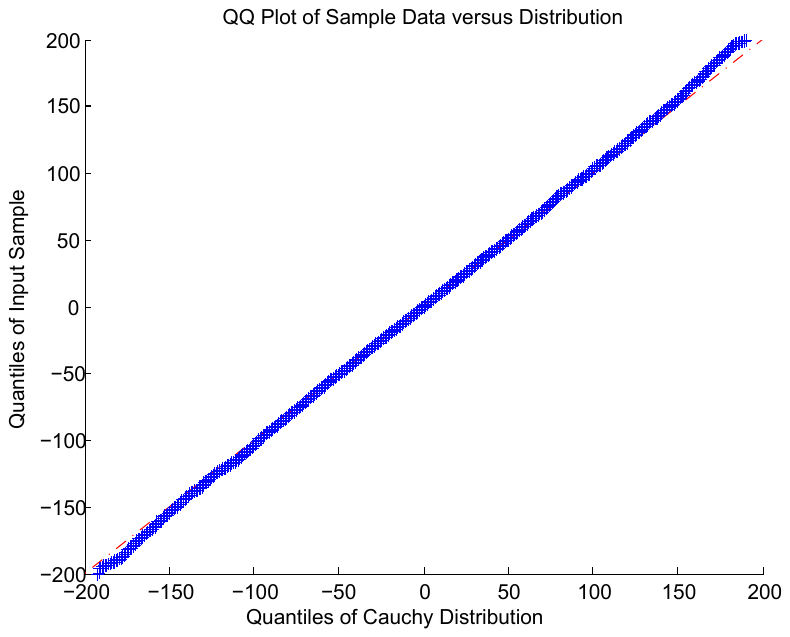}}\label{fig:qqplot.suzz05.t1.d1}
\caption{{\it Q-Q plots of various one-dimensional Speed Up algorithms targeting a Cauchy distribution. The algorithms have ran until $N=10^4$ switches of direction have occurred. The sample is created using the $\delta$-skeleton of the process for $\delta=0.1$. As $SUZZ(k)$ we denote the SUZZ algorithm with speed function of the form (\ref{weird.d1.speed:1}) and $k$ the parameter appearing in the equation.}}\label{fig:qqplot.t1.d1}
\end{figure}

Next, we present results on two twenty-dimensional targets. The first target is a twenty dimensional distribution with density of the form 
\begin{equation}\label{subexponential.20:1}
    \pi(x)=\frac{1}{H}\exp \left\{ -\left(1+\|x  \|_2^2\right)^{1/4} \right\},
\end{equation}
which we will call Sub-exponential($0.5$), since the tails decay like $\exp\{ -\| x \|_2^{0.5} \}$, slower than any exponential target but faster than any polynomial.  

The second target is a twenty-dimensional Student distribution, with $3$ degrees of freedom (denoted by Student($3$)), with scale matrix given by $B$ where

\begin{equation}\label{scale.matrix.student.3:1}
    B(i,j)=5, i \neq j , B(i,i)=30, i=1,2,3 , B(i,i)=20, i=4,5, \text{ and } B(i,i)=10, i=6,...,20.
\end{equation}
This means that 
\begin{equation}\label{20.d.student.3:1}
\pi(x)=\frac{1}{H}\left(1+\frac{1}{3}x^TB^{-1}x \right)^{-23/2}.
\end{equation}

The first target is in the setting of Proposition \ref{special.case.theorem:1}, therefore we know that SUZZ will be exponentially ergodic. Even though the second target is not in the setting of that proposition and no theoretical guarantees are established for the rate of convergence, we will see that the SUZZ process leads to numerical gains for both targets, compared to the original Zig-Zag and to TRWM algorithm, introduced in the one-dimensional simulations. We conjecture that some of the assumptions made in this document might not be necessary and the class of targets on which SUZZ can work well could be larger.

For both distributions we used four different algorithms to target them and we compare their performances. First of all, we used an original Zig-Zag process (ZZ). We also used two SUZZ processes, SUZZ($0$) and SUZZ($1$), where SUZZ($k$) denotes the SUZZ process with speed function given by
(\ref{speed.d20:1}). 
Note that SUZZ($0$) has non-explosive deterministic dynamics, while SUZZ($1$) has explosive ones. We present a general way to construct the deterministic dynamics for this type of speed functions in Appendix \ref{ode.solution:00}. Finally, we also used the Transformed Random Walk Metropolis (TRWM) algorithm, described in the one-dimensional simulations. The proposal distribution we used was a 20-dimensional Normal with identity covariance matrix and the parameters of the space transformation were tuned using the guidance of the discussion in \cite{johnson.geyer:12}. For each of the four algorithms presented, we simulated 25 independent realisations of each process, until $N=10^6$ switches of direction occurred for the ZZ or SUZZ algorithms. In the case of TRWM we simulated for $N=10^6$ steps. Having simulated a continuous time path, in order to construct a sample from the ZZ and the SUZZ algorithms we used the same procedure as in the one-dimensional simulations. We used the $\delta$-skeleton of the process, where $\delta$ was chosen after an initial run of the algorithm such that the sample size was roughly equal to the number of direction switches ($N=10^6$). As mentioned in the one-dimensional simulations, setting $\delta$ this way guarantees fairness between the performance evaluation across all algorithms. 
For TRWM, since we run the algorithm for $N=10^6$ steps, the sample generated had a size of $N=10^6$. All simulations were performed using {\tt MATLAB} in a computer with i7-8550U CPU and 1.80 GHz.

We present our results in Tables \ref{numerics.Sub05.log:1}, \ref{numerics.Sub05.mean:1} and \ref{numerics.20d.Student.3.pos.corr.fixed.N.SUZZ:1}. In Tables \ref{numerics.Sub05.log:1} and \ref{numerics.Sub05.mean:1} we report results on the Sub-exponential($0.5$) target, while in Table \ref{numerics.20d.Student.3.pos.corr.fixed.N.SUZZ:1} we report results concerning the Student($3$) target. We present average and median ESS across 25 realizations (standard deviation in parenthesis) and empirical probabilities of squares centered around $0$, containing $0.9$, $0.99$ and $0.999$ of the mass of the target and denoted Sq. 0.9, Sq 0.99 and Sq. 0.999 respectively. These squares were estimated using $\mathtt{mvtnorm}$ and $\mathtt{adaptMCMC}$ of $\mathtt{R}$. We also report the median ESS per likelihood evaluation and per minute of implementation time. For all algorithms, we consider the ESS of the first coordinate of the process, computed using the routine {\tt coda} of {\tt R}, but we note here that we recovered similar results when using the routine {\tt mcmcse} to estimate the multivariate ESS of the twenty dimensional algorithms. For the Student($3$) distribution (Table \ref{numerics.20d.Student.3.pos.corr.fixed.N.SUZZ:1}), the ESSs were calculated after we transformed the sample via the function (\ref{simulations.transformation:1})
so that we can guarantee that the variance of the ESS is finite, and the computation of ESS consistent across all 25 realisations of the chains. We did the same for the Sub-exponential($0.5$) distribution (Table \ref{numerics.Sub05.log:1}), but for that target we also present the ESS without any transformation of the sample (Table \ref{numerics.Sub05.mean:1}), since the variance of the ESS is finite when estimating the expectation of this target.  

All four algorithms provided a decent estimation of the probabilities of the squares, which can increase our trust that all algorithms converged to the right distribution. In terms of ESS, we observe that all SUZZ algorithms vastly outperformed the ZZ algorithm in terms of every criterion we used, i.e. ESS per number of switches, per number of likelihood evaluations and per implementation minutes. This shows that using a speed function in the context of PDMP algorithms can lead to significant benefits. Furthermore, the SUZZ algorithms can compare favourably to a state of the art algorithm like the TRWM in all three criteria (ESS per switches, ESS per likelihood evaluations and ESS per implementation minutes). For example, the SUZZ($1$) algorithm has twice better ESS per implementation time than the TRWM on the Student target. Notably, if the criterion is ESS per number of switches, which gives a theoretical upper bound on the ESS per likelihood evaluations for the SUZZ algorithm, the SUZZ algorithms perform at least 20 times better than the TRWM. 

We also note that although in our simulations the TRWM takes a lot less time to be implemented, there does not seem to be enough space to further reduce the implementation time of the TRWM code. On the other hand, the code of SUZZ is quite more complicated and a more qualified programmer could probably reduce the implementation time even further. More specifically, most of the simulation time was spent in finding the maximum of the rate function over a specific time horizon in order to perform Poisson thinning. If one could reduce the time spent in this type of maximisation sub-routines, one could significantly reduce the implementation time of SUZZ. Furthermore, while the likelihood evaluations of TRWM are always equal to the number of steps of the algorithm, one could try to further reduce the number of likelihood evaluations of the SUZZ algorithm if one has access to extra information on the structure of the target. One could also use ideas from \cite{sutton.fearnhead:21}, for example by adapting the time horizon over which the optimisation of the rate takes place, taking into account the previous switching times. This can be done without losing any theoretical guarantees since any choice of time horizon leads to stochastically identical algorithms. Furthermore, one can use ideas from \cite{corbella.spencer.roberts:22} to further optimise the Poisson thinning procedure and the implementation time, for example with the use of automatic differentiation schemes.

\begin{table}
\footnotesize
\begin{tabular}{l l l l l l l l l}
\multicolumn{7}{c}{20 Dimensional Sub-Exponential($0.5$), Number of Switches $N=10^6$, With Space Transformation.}\\ \hline
Algorithms  & ESS(SD)   & Median ESS   & Sq. 0.9 & Sq. 0.99 & Sq. 0.999 & ESS/Lik.Eval.  &  ESS/min          \\
ZZ 	        & 103661.4 (6347.7)  & 104665.5     &	0.9008	          & 0.9905	&  0.9991 &  $0.3 \cdot 10^{-3}$     & 124.9   \\
SUZZ(0)	    & {\bf 142663.2} (1511.3)      & {\bf 142382.6} &	{\bf 0.8998}            & {\bf 0.9899}  & {\bf 0.9990} 	      & $3.9 \cdot 10^{-3}$ & 2847.7              \\
SUZZ(1)	        & 134561.8 (2453.4)      & 134140.1 	    &   	0.8982         &	0.9897				 & {\bf 0.9990}	 & ${\bf 6.3 \cdot 10^{-3}}$   & {\bf 4471.3}    \\
TRWM 	        & 2767.0 (68.5)      & 2753.5    &	0.8994 	          & 0.9902 			 & 0.9992 	    & $2.8 \cdot 10^{-3}$ & 1966.8    \\
\end{tabular}
\caption{\vspace{1 mm}
{\it SUZZ, ZZ and TRWM algorithms targeting a twenty-dimensional Sub-exponential($0.5$) distribution, with density given by (\ref{subexponential.20:1}). For SUZZ($k$), we use the speed function as in (\ref{speed.d20:1}). The algorithms ran until $N=10^6$ switches (or steps for the TRWM) occurred and the average ESS (with standard deviation in a parenthesis) along with the median ESS is presented. The median ESS per average likelihood evaluations and per average minutes of implementation time is also presented. All ESS concern the first coordinate of the twenty-dimensional process and are calculated after we transform the sample via the function $f$ as in (\ref{simulations.transformation:1}). An estimation of probabilities assigned to various squares of $\mathbb{R}^{20}$ by the target distribution is also presented. The square denoted by "Sq $a$" means that target assigns probability $a$ inside the square. The best performance is highlighted with bold letters.}} \label{numerics.Sub05.log:1}
\end{table}

\begin{table}
\footnotesize
\begin{tabular}{l l l l l l l l l}
\multicolumn{7}{c}{20 Dimensional Sub-Exponential($0.5$), Number of Switches $N=10^6$, No Space Transformation.}\\ \hline
Algorithms  & ESS(SD)   & Median ESS  & ESS/Lik.Eval.  &  ESS/min          \\
ZZ 	        & 54925.7 (3113.5)      & 55460.3            & $0.1 \cdot 10^{-3}$	      &	66.2	  \\
SUZZ(0)	    & 80123.7 (1067.4)      & 80015.9      & $2.2 \cdot 10^{-3}$	      & 1600.3      \\
SUZZ(1)	        & {\bf 92356.6} (1459.1)      & {\bf 92214.6}    & ${\bf 4.3 \cdot 10^{-3}}$    & {\bf 3073.8}            \\
TRWM 	        & 2162.5 (56.7)      & 2162.6    & $2.2 \cdot 10^{-3}$   & 1544.7      \\
\end{tabular}
\caption{\vspace{1 mm} {\it SUZZ, ZZ and TRWM algorithms targeting a twenty-dimensional Sub-exponential($0.5$) distribution, with density given by (\ref{subexponential.20:1}). For SUZZ($k$), we use the speed function as in (\ref{speed.d20:1}). The algorithms ran until $N=10^6$ switches (or steps for the TRWM) occurred and the average ESS (with standard deviation in a parenthesis) along with the median ESS are presented. The median ESS per average likelihood evaluations and per average minutes of implementation time are also presented. All ESS concern the first coordinate of the twenty-dimensional process and are calculated without transforming the sample. The best performance is highlighted with bold letters.}}\label{numerics.Sub05.mean:1}
\end{table}

\begin{table}
\footnotesize
\begin{tabular}{l l l l l l l l l}
\multicolumn{7}{c}{20 Dimensional Student($3$), Number of Switches $N=10^6$, With Space Transformation.}\\ \hline
Algorithms  & ESS(SD)   & Median ESS       & Sq. $0.9$        & Sq. $0.99$            & Sq. $0.999$ & ESS/Lik.Eval.  &  ESS/min          \\
ZZ 	        & 16095.0 (717.8)      & 16151.2	      &	0.8980	          & 0.9888				 &  {\bf 0.9989} & $0.3 \cdot 10^{-3}$    & 734.1   \\
SUZZ(0)	    & {\bf 25882.6} (421.6)      & {\bf 25943.5}   	      & 0.8978            & 0.9892   & 0.9986      	     & ${\bf 1.4 \cdot 10^{-3}}$	&     1005.6      \\
SUZZ(1)	        & 23002.8 (511.0)      & 23052.3                &	0.8994	          &	{\bf 0.9898}				 & {\bf 0.9989}	    & $1.1 \cdot 10^{-3}$   & {\bf 1746.4}    \\
TRWM 	        & 1153.3 (80.6)      & 1158.5                 &	{\bf 0.8996}	          & {\bf 0.9902}				 & {\bf 0.9991}	    & $1.2 \cdot 10^{-3}$ & 827.5    \\
\end{tabular}
\caption{\vspace{1 mm} {\it SUZZ, ZZ and TRWM algorithms targeting a twenty-dimensional Student($3$) distribution with scale matrix given by (\ref{scale.matrix.student.3:1}). For SUZZ($k$), we use the speed function as in (\ref{speed.d20:1}). The algorithms ran until $N=10^6$ switches (or steps for the TRWM) occurred and the average ESS (with standard deviation in a parenthesis) along with the median ESS are presented. The median ESS per average likelihood evaluations and per average minutes of implementation time are also presented. All ESS concern the first coordinate of the twenty-dimensional process and are calculated after we transform the sample via the function $f$ as in (\ref{simulations.transformation:1}). An estimation of probabilities assigned to various squares of $\mathbb{R}^{20}$ by the target distribution is also presented. The square denoted by "Sq $a$" means that target assigns probability $a$ inside the square. The best performance is highlighted with bold letters.}}\label{numerics.20d.Student.3.pos.corr.fixed.N.SUZZ:1}
\end{table}

Finally, we should note that we tested SUZZ algorithms on targets where Assumption \ref{non.explo.condition:0} is not verified and the process will a.s.\ explode in finite time. Specifically, we targeted a one-dimensional Cauchy $\pi(x)=\frac{1}{Z}\left( 1+x^2 \right)^{-1}$
with SUZZ($1$), i.e. speed function $s(x)=1+x^2$. Very quickly there were numerical issues, with MATLAB reporting NaN. A diagnostic test we propose for one to check possible explosivity is to construct a large square (for example $[-10^8,10^8]^d$ for a $d$-dimensional target) and change the process so that whenever it hits the boundary of that square, a switch of direction occurs. Meanwhile one can count the number of times the process hit the boundary of the square. If the proportion of direction switches due to hitting the boundary over the overall number of switches is large, there is a good chance that the algorithm explodes and should not be used.

\section*{Acknowledgements}
We would like to thank Professor Anthony Lee for the indication of a simpler proof of Theorem \ref{non.geometric.ergodicity:0}. We would also like to thank George Deligiannidis and Krzysztof Latuszynski for helpful discussions. Finally, we would like to thank the associate editor and all five anonymous referees for their comments that vastly improved the quality of this manuscript.

%G. Vasdekis is supported by the Department of Statistics at the University of Warwick (NE/T00973X/1), under Dr. Richard Everitt. Most of this work was conducted while he was a PhD student at the MASDOC DTC between the Faculty of Mathematics and the Department of Statistics at the University of Warwick and a preliminary version of this work appears in his doctoral thesis \cite{vasdekis:20}.
G. Vasdekis was supported by the EPSRC as part of the MASDOC DTC (EP/HO23364/1) and the Department of Statistics at the University of Warwick (EP/N509796/1). G. O. Roberts was supported by EPSRC under the CoSInES (EP/R018561/1) and Bayes for Health (EP/R034710/1) programmes.

%............Bibliography...........................

\bibliographystyle{plain} % Style BST file (imsart-number.bst or imsart-nameyear.bst)
\bibliography{Biblio}       % Bibliography file (usually '*.bib')

\begin{appendix}

\section{A formal construction of the SUZZ process}\label{formal.cosntruction:00}
Formally, the SUZZ process is constructed as follows. 

Let $(\tilde{E}_n)_{ n \in \mathbb{N}}$ be i.i.d. exponential random variables with mean 1 and $(u_n)_{ n \in \mathbb{N}}$ be i.i.d. uniform in $[0,1]$ random variables, independent of the $\tilde{E}_n$'s. Suppose that the process starts from $(x,\theta) \in E$. Let $\{ \Phi_t(x,\theta) \}_{t \geq 0}$ be the flow on $\mathbb{R}^d$ that solves the ODE system 

\begin{equation}\label{det.dyn:001}
\left\{\begin{array}{l}
\dfrac{d\Phi_t(x,\theta)}{dt}=\theta s(\Phi_t(x,\theta)) , \ t \in [0,t^*(x,\theta))\ \vspace{1.5mm}\\ 
\Phi_0(x,\theta)=x,
\end{array}\right.
\end{equation}
where $t^*(x,\theta)=\sup \{ t \geq 0 : \Phi_t(x,\theta) \in \mathbb{R}^d \}$ is the explosion time of the flow, with the convention that if the flow does not explode this is set $t^*(x,\theta)=+\infty$. As we will see in Appendix \ref{ode.solution:00}, if $s \in C^1$, the ODE system (\ref{det.dyn:001}) has a unique solution, and this solution flow moves in a straight line parallel to $\theta \in \{ \pm 1 \}^d$. We define for all $(y,\eta) \in E$
\begin{equation}\label{def.lam:1}
\lambda(y,\eta)=\sum_{i=1}^d\lambda_i(y,\eta).
\end{equation}
Let 
\begin{equation*}
\tau_1=\inf \{ t \in [0,t^*(x,\theta)) : \int_0^t\lambda(\Phi_u(x,\theta),\theta)du  \geq \tilde{E}_1 \},
\end{equation*}
where we will always use the convention that $\inf \emptyset = +\infty$. Let %$\tau_1= \min \{ \tau_1^i , i \in \{ 1,...,d \} \}$,
$T_1=\tau_1$ %and $i_1=\argmin \{ \tau_1^i , i \in \{ 1,...,d \} \}$
.

If $\tau_1=\infty$ we set $(X_t,\Theta_t)=(\Phi_t(x,\theta),\theta)$ for all $t < t^*(x,\theta)$ and $(X_t,\Theta_t)=\partial$ for $t \geq t^*(x,\theta)$.

If $\tau_1<\infty$, we define $i_1$ to be the a.s.\ unique $i \in \{ 1,...,d \}$ that satisfies 
\begin{equation*}
u_1 \in \left[ \frac{\sum_{k=1}^{i-1}\lambda_k(\Phi_{\tau_1}(x,\theta),\theta)}{\lambda(\Phi_{\tau_1}(x,\theta),\theta)},\frac{\sum_{k=1}^{i}\lambda_k(\Phi_{\tau_1}(x,\theta),\theta)}{\lambda(\Phi_{\tau_1}(x,\theta),\theta)} \right] . 
\end{equation*}

We set $(X_t,\Theta_t)=(\Phi_t(x,\theta),\theta)$ for all $0 \leq t < T_1$. Then set $(X_{T_1},\Theta_{T_1})=(\Phi_{\tau_1}(x,\theta),F_{i_1}(\theta))$, where $F_{i_1}(\theta) \in \{ \pm 1 \}^d$ as in (\ref{multi.zz.flip.ssymbol:1}).

We then continue the construction inductively, for any $n \in \mathbb{N}$. If $T_n<\infty$ and assuming that the process is constructed until time $T_n$, we then consider the flow $ \{ \Phi_t\left(X_{T_n},\Theta_{T_n}\right) \}_{ t \geq 0}$, we let
 \begin{equation}\label{def.taun:1}
\tau_{n+1}=\inf \{ t \in [0,t^*(X_{T_n},\Theta_{T_n})) : \int_0^t\lambda(\Phi_u(X_{T_n},\Theta_{T_n}),\Theta_{T_n}) \ du  \geq \tilde{E}_{n+1} \}.
\end{equation}
%and let $\tau_{n+1}= \min \{ \tau_{n+1}^i , i \in \{ 1,...,d \} \}$,
and set $T_{n+1}=T_n+\tau_{n+1}$.

%and $i_{n+1}=\argmin \{ \tau_{n+1}^i , i \in \{ 1,...,d \} \}$.

If $\tau_{n+1}=\infty$, we set $(X_{T_n+t},\Theta_{T_n+t})=(\Phi_t(X_{T_n},\Theta_{T_n}),\Theta_{T_n})$ for all $t < t^*(X_{T_n},\Theta_{T_n})$ and $(X_{T_n+t},\Theta_{T_n+t})=\partial$ for $t \geq t^*(X_{T_n},\Theta_{T_n})$.

If $\tau_{n+1}<\infty$, we define $i_{n+1}$ to be the a.s.\ unique $i \in \{ 1,...,d \}$ that satisfies 
\begin{equation*}
u_{n+1} \in \left[ \frac{\sum_{k=1}^{i-1}\lambda_k(\Phi_{\tau_{n+1}}(X_{T_n},\Theta_{T_n}),\Theta_{T_n})}{\lambda(\Phi_{\tau_{n+1}}(X_{T_n},\Theta_{T_n}),\Theta_{T_n})},\frac{\sum_{k=1}^{i}\lambda_k(\Phi_{\tau_{n+1}}(X_{T_n},\Theta_{T_n}),\Theta_{T_n})}{\lambda(\Phi_{\tau_{n+1}}(X_{T_n},\Theta_{T_n}),\Theta_{T_n})} \right] . 
\end{equation*} 
Then set $(X_{T_n+t},\Theta_{T_n+t})=(\Phi_t(X_{T_n},\Theta_{T_n}),\Theta_{T_n})$ for all $0 \leq t < T_{n+1}-T_n$. Then set $\left(X_{T_{n+1}},\Theta_{T_{n+1}}\right)=\left(\Phi_{\tau_{n+1}}(X_{T_n},\Theta_{T_n}),F_{i_{n+1}}\left(\Theta_{T_n}\right)\right)$.

This defines the process until time 
$\xi$ as in (\ref{definition.of.xi:00}).
In the case where $\xi < \infty$, $\xi$ is the first time that the process has had infinitely many switches of direction. We set $(X_t,\Theta_t)=\partial$ for all $t \geq \xi$. 
%Then we let $\xi$ and $\zeta$ be as in (\ref{definition.of.xi:00}) and (\ref{definition.of.zeta:00}). Using the inductive construction, we have constructed the process until time $ \min\{ \xi,\zeta \}=\xi \wedge \zeta$. For formality let us define $Z_t=\partial$ for $t \geq \xi \wedge \zeta$. 

\section{Solution of ODE (\ref{suzz.heuristic.ode:01})}\label{ode.solution:00}
Here we explain why the ODE system (\ref{suzz.heuristic.ode:01}) has a unique solution when $s \in C^1$. Note that the solution flow $\{ \Phi_t(x,\theta) , t \geq 0\}$, representing the solution after time $t$ when starting from $(x,\theta)$, must move in a straight line in $\mathbb{R}^d$, parallel to $\theta \in \{ \pm 1 \}^d$. Therefore, if the process starts from $(x,\theta)=(x_1,...,x_d;\theta_1,...,\theta_d)$ , then at time $t$ is has position $\Phi_t(x,\theta)=(X^1(t),...,X^d(t))$ that satisfies for all $i$,
\begin{equation}\label{ode.solution:1}
    X^i(t)=x_i+\theta_1 \theta_i \left( X^1(t)-x_1\right)=y_i+\theta_1 \theta_i X^1(t)
\end{equation}
where
\begin{equation*}
    y_i=x_i-\theta_1 \theta_i x_1
\end{equation*}
and where we omit including the dependence of $X^i(t)$ on $(x,\theta)$ for notation convenience. 

Therefore, as long as we solve $\{ X^1(t), t \geq 0 \}$, we will have completely identified the solution $\{ \Phi_t(x,\theta) , t \geq 0 \}$. Also, from (\ref{suzz.heuristic.ode:01}) and (\ref{ode.solution:1})  we get that $X^1(t)$ satisfies
\begin{equation*}
\frac{d X^1(t)}{dt}=\theta_1 s\left( X^1(t),y_2+\theta_1\theta_2 X^1(t),...,y_d+\theta_1\theta_dX^1(t)\right)
\end{equation*}
Separating the variables we get that $X^1(t)$ satisfies
\begin{equation*}
\theta_1t=\int_{x_1}^{X^1_t}\frac{1}{s\left( u,y_2+\theta_1\theta_2u,...,y_d+\theta_1 \theta_d u \right) }du,
\end{equation*}
which has a unique solution
\begin{equation}
X^1(t)=f^{-1}(f(x_1)+\theta_1 t),
\end{equation} 
where
\begin{equation*}
   f(a)=\int_0^a\frac{1}{s\left( u,y_2+\theta_1\theta_2u,...,y_d+\theta_1\theta_du \right)} du
\end{equation*}
is invertible since it is continuous and strictly increasing. The solution is defined until the explosion time
\begin{equation}\label{explosion.time.formula:101}
   t^*(x,\theta)=\theta_1 \int_{x_1}^{\theta_1 \cdot \infty} \frac{1}{s\left( u,y_2+\theta_1\theta_2u,...,y_d+\theta_1\theta_du \right)} du \in (0,+\infty],
\end{equation}
where we use the convention $\theta_1 \cdot \infty$ to denote $+\infty$ if $\theta_1=+1$ and $-\infty$ if $\theta_1=-1$.

A natural family of speed functions, for which the ODE (\ref{suzz.heuristic.ode:01}) can be solved analytically is 
\begin{equation*}
    s(x)=\left(1+x^TBx\right)^{\frac{1+n}{2}}
\end{equation*}
where $n=0,1,2,3,...$ and $B$ a positive definite matrix. For $n=0$, the ODE solution is non-explosive for all starting value $(x,\theta) \in \mathbb{R}^d \times \{ -1,+1 \}^d$, while for $n>0$, the solution explodes in finite time. Here we will present two specific examples of speed functions that are used throughout the numerical simulations, along with the solution to the ODE they induce. More details can be found in Section 5.8.2 of \cite{vasdekis:20}.

\begin{itemize}
    \item {\bf Non-explosive deterministic dynamics:} A natural speed function to use is 
    \begin{equation*}
        s(x)=\left( 1+x^Tx \right)^{1/2}.
    \end{equation*}
    The solution to the ODE (\ref{suzz.heuristic.ode:01}) starting from $(x,\theta) \in E$, is given by (\ref{ode.solution:1}), where
    \begin{equation*}
        X^1(t)=\frac{b^2(t)-a}{2b(t)}-\frac{c}{d},  \  t \in [0,+\infty),
    \end{equation*}
    and
    \begin{align*}
        &b(t)=\left(Y_0+\sqrt{Y_0^2+a}\right) \cdot \exp \{ \sqrt{d}\theta_1 t \}, \ Y_0=x_1+\frac{c}{d}, \\
       & a=\frac{1+\| y \|_2^2}{d}-\frac{c^2}{d^2}, \ c=\theta_1 (y \cdot \theta), \  y_i=x_i-\theta_1 \theta_i x_1.
    \end{align*}
    
    \item {\bf Explosive deterministic dynamics:} We use the speed function
    \begin{equation*}
        s(x)=\left( 1+x^Tx \right).
    \end{equation*}
    The solution to the ODE (\ref{suzz.heuristic.ode:01}) starting from $(x,\theta) \in E$, is given by (\ref{ode.solution:1}), where
    \begin{equation*}
        X^1(t)=-c_1+c_2\tan{\left( \arctan{\left( \frac{x_1+c_1}{c_2}\right)}+\theta_1c_2d\ t \right)}, \ t \in [0,t^*(x,\theta)),
    \end{equation*}
    and
    \begin{equation*}
        c_2=\sqrt{\frac{1+y_2^2+...+y_d^2}{d}-c_1^2}, \ c_1=\frac{\left( y \cdot \theta \right)}{d}\theta_1, \ y_i=x_i-\theta_1 \theta_i x_1,
    \end{equation*}
and where the explosion time is
    \begin{equation*}
        t^*(x,\theta)=\frac{ \frac{\pi}{2}-\theta_1\arctan{\left( \frac{x_1+c_1}{c_2}\right)} }{c_2d}.
    \end{equation*}

\end{itemize}

\section{The generator}

\begin{proposition}\label{generator:1}
Let $(Z_t)_{t \geq 0}=(X_t,\Theta_t)_{t \geq 0}$ be a SUZZ process with strictly positive speed function $s \in C^2$. For any function $f \in C^1(E)$ and any $(x,\theta) \in E$,
\begin{equation}\label{generator.speed.up:1}
\lim_{t \rightarrow 0}\dfrac{\mathbb{E}_{x,\theta}[f(Z_t)]-f(x,\theta)}{t}=\mathcal{L}f(x,\theta):=\sum_{i=1}^d  \theta_i s(x) \partial_i f(x,\theta) + \lambda_i(x,\theta)(f(x,F_i(\theta))-f(x,\theta)) .
\end{equation}
\end{proposition}

Before we prove Proposition \ref{generator:1} we need the following technical results. 

\begin{definition}
Let $m \in \mathbb{N}$. We denote with $O_m$ the ball in $\mathbb{R}^d$, centred around $0$, having radius $m$. Recall that $\zeta_m=\inf\{ t \geq 0 : X_t \notin O_m \}$ is the first exit time of the $O_m$ for the SUZZ process and $\zeta=\lim_{m \rightarrow \infty}\zeta_m$ is the explosion time of the process. Finally, recall that $T_1,T_2,...$ are the switching times of the process and $\xi = \lim_{ n \rightarrow \infty} T_n$.
\end{definition}
\begin{lemma}\label{infinite.switches.lemma:1}
Assume that the speed function $s \in C^1$ is strictly positive and the rate functions $\lambda_i$ are locally bounded for all $i$. Then almost surely $\zeta_m < \xi$ for all $m \in \mathbb{N}$. Therefore, a.s.\ $\zeta \leq \xi$.
\end{lemma}

\begin{proof}[Proof of Lemma \ref{infinite.switches.lemma:1}]
Let $\lambda$ be as in (\ref{def.lam:1}) and let $m \in \mathbb{N}$. Let $\bar{\lambda}$ be an upper bound of $\lambda$ on $O_m$ and let $\tilde{E}_1,\tilde{E}_2,...$ be a configuration of the i.i.d. exponential random variables, with expectation 1, used to construct the $(X_t,\Theta_t)_{ t \geq 0}$ process, such that $\xi \leq \zeta_m$. We will show that the event of such a configuration has probability zero. In that configuration, for all $t < \xi$, we have $t < \zeta_m$, therefore $X_t \in O_m$. By the definition of the switching times $T_k$ in (\ref{def.taun:1}) (and writing $T_0=0$) we get
\begin{equation*}
\sum_{k=1}^{\infty}\tilde{E}_k=\sum_{k=1}^{\infty}\int_{T_{k-1}}^{T_k}\lambda (X_t,\Theta_t)dt \leq \xi \bar{\lambda}.
\end{equation*}
and therefore 
\begin{equation*}
\mathbb{P}\left( \{ \xi < \infty \} \cap \{ \xi \leq \zeta_m \} \right) \leq \mathbb{P}\left( \sum_{k=1}^{\infty}\tilde{E}_k \leq \xi \bar{\lambda}<\infty \right) =0.
\end{equation*}
Let $t_m$ be the maximum time it takes for a flow that starts from inside $O_m$ and solves (\ref{suzz.heuristic.ode:01}) to exit $O_m$. For any $n$, on the event $\{ \tilde{E}_n \geq \bar{\lambda}t_m  \}$, if the process has not escaped the ball $O_m$ until the $n-1$th switch, it does so following the dynamics before the $n$th switch occurs. Since $\mathbb{P} \left( \tilde{E}_n \geq \bar{\lambda}t_m  \right) =a>0$ we have for all $n$, $\mathbb{P}\left( \zeta_m > n \right) \leq (1-a)^n$ and therefore $\mathbb{P}\left( \zeta_m=+\infty \right) =0.$

Overall this gives, \begin{equation*}
    \mathbb{P}\left( \xi \leq \zeta_m \right) =\mathbb{P}\left( \left\{ \xi \leq \zeta_m \right\}  \cap \left\{ \xi < \infty \right\}\right) + \mathbb{P}\left( \left\{ \xi \leq \zeta_m \right\}  \cap \left\{ \xi = \infty \right\}\right) \leq \mathbb{P}\left(\zeta_m=\infty\right)=0,
\end{equation*}
which concludes the proof.
\end{proof}

\begin{lemma}%[Local Lemma]
\label{locality.lemma:1}
If the speed function $s \in C^1$ and the rate functions are locally bounded, then for all $x \in \mathbb{R}^d$ and any neighbourhood $U_x$ of $x$ there exists a time $t>0$ such that for any $\theta \in \{ \pm 1 \}^d$ if the SUZZ starts from $(x,\theta)$, then $X_s \in U_x$ for all $0 \leq s \leq t$.
\end{lemma}

\begin{proof}[Proof of Lemma \ref{locality.lemma:1}]
Let $x \in \mathbb{R}^d$ and consider a small neighbourhood $U_x$ of $x$. Let $\bar{s}$ be an upper bound for $s$ on $U_x$. Take $t$ small enough so that $t \sqrt{d} \bar{s} < dist(x,\partial U_x)$, where $dist$ denotes the Euclidean distance between a point and a set. This way, any path starting from $x$, and following a straight line, with speed function $s$ in each coordinate, for time less than $t$, will not have exit $U_x$. Then any path moving in directions $\{ \pm 1 \}^d$, with speed function $s(x)$ in each component, that switches direction finitely many times, will not have exit $U_x$. From Lemma \ref{infinite.switches.lemma:1}, a.s.\ the original Zig-Zag process will switch direction finitely many times until it exits the bounded set $U_x$ and this proves that the process a.s.\ stays inside $U_x$ until time $t$. 
\end{proof}

\begin{proof}[Proof of Proposition \ref{generator:1}]
Fix a starting point $(x,\theta) \in \mathbb{R}^d$. We know from Lemma \ref{locality.lemma:1} that for some neighbourhood $U_x$ of $x$ and for small $t_0$, if the process starts from $(x,\theta)$, then a.s. $X_s \in U_x$ for all $s \in [0,t_0]$. Therefore, the quantity $\mathbb{E}_{x,\theta}[f(Z_s)]$ is 
well-defined for all $s \in [0,t_0]$ so the limit makes sense. For the rest of the proof we will always assume that $t \leq t_0$.\\
Write 
 $S_i(t)=\{ \text{the } i \text{ coordinate switches before time } t \text{ and is the first coordinate to switch}\}$, for $i=1,..,d$ and 
$S_0(t)=\{ \text{ no coordinate switches until time } t \}$.
Note that if the process starts from $(x,\theta)$ and if $T_i$ is the first arrival time of the Poisson process with intensity $t \rightarrow \lambda_i(\Phi_t(x,\theta),\theta)$, then $\mathbb{P}\left( T_i \geq t \right)=\exp\left\{ -\int_0^t\lambda_i(\Phi_u(x,\theta),\theta) du \right\}$ therefore the density of $T_i$ is
\begin{equation*}
f_{T_i}(t)=\lambda_i(\Phi_t(x,\theta),\theta)\exp\left\{ -\int_0^t\lambda_i(\Phi_u(x,\theta),\theta)\ du  \right\}.
\end{equation*} 

For $f \in C^1(E)$, conditioning on the coordinate which was the first to switch before $t$ (or whether no switch occurred), we write
\begin{align}\label{eq.prop.b1:1}
&\dfrac{\mathbb{E}_{x,\theta}[f(X_t,\Theta_t)]-f(x,\theta)}{t} \\
&=\mathbb{P}_{x,\theta}(S_0(t))\dfrac{f(\Phi_t(x,\theta),\theta)-f(x,\theta)}{t}+ \sum_{i=1}^d \frac{\mathbb{E}_{x,\theta}\left[ \left(f(X_t,\Theta_t)-f(x,\theta) \right) 1_{S_i(t)} \right] }{t}. \nonumber
\end{align}

If we write $\lambda(x,\theta)=\sum_{i=1}^d\lambda_i(x,\theta)$ then we observe that
\begin{equation*}
\mathbb{P}_{x,\theta} \left( S_0(t) \right)=\exp \left\{ -\int_0^t\lambda(\Phi_u(x,\theta),\theta) du \right\} \xrightarrow{t \rightarrow 0}1.  
\end{equation*}
therefore

\begin{equation}\label{eq.prop.b1:2}
    \lim_{t \rightarrow 0}\mathbb{P}_{x,\theta}(S_0(t))\dfrac{f(\Phi_t(x,\theta),\theta)-f(x,\theta)}{t}=\sum_{i=1}^d\theta_is(x)\partial_if(x,\theta).
\end{equation}

Furthermore, for all $i \in \{ 1,...,d  \}$, conditioning on the first switch occurring at time $u \leq t$ and being of the $i$th coordinate, we have

\begin{align}\label{eq.prop.b1:3}
&\mathbb{E}_{x,\theta}\left[\left( f(X_t,\Theta_t) - f(x,\theta) \right) 1_{S_i(t)} \right] \\
&=\int_0^t\mathbb{E}_{\Phi_u(x,\theta),F_i(\theta)}\left[f(X_{t-u},\Theta_{t-u})-f(x,\theta)\right] \mathbb{P}_{x,\theta}\left( \text{no switches until time } u \text{ for any component } j \neq i  \right) \nonumber \\
& \ \ \ \ \ \ \ \ \ \ \cdot f_{T_i}(u)du \nonumber \\
&=\int_0^t\mathbb{E}_{\Phi_u(x,\theta),F_i(\theta)}[f(X_{t-u},\Theta_{t-u})-f(x,\theta)] \lambda_i(\Phi_u(x,\theta),\theta)\exp\left\{ -\int_0^u \lambda(\Phi_{u'}(x,\theta),\theta) du' \right\}du \nonumber \\
&=\int_0^t\mathbb{E}_{\Phi_u(x,\theta),F_i(\theta)}\left[f(X_{t-u},\Theta_{t-u})-f(x,F_i(\theta))\right] \lambda_i(\Phi_u(x,\theta),\theta)\exp\left\{ -\int_0^u \lambda(\Phi_{u'}(x,\theta),\theta) du' \right\}du \nonumber \\
&+\int_0^t\left(f(x,F_i(\theta))-f(x,\theta)\right) \lambda_i(\Phi_u(x,\theta),\theta)\exp\left\{ -\int_0^u \lambda(\Phi_{u'}(x,\theta),\theta) du' \right\}du. \nonumber 
\end{align}

Let $\epsilon >0$. Since $f \in C^0$ we can assume that the neighborhood $U_x$ is small enough such that if $y \in U_x$, then $|f(y,F_i\left(  \theta \right))-f(x,F_i\left(  \theta \right))| < \frac{\epsilon}{4}$. From Lemma \ref{locality.lemma:1} we know that if the process starts from $(x,\theta)$, then for any path that switches direction finitely many times by time $t$, we have for all $u' \leq t$, $X_{u'} \in U_x$. Therefore, for any $u \leq t$, if the process starts from $\left(\Phi_u\left(x,\theta \right),F_i(\theta)\right)$, then $X_{t-u} \in U_x$ a.s. Let $M$ be such that for any $y \in U_x$ and  $\eta \in \{ -1,+1 \}^d$, $|f(y,\eta)| \leq M$. Then, for any $u \leq t$, 

\begin{align*}
&\mathbb{E}_{\Phi_u(x,\theta),F_i(\theta)}\left[\left| f(X_{t-u},\Theta_{t-u})-f(x,F_i(\theta)) \right| \right] \\
&\leq \mathbb{P}_{\Phi_u(x,\theta),F_i(\theta)}\left(  \text{no switch by time } t-u \right)\left| f(\Phi_{t-u}\left(\Phi_u(x,\theta),F_i\left( \theta \right) \right), F_i\left( \theta \right) )-f(x,F_i(\theta)) \right| \\
&+ \mathbb{P}_{\Phi_u(x,\theta),F_i(\theta)}\left( \text{switch occurs by time } t-u \right) 2M
 <\frac{\epsilon}{4}+\frac{\epsilon}{4}=\frac{\epsilon}{2},
\end{align*}
when $t$ is small enough since $\lambda$ is bounded on $U_x$.
Therefore, since $\lambda_i \in C^0$ for all $i$ and $\Phi_u(x,\theta)$ is continuous over $u$, we get
\begin{align}\label{eq.prop.b1:4}
&\frac{1}{t}\left| \int_0^t\mathbb{E}_{\Phi_u(x,\theta),F_i(\theta)}\left[ f(X_{t-u},\Theta_{t-u})-f(x,F_i(\theta))\right] \lambda_i(\Phi_u(x,\theta),\theta)\exp\left\{ -\int_0^u \lambda(\Phi_{u'}(x,\theta),\theta) du' \right\}du \right| \\
& \leq \epsilon \lambda_i(x,\theta). \nonumber
\end{align}
Furthermore, from the fundamental Theorem of calculus, 
\begin{align}\label{eq.prop.b1:5}
&\frac{1}{t}\int_0^t\left(f(x,F_i(\theta))-f(x,\theta)\right) \lambda_i(\Phi_u(x,\theta),\theta)\exp\left\{ -\int_0^u \lambda(\Phi_{u'}(x,\theta),\theta) du' \right\}du \\
&\xrightarrow{t \rightarrow 0} \left( f(x,F_i(\theta)) -f(x,\theta) \right) \lambda_i(x,\theta). \nonumber
\end{align}

Combining (\ref{eq.prop.b1:3}), (\ref{eq.prop.b1:4}) and (\ref{eq.prop.b1:5}) we get that
\begin{equation*}
\lim_{t \rightarrow 0}\frac{\mathbb{E}_{x,\theta}\left[\left( f(X_t,\Theta_t) - f(x,\theta) \right) 1_{S_i(t)} \right]}{t}=\lambda_i(x,\theta)\left( f(x,F_i(\theta)) -f(x,\theta) \right)
\end{equation*}
for all $i \in \{ 1,...,d \}$, which combined with (\ref{eq.prop.b1:1}) and (\ref{eq.prop.b1:2}) proves the result.

\end{proof}

\section{Proof of Non-explosivity of SUZZ (Theorem \ref{non.explo:1})}\label{app.non.explo:00}

Before we prove non-explosivity of the process, we prove the following useful result, which is of independent interest. The result states that under Assumption \ref{non.explo.condition:0}, independently of the starting point, the process cannot follow the deterministic dynamics until the explosion time, but has to switch direction beforehand. Naturally, this is strongly connected with the notion of non-explosion and justifies the existence of Assumption \ref{non.explo.condition:0}.

\begin{proposition}\label{switch.before.boom:1}
Assume that $s \in C^2$ is strictly positive, the rates satisfy (\ref{rates.formula:2}) and Assumption \ref{non.explo.condition:0} hods. For any starting point $(x,\theta)$, let $t^*(x,\theta)$ be the explosion time of the deterministic flow solving the ODE (\ref{det.dyn:001}), and let $T_1$ be the first switching time of the process. Then $\mathbb{P}_{x,\theta}\left( T_1 < t^*(x,\theta) \right)=1$.
\end{proposition}

\begin{proof}[Proof of Proposition \ref{switch.before.boom:1}]
Suppose that the process starts from \\
$(x,\theta)=(x_1,..,x_d;\theta_1,..,\theta_d)$. 
As proven in Appendix \ref{ode.solution:00}, for any $t<T_1$ we have $X_t=(X^1_t,...,X^d_t)$ with $X^i_t$ as in (\ref{ode.solution:1}).

 Consider the Poisson process with rate $\{ m(t)=\lambda(X_t,\theta)=\sum_{i=1}^d\lambda_i(X_t,\theta), t \geq 0 \}$. If $y_i$ as in (\ref{ode.solution:1}), then using the fact that $a^+ \geq a$ for any $a \in \mathbb{R}$, and using the notation $\theta_1 \cdot \infty$ to denote $+\infty$ if $\theta_1=+1$ or $-\infty$ when $\theta_1=-1$, we get 
\begin{align*}
&\int_0^{t^*(x,\theta)}m(t)dt=\int_0^{t^*(x,\theta)}\lambda(X_t^1,y_2+\theta_1\theta_2X_t^1,...,y_d+\theta_1\theta_dX_t^1)dt\\
&=\int_{x_1}^{\theta_1 \cdot\infty}\lambda(u,y_2+\theta_1\theta_2u,..,y_d+\theta_1\theta_du)\dfrac{1}{s(u,y_2+\theta_1\theta_2u,..,y_d+\theta_1\theta_du)}\theta_1du  \\
& \geq \int_{x_1}^{\theta_1 \cdot\infty}\sum_{i=1}^d\theta_1\theta_i \partial_iU(u,y_2+\theta_1\theta_2u,..,y_d+\theta_1\theta_du)- \theta_1\theta_i\dfrac{\partial_is(u,y_2+\theta_1\theta_2u,..,y_d+\theta_1\theta_du)}{s(u,y_2+\theta_1\theta_2u,..,y_d+\theta_1\theta_du)}du\\
&=\int_{x_1}^{\theta_1 \cdot \infty}\dfrac{d}{du}[U(u,y_2+\theta_1\theta_2u,..,y_d+\theta_1\theta_du)-\log s(u,y_2+\theta_1\theta_2u,..,y_d+\theta_1\theta_du)]du\\
&=\lim_{u \rightarrow \theta_1 \cdot \infty}U(u,c_2+\theta_1\theta_2u,..,c_d+\theta_1\theta_du)-\log s(u,c_2+\theta_1\theta_2u,..,c_d+\theta_1\theta_du) - C= +\infty.
\end{align*}
by Assumption \ref{non.explo.condition:0}. Therefore, $\mathbb{P}_{x,\theta}\left( T_1 \geq t^*(x,\theta) \right)=\exp \left\{ -\int_0^{t^*(x,\theta)}m(t)dt \right\}=0$.
\end{proof}

To prove that the process will not explode, we use standard techniques from \cite{meyn.tweedie.3:93} which depend on the generator of the process. Although we can define the operator $\mathcal{L}$ in (\ref{generator.speed.up:1}), we cannot immediately conclude that this is the strong generator of the process. This is because the strong generator is defined through uniform convergence and we only defined  $\mathcal{L}$ in (\ref{generator.speed.up:1}) as a point-wise limit. This complicates the proof. However, the techniques in \cite{meyn.tweedie.3:93} only require us to use the generator of the process restricted in a bounded domain, which we introduce now. 

\begin{definition}\label{m.process}
Let $O_m$ be the ball of radius $m$ centered around $0$ and let $E_m=O_m \times \{ -1,+1 \}^d$. Starting from $(x,\theta) \in E$ we define the stopped $m$-process as the restriction of the SUZZ on $O_m$, stopped when exiting $O_m$, i.e. $(Z^m_t)_{t \geq 0}=(X^m_t,\Theta^m_t)_{t \geq 0}=(X_{t \wedge \zeta_m},\Theta_{t \wedge \zeta_m})_{t \geq 0}$.
\end{definition}

Since the switching rate of $Z^m$ is bounded as the process is defined on a bounded set and $\lambda_i$ are locally bounded, we have that for any $T>0$, if $N_T$ is the number of switching events before time $T$, then $\mathbb{E}_{x,\theta}[N_T]<\infty$ for any $(x , \theta) \in E_m$. Therefore $Z^m$ is a PDMP that can be seen in the setting of \cite{davis:84} and we have the following as a result of Theorem 5.5 of \cite{davis:84}.

\begin{proposition}\label{generator.m.process}
Let $\mathcal{L}$ the operator defined in (\ref{generator.speed.up:1}). The extended generator $\mathcal{L}^m$ for $Z^m$ has domain $\mathcal{D}(\mathcal{L}^m) \supset C^1(E)$ and for any function $f \in C^1(E)$ we have
\begin{equation*}
\mathcal{L}^mf(x,\theta)=\mathcal{L}f(x,\theta)1_{x \in O_m}.
\end{equation*}
\end{proposition}

Let $n \in \mathbb{N}$ such that (\ref{the.arbitrary.second.der.A.assumption:2}) holds. Let $A_i$ be as in (\ref{rates.formula:2}). For some $a \in (0,1)$ and $\delta>0$, consider the function
\begin{equation}\label{V.function:0}
V(x,\theta)=\exp \{ aU(x)- a\log s(x) + \sum_{i=1}^d \phi(\theta_i A_i(x)) \},
\end{equation}
where 
\begin{equation}\label{phi.function.ch5:1}
\phi(s)=\dfrac{1}{2} \sgn(s) h_{n+1}\left( \delta |s| \right),
\end{equation}
and $h_{n+1}$ as in
(\ref{iterative.logarithms:1}).
The proof of non-explosion relies on the following lemma.
\begin{lemma}\label{lyapunov:1}
 Assume that the rates satisfy (\ref{rates.formula:2}) and Assumptions \ref{non.explo.condition:0}, \ref{ass.lyapunov:1}, \ref{the.one.that.always.holds:1} and \ref{ass.lyapunov:2} hold. Let $\mathcal{L}$ be the operator defined in (\ref{generator.speed.up:1}). Then, there exist $a \in (0,1)$ and $\delta >0$ for which $V$ introduced in (\ref{V.function:0}) is a norm-like function, i.e.  $\lim_{\| x \| \rightarrow \infty}V(x,\theta)=+\infty $ and there exists a compact set $C$ and $b, c>0$ such that for all $(x,\theta) \in E$
\begin{equation}\label{lyapunov:2}
 \mathcal{L}V(x,\theta) \leq -c V(x,\theta) +b 1_{(x,\theta)\in C}.
\end{equation}
\end{lemma}

\begin{proof}[Proof of Lemma \ref{lyapunov:1}]
One can verify that $V \in C^1$ therefore $V \in D(\mathcal{L})$. Note that 
\begin{equation*}
V(x,F_i(\theta))-V(x,\theta)=V(x,\theta)\left( \exp\{ \phi(-\theta_iA_i(x))-\phi(\theta_iA_i(x)) \}-1 \right)
\end{equation*}
We then calculate
\begin{align}
\dfrac{\mathcal{L}V(x,\theta)}{V(x,\theta)}=\sum_{i=1}^d \{ &\theta_i a A_i(x) + \sum_{j=1}^d \theta_i \theta_j s(x) \partial_iA_j(x)\phi^{'}(\theta_jA_j(x)) \nonumber \\
&+ [(\theta_i A_i(x))^++\gamma_i(x)](\exp \{ \phi (-\theta_i A_i(x)) - \phi(\theta_i A_i(x)) \} -1) \} \label{lyapunov:2.5} \vspace{3mm}
\end{align}
\vspace{3mm}
Note that 
\begin{equation*}
\exp\{ \phi (-u) - \phi(u) \}=\left( 1+ h_n\left( \delta |u| \right)\right)^{-\sgn(u)}
\end{equation*}
and 
\begin{equation*}
\phi^{'}(u)=\dfrac{\delta}{2}\dfrac{1}{(1+\delta |u|)}\prod_{k=1}^n\frac{1}{1+ h_k(\delta |u|)} \leq \dfrac{\delta}{2}\dfrac{1}{(1+\delta |u|)(1+ \log ( 1+\delta |u|))}.
\end{equation*}
Consider the $i$th component of the sum in the RHS of (\ref{lyapunov:2.5}) and the following cases.

{\bf Case 1: }$\theta_i A_i(x) \geq 0$. Then the $i$th component of the sum in the RHS of (\ref{lyapunov:2.5}) can be written as
\begin{align}
&a |\theta_i A_i(x)| + \sum_{j=1}^d \theta_i \theta_j s(x) \partial_iA_j(x) \phi^{'}(\theta_j A_j(x)) \nonumber \\
&+\left(|\theta_i A_i(x)|+\gamma_i(x)\right)\left( \dfrac{1}{1+h_n\left(\delta \left| A_i(x) \right|\right)}-1 \right) \nonumber \\
& \leq |A_i(x)| \left[ a - 1 + \dfrac{1}{1+h_n\left(\delta \left| A_i(x) \right| \right)} \right] + \sum_{j=1}^d \theta_i \theta_j s(x) \partial_iA_j(x) \phi^{'}(\theta_j A_j(x))\label{lyapunov:3}
\end{align}
where we used that $\gamma_i(x) \geq 0$. 

{\bf Case 2: }$\theta_i A_i(x)<0$.  Then the $i$th component of the sum in the RHS of (\ref{lyapunov:2.5}) can be written as
\begin{align*}
& a \theta_i A_i(x) + \sum_{j=1}^d \theta_i \theta_j s(x) \partial_iA_j(x) \phi^{'}(\theta_j A_j(x)) \\
&+ [(\theta_i A_i(x))^+ +\gamma_i(x)] \underbrace{\log(1+ ... \log}_\text{$n$} ( 1+ \delta |A_i(x)|)...)  \\
& \leq -a|A_i(x)|+\bar{\gamma} \underbrace{\log( 1 + ... \log }_\text{$n$}(1+ \delta |A_i(x)| )... ) + \sum_{j=1}^d \theta_i \theta_j s(x) \partial_iA_j(x) \phi^{'}(\theta_j A_j(x))  \\
&\leq |A_i(x)|\left( -a+\bar{\gamma}\delta \right) + \sum_{j=1}^d \theta_i \theta_j s(x) \partial_iA_j(x) \phi^{'}(\theta_j A_j(x)),
\end{align*}
where we used that $\log(1+x) \leq x$ 
%%%%  IMPORTANT IF YOU ADD ^{1+\delta} :   and $\log(1+\delta x^{1+\delta}) \leq \delta x$ for $x \geq 0$. This second inequality can be checked by studying the first two derivatives of the function $\log(1+\delta x^{1+\delta})-\delta x$, as a function of $x$.
Combining the two different cases, we get overall that, 
\begin{align}\label{lyapunov:4}
\dfrac{\mathcal{L}V(x,\theta)}{V(x,\theta)} &\leq \sum_{i=1}^d|A_i(x)|\max \{ a-1+\dfrac{1}{1+h_n\left( \delta \left| A_i(x) \right| \right)},-a+\delta \bar{\gamma} \} \nonumber \\
&+ \sum_{i=1}^d \sum_{j=1}^d \theta_i \theta_j s(x) \partial_iA_j(x) \phi^{'}(\theta_j A_j(x)).
\end{align}
Let us set $\epsilon=\bar{\gamma}>0$, if $\bar{\gamma}>0$ or $0<\epsilon<A/(2d^2)$, if $\bar{\gamma}=0$. Since we have assumed that $A>3d\bar{\gamma}$ in Assumption \ref{ass.lyapunov:1}, we get that $\dfrac{A}{d}>\bar{\gamma}+2\epsilon$. We will choose $\delta>0$ small enough to be specified later and given $\delta$, we set $a=\delta \bar{\gamma} + \delta \epsilon$. Then the second part of the maximum of (\ref{lyapunov:4}) is equal to $-\delta \epsilon <0$. 

Consider the function 
\begin{equation*}
f(z)=\max \left\{ -\delta \epsilon , \delta \epsilon + \delta \bar{\gamma} -1 + \dfrac{1}{1+h_n\left( \delta z \right)} \right\},
\end{equation*}
so that the first term of the RHS of (\ref{lyapunov:4}) is equal to
\begin{equation*}
    \sum_{i=1}^d|A_i(x)|f(|A_i(x)|).
\end{equation*}
 Our goal will be to show that $\sum_{i=1}^d|A_i(x)|f(|A_i(x)|)<0$ for $\| x \|$ large enough. One can verify that 
\begin{equation*}
f(A)<0 \iff A \geq P(\delta)=
\dfrac{1}{\delta} h_n^{-1} \left( \frac{\delta \epsilon + \delta \bar{\gamma}}{1-\delta \epsilon - \delta \bar{\gamma}} \right)
\end{equation*}
and
\begin{equation*}
f(A)=-\delta \epsilon \iff A \geq M(\delta)=
\dfrac{1}{\delta} h_n^{-1} \left( \frac{2 \delta \epsilon + \delta \bar{\gamma}}{1-2\delta \epsilon - \delta \bar{\gamma}} \right)
\end{equation*}
where $h_n$ as in (\ref{iterative.logarithms:1}). Now, from L'H\^{o}pital's rule,
\begin{equation*}
\lim_{\delta \rightarrow 0}M(\delta)=\bar{\gamma}+2\epsilon \text{  ,  } \lim_{\delta \rightarrow 0}P(\delta)=\bar{\gamma}+\epsilon
\end{equation*}
so if we choose $\delta$ small enough, we have $M(\delta)<A/d$. Suppose $k= \argmax\{ |A_i(x)| : i=1,...,d \} $ so that 
\begin{equation*}
|A_k(x)| \geq \dfrac{\sum_{i=1}^d|A_i(x)|}{d}>\dfrac{A}{d}>M(\delta).
\end{equation*}
Therefore
\begin{equation}\label{A_k.upperbound:1}
|A_k(x)|f(|A_k(x)|) \leq -|A_k(x)|\delta \epsilon \leq -\dfrac{\sum_{i=1}^d|A_i(x)|}{d} \delta \epsilon.
\end{equation}
For any other coordinate $i$, the contribution to the sum $\sum_{i=1}^d|A_i(x)|f(|A_i(x)|)$ will be positive if and only if $|A_i(x)|\leq P(\delta)$. Then, using that $\dfrac{1}{1+h_n\left(\delta z 
\right)} \leq 1$, we can bound
\begin{equation*}
f(z)\leq \delta \epsilon + \delta \bar{\gamma}
\end{equation*}
for $z \geq 0$, and therefore
\begin{equation}\label{A_i.upperbound:1}
\sum_{i \neq k}^d|A_i(x)|f(|A_i(x)|) \leq (d-1) P(\delta)(\delta \epsilon + \delta \bar{\gamma}).
\end{equation}
Recall that when $\bar{\gamma}>0$, we have picked $\epsilon =\bar{\gamma}$, so due to (\ref{non.decay.of.rates:1}) we get 
\begin{equation}\label{again.one.more.lower.bound.on.A:100345}
A>d(d-1)\frac{(\bar{\gamma}+\epsilon)^2}{\epsilon}.
\end{equation}
On the other hand, if $\bar{\gamma}=0$ we have picked $\epsilon<A/(2d^2)$ so (\ref{again.one.more.lower.bound.on.A:100345}) holds in this case as well. Therefore
\begin{equation*}
\lim_{\delta \rightarrow 0}\dfrac{1}{\sum_{i=1}^d|A_i(x)|}(d-1)(\bar{\gamma}+\epsilon)P(\delta)-\dfrac{\epsilon}{d}=\dfrac{1}{\sum_{i=1}^d|A_i(x)|}(d-1)(\bar{\gamma}+\epsilon)^2-\dfrac{\epsilon}{d}<0.
\end{equation*}
Combining this with (\ref{A_k.upperbound:1}) and (\ref{A_i.upperbound:1}) we get
\begin{align}\label{lyapunov:5}
&\sum_{i=1}^d|A_i(x)|f(|A_i(x)|)=|A_k(x)|f(A_k(x))+\sum_{ i \neq k}|A_i(x)|f(|A_i(x)|)  \nonumber \\
& \leq \delta \sum_{i=1}^d|A_i(x)| \left[ - \dfrac{\epsilon}{d} + \dfrac{1}{\sum_{i=1}^d|A_i(x)|}(d-1)(\bar{\gamma}+\epsilon)P(\delta) \right] \leq  -c  \sum_{i=1}^d|A_i(x)| 
\end{align}
for some $c>0$, assuming $\delta$ is small enough.

To finish the proof of the drift condition (\ref{lyapunov:2}), let us consider the last term of the RHS in (\ref{lyapunov:4}). Here, due to (\ref{the.arbitrary.second.der.A.assumption:1}) and assuming that $x \notin C$ for some compact set large enough, we can write
\begin{align}\label{lyapunov:6}
&\sum_{j=1}^d \theta_i \theta_j s(x) \partial_iA_j(x) \phi^{'}(\theta_j A_j(x)) \leq \sum_{i,j=1}^d\dfrac{\delta}{2}\dfrac{s(x)|\partial_iA_j(x)|} {(1+\delta|A_j(x)|)(1+\log(1+\delta |A_j(x)|))} \nonumber \\
&=\left ( \sum_{k=1}^d|A_k(x)| \right ) \frac{1}{2} \sum_{i=1}^d\sum_{j=1}^d \dfrac{s(x)|\partial_iA_j(x)|}{\sum_{k=1}^d|A_k(x)|}\dfrac{1 }{\left( \delta^{-1}+ |A_j(x)| \right) \left(1+ \log \left( 1+ \delta |A_j(x)| \right) \right)}   \nonumber \\  
&\leq \sum_{i=1}^d|A_i(x)|\dfrac{c}{2}.
\end{align}
Then, combining (\ref{lyapunov:4}), (\ref{lyapunov:5}) and (\ref{lyapunov:6}) we get for $x \notin C$,
\begin{equation*}
\dfrac{\mathcal{L}V(x,\theta)}{V(x,\theta)} \leq -\dfrac{c}{2}\sum_{i=1}^d|A_i(x)| \leq -\dfrac{c}{2}A,
\end{equation*}
and this proves (\ref{lyapunov:2}) since $V$ and $\mathcal{L}V$ are bounded on $C$.

Finally, we need to prove that $\lim_{\| x \| \rightarrow \infty}V(x,\theta)=+ \infty$. 
For this, we can find $C,C'>0$ such that for all $\| x \|$ large enough, 
\begin{align*}
&V(x,\theta)  \geq C\exp \left\{ aU(x)-a \log s(x) -\sum_{i=1}^d \dfrac{1}{2}h_{n+1}\left(\delta |A_i(x)|\right) \right\}  \\
& \geq C \exp \left\{ aU(x)-a \log s(x) \right\}  \prod_{i=1}^d(1+h_n(\delta |A_i(x)|))^{-1/2}    \\
& \geq C \exp \{ aU(x)-a \log s(x) \} \left(1 + h_n\left( s(x) \| \nabla\left( U(x)-\log s(x)\right)  \|_1 \right) \right)^{-d/2}  \\
 &\geq C' \left( \left( U(x)-\log s(x) \right) \left(1 + h_n\left( s(x) \| \nabla\left( U(x)-\log s(x)\right)  \|_1 \right) \right)^{-1} \right)^{d/2} \xrightarrow{\| x \| \rightarrow \infty} + \infty,
\end{align*} 
due to (\ref{the.arbitrary.second.der.A.assumption:2}) and Assumption \ref{non.explo.condition:0}, and where we used that $\delta \leq 1$ in the third inequality. This completes the proof. 
\end{proof}

\begin{proof}[Proof of Theorem \ref{non.explo:1}]
Under the assumptions of Theorem \ref{non.explo:1}, we get from Lemma \ref{lyapunov:1} and Proposition \ref{generator.m.process} that there exists a norm-like function $V$ and constants $c,b>0$ such that $\mathcal{L}^mV(x,\theta)\leq c V(x,\theta)+b$ for all $m \in \mathbb{N}$. The assumptions of Theorem 2.1 in \cite{meyn.tweedie.3:93} are satisfied and this proves that the process is non-explosive, i.e. if $\zeta$ as in (\ref{definition.of.zeta:00}) then $\zeta = + \infty$ a.s. Finally, from Lemma \ref{infinite.switches.lemma:1} if $\xi$ as in (\ref{definition.of.xi:00}) then $\xi =+ \infty$ a.s. 
\end{proof}

\section{Proof of Theorem \ref{inv.measure.speed.up:1} (Invariant Measure)}\label{app.invariant:00}

For this section, we first recall the definition of the strong generator of the process.

\begin{definition}
Let $\left( P^t \right)_{t \geq 0}$ be the transition semigroup of the process. We define $\mathcal{D}(\mathcal{A})$ to be the set of all the Borel functions $f$ such that the limit 
\begin{equation*}
    \lim_{t \rightarrow 0}\frac{P^tf-f}{t}
\end{equation*}
exists in the uniform norm (over $(x,\theta) \in E$). We define the strong generator as the operator $\mathcal{A}$, acting on any $f \in \mathcal{D}(\mathcal{A})$ as
\begin{equation*}
    \mathcal{A}f(x,\theta)=\lim_{t \rightarrow 0}\frac{P^tf(x,\theta)-f(x,\theta)}{t}.
\end{equation*}
\end{definition}
We begin by formally proving that a large class of functions belong to the domain of the strong generator of the SUZZ process and for these functions the strong generator is given by the operator $\mathcal{L}$ introduced in (\ref{generator.speed.up:1}).

\begin{lemma}\label{lemma.C1.c.in.the.domain}
Let us assume that the rates satisfy (\ref{rates.formula:2}), that Assumptions \ref{ass.lyapunov:1}, \ref{the.one.that.always.holds:1}, \ref{ass.lyapunov:2} and \ref{non.evanescence.assumption:1} hold and that $(Z_t)_{t \geq 0}$ is a SUZZ process with speed function $s$. If $f \in C^1_{c}(E)$ then $f$ is in the domain of the strong generator $\mathcal{A}$ of $Z$ and $\mathcal{A}f=\mathcal{L}f$, where $\mathcal{L}$ is the operator defined in (\ref{generator.speed.up:1}). This means that 
\begin{equation*}
\left|\frac{\mathbb{E}_{x,\theta}\left[ f(X_t,\Theta_t)\right]-f(x,\theta)}{t}-\mathcal{L}f(x,\theta) \right| \xrightarrow{t \rightarrow 0}0
\end{equation*}
uniformly in $(x,\theta) \in E$.
\end{lemma}

\begin{proof}[Proof of Lemma \ref{lemma.C1.c.in.the.domain}]
Let $K$ be a compact set that contains the support of $f$ and let 
\begin{equation*}
K'= \{ (x,\theta): \text{ there exists a } y, \text{ with } \| x-y \|<\epsilon , (y,\theta) \in K \}=K+\epsilon B(0,1)
\end{equation*}
for some $\epsilon >0$. Let $\bar{s}$ be an upper bound on the speed function $s$ on $K'$. Then for all $t < t_1=\epsilon/(\sqrt{d}\bar{s})$, if the process starts from any $(x,\theta) \notin K'$, then the process will not have hit $K$ until time $t$, and since the support of $f$ is contained in $K$, $P^tf(x,\theta)=\mathbb{E}_{x,\theta}[f(X_t)]=0$, for all $(x,\theta) \notin K'$. Note also that for all $(x,\theta) \notin K'$, $\mathcal{L}f(x,\theta)=0$ and $f(x,\theta)=0$.\\
Now, let us focus on $(x,\theta) \in K'$. Pick \begin{equation*}
K''=K'+\epsilon B(0,1)
\end{equation*}
and let $\bar{\bar{s}}$ be an upper bound of $s$ on $K''$. Then for all $t \leq t_2=\epsilon/(\sqrt{d}\bar{\bar{s}})<t_1$, the process starting from $K'$ will not have exited $K''$ by time $t$ and if we cover $K''$ by some $O_m$ for some large $m$, then a.s.\ $Z(t)=Z^m(t)$ for all $t \leq t_2$ as long as we start from somewhere in $K'$. Then, for any $(x,\theta) \in K'$ and any $t<t_2$ 
\begin{align*}
\frac{\mathbb{E}_{x,\theta}[f(Z_t)]-f(x,\theta)}{t}-\mathcal{L}f(x,\theta)=\frac{\mathbb{E}_{x,\theta}[f(Z^m_t)]-f(x,\theta)}{t}-\mathcal{L}^m f(x,\theta),
\end{align*}
so overall for all $t < t_2$
\begin{align*}
&\sup_{(x,\theta) \in E}\left| \frac{\mathbb{E}_{x,\theta}[f(Z_t)]-f(x,\theta)}{t}-\mathcal{L}f(x,\theta) \right|  \leq \sup_{(x,\theta) \in K''} \left| \frac{\mathbb{E}_{x,\theta}[f(Z^m_t)]-f(x,\theta)}{t}-\mathcal{L}^mf(x,\theta) \right| \xrightarrow{t \rightarrow 0}0,
\end{align*}
where the convergence can be seen to hold using the proof of Proposition 15b of \cite{durmus.guillin.monmarche:18.2}.
\end{proof} 

We also have the following.

\begin{lemma}\label{path.differentiability.of.semigroup:1}
Assume that the assumptions of Lemma \ref{lemma.C1.c.in.the.domain} hold. If $f \in C^1_c$, then for all $t_0>0$, $P^{t_0}f$ is differentiable along the deterministic flow of the SUZZ process, i.e. for all $(x,\theta) \in E$ there exists a function $D P^{t_0}f (x,\theta) :E \rightarrow \mathbb{R}$ such that for all $(x,\theta) \in E$,
\begin{equation*}
\lim_{t \rightarrow 0}\frac{P^{t_0}f(\Phi_t(x,\theta),\theta)-P^{t_0}f(x,\theta)}{t}= D P^{t_0}f(x,\theta).
\end{equation*}
\end{lemma}

\begin{proof}[Proof of Lemma \ref{path.differentiability.of.semigroup:1}]
Fix $t_0>0$ and let's write $g=P^{t_0}f$ for notational convenience. Due to Lemma \ref{lemma.C1.c.in.the.domain}, $f$ is in the domain of the strong generator of the SUZZ process, therefore, using standard results (see for example \cite{ethier.kurtz:86}), $g$ is also in the domain of the strong generator of the SUZZ process. From the proof of Proposition \ref{generator:1}, rewriting (\ref{eq.prop.b1:1}) we get
\begin{align*}\label{eq.prop.b1:1}
&\mathbb{P}_{x,\theta}(S_0(t))\dfrac{g(\Phi_t(x,\theta),\theta)-g(x,\theta)}{t}\\
&= \dfrac{\mathbb{E}_{x,\theta}[g(X_t,\Theta_t)]-g(x,\theta)}{t}- \sum_{i=1}^d \frac{\mathbb{E}_{x,\theta}\left[ \left(g(X_t,\Theta_t)-g(x,\theta) \right) 1_{S_i(t)} \right] }{t}. \nonumber
\end{align*}
The first term of the RHS is finite since $g$ is in the domain of the strong generator of SUZZ and the second term can be seen to be finite using the same argument as in the proof of Proposition \ref{generator:1}. Since $\lim_{t \rightarrow 0}\mathbb{P}_{x,\theta}(S_0(t))=1$, we get the result.
\end{proof}

\noindent The following lemma is the stepping stone to prove Theorem \ref{inv.measure.speed.up:1}.

\begin{lemma}\label{zero.generator:2}
Assume that the rates satisfy (\ref{rates.formula:2}), Assumptions \ref{ass.lyapunov:1}, \ref{the.one.that.always.holds:1}, \ref{ass.lyapunov:2} and \ref{non.evanescence.assumption:1} hold and the $(Z_t)_{t \geq 0}$ is a SUZZ process with speed function $s$. Let $(P^t)_{t \geq 0}$ be the transition semi-group of a SUZZ process with speed function $s$ and let $\mathcal{A}$ be the strong generator of the SUZZ process. 
 If $f \in C^1_{c}(E)$ then for all $t_0>0$, $P^{t_0}f \in \mathcal{D}(\mathcal{A})$ and 
\begin{equation}\label{zero.generator:3}
\int_{E}\mathcal{A}P^{t_0}f(x,\theta)\mu(dx,d\theta)=0.
\end{equation}
\end{lemma}
\noindent If we could guarantee that for any $f \in C^1_c(E)$, for all $t_0 >0$, $P^{t_0}f \in C^1_c(E)$, then Lemma \ref{zero.generator:2} would be easy to verify, using similar calculations to the proof of Proposition 5 of \cite{bierkens.roberts_scaling:2017}. 
In our setting, due to the fact that we allow explosive deterministic dynamics, we cannot guarantee that $P^{t_0}f \in C^1_c$. However,
Lemma \ref{path.differentiability.of.semigroup:1} guarantees that the function $P^{t_0}f$ must have a derivative along lines parallel to the vectors $\{ -1,+1 \}^d$. Therefore, the fundamental theorem of calculus and an integration by parts can be used along such lines. When we will integrate $\mathcal{A}P^{t_0}f$ over the ball $O_m$, we may do the integration over many different lines parallel to some vector $\{ -1,+1 \}^d$ and apply the integration by parts technique in each of these lines to get the result. This is the main idea of the following proof.

\begin{proof}[Proof of Lemma \ref{zero.generator:2}]
We begin by noticing that, since $f \in \mathcal{D}(\mathcal{A})$, $P^{t_0}f \in \mathcal{D}(\mathcal{A})$ as well (see for example Proposition 1.1.5 of \cite{ethier.kurtz:86}). If $E_m=O_m \times \{ \pm 1 \}^d$ then
\begin{equation*}
\left| \int_{E_m}\mathcal{A}P^{t_0}f(x,\theta)\mu(dx,d\theta)-\int_{E}\mathcal{A}P^{t_0}f(x,\theta)\mu(dx,d\theta) \right| \xrightarrow{m \rightarrow \infty}0,
\end{equation*}
since $E_m \nearrow E$ and $\mathcal{A}P^{t_0}f=P^{t_0}\mathcal{A}f=P^{t_0}\mathcal{L}f$ is bounded, due to that $\mathcal{L}f$ is bounded. Therefore, it suffices to prove that 
\begin{equation}
\lim_{m \rightarrow \infty}\int_{E_m}\mathcal{A}P^{t_0}f(x,\theta)\mu(dx,d\theta)=0.
\end{equation}
From now on, let us write $g=P^{t_0}f$ for notational convenience. From Lemma \ref{path.differentiability.of.semigroup:1} we get that $g$ has a derivative along the deterministic dynamics of the SUZZ process. This means that
 there exists a function $D g :E \rightarrow \mathbb{R}$ such that if $X_t$ satisfies ODE (\ref{suzz.heuristic.ode:01}) with starting point $(x,\theta)$ then for all $t \geq 0$
\begin{equation*}
g(X_t,\theta)-g(x,\theta)=\int_0^tD g(X_u,\theta)du.
\end{equation*}
Furthermore, using the same argument as in the proof of Proposition \ref{generator:1} we get that for all $(x,\theta) \in E_m$,
\begin{equation*}
\mathcal{A}g(x,\theta)=D g(x,\theta)+\sum_{i=1}^d\lambda_i(x,\theta)\left( g(x,F_i(\theta))-g(x,\theta) \right).
\end{equation*}
Our goal is to use an integration by parts technique to control the first part of the sum of the generator. We fix a $\theta \in \{ -1,+1 \}^d$. We use a linear, invertible transformation $A$ on $\mathbb{R}^d$ such that $A\theta=\sqrt{d}e_1$, $AO_m=O_m$ and $|detA|=1$ and let $e_1=(1,0,...,0) \in \mathbb{R}^d$. We use the transformation $y=(y_1,...,y_d)=Ax$. Also, given $y_2,y_3,...,y_d$ with $y_2^2+...+y_d^2<m^2$ we write 
\begin{equation*}
y_1^*=\theta_1 \sqrt{m^2-y_2^2-...-y_d^2}, 
\end{equation*}
and we omit the dependence on $y_2,...,y_d$ and $\theta_1$ for ease of notation. We also write $x_0=A^{-1}\left( -y_1^*,y_2,...,y_d \right) '$. 

We further consider the solution $X_t$ to the ODE (\ref{suzz.heuristic.ode:01}) starting from $(x_0,\theta)$ and we write $Y_t=(Y^1_t,...,Y^d_t)=AX_t$ so that $Y_t$ starts from $(-y_1^*,y_2,...,y_d)$ and solves the ODE $dY_t/dt=\sqrt{d}s(A^{-1}Y_t) e_1$. Also write $t^*$ such that $Y^1_{t^*}=y_1^*$. 
Then we can write,
\begin{align*}
&\int_{O_m}D g(x,\theta)\exp\{ -U(x) \}dx=\int_{O_m}D g(A^{-1}y,\theta)\exp\{ -U(A^{-1}y) \}dy \\
&= \int_{-m}^m\int_{-\sqrt{m^2-y_2^2}}^{\sqrt{m^2-y_2^2}}...\int_{-y_1^*}^{y_1^*}D g(A^{-1}y,\theta)\exp\{ -U(A^{-1}y) \}dy_1dy_d...dy_2\\
&=\int_{-m}^m\int_{-\sqrt{m^2-y_2^2}}^{\sqrt{m^2-y_2^2}}...\int_{0}^{t^*}D g(A^{-1}Y_t,\theta)\exp\{ -U(A^{-1}Y_t) \}s(A^{-1}Y_t)\sqrt{d} \ dt \ dy_d...dy_2
\end{align*}
Having fixed $y_2,...,y_d$, write $z_1=\left( -y_1^*,y_2,...,y_d \right), z_2=\left( y_1^*,y_2,...,y_d \right) \in \partial O_m$ and $y=(y_1,...,y_d)$ and using integration by parts we get
\begin{align*}
&\int_{0}^{t^*}D g(A^{-1}Y_t,\theta)\exp\{ -U(A^{-1}Y_t) \}s(A^{-1}Y_t)\sqrt{d} \ dt \\
&=g(A^{-1}z_2,\theta)\exp\{ -U(A^{-1}z_2) \}s(A^{-1}z_2)\sqrt{d}-g(A^{-1}z_1,\theta)\exp\{ -U(A^{-1}z_1) \}s(A^{-1}z_1)\sqrt{d}\\
&-\int_{0}^{t^*} g(A^{-1}Y_t,\theta)\frac{d}{dt}\left[ \exp\{ -U(A^{-1}Y_t) \}s(A^{-1}Y_t) \right]\sqrt{d} \ dt\\
&=g(A^{-1}z_2,\theta)\exp\{ -U(A^{-1}z_2) \}s(A^{-1}z_2)\sqrt{d}-g(A^{-1}z_1,\theta)\exp\{ -U(A^{-1}z_1) \}s(A^{-1}z_1)\sqrt{d}\\
&-\int_{0}^{t^*} g(A^{-1}Y_t,\theta)\exp\{ -U(A^{-1}Y_t) \}\\
& \hspace{10 mm} \sum_{i=1}^d \left\{ -\partial_iU(A^{-1}Y_t) \theta_i s(A^{-1}Y_t) +   \partial_is(A^{-1}Y_t)\theta_i \right\} s(A^{-1}Y_t)\sqrt{d} \ dt\\
&=g(A^{-1}z_2,\theta)\exp\{ -U(A^{-1}z_2) \}s(A^{-1}z_2)\sqrt{d}-g(A^{-1}z_1,\theta)\exp\{ -U(A^{-1}z_1) \}s(A^{-1}z_1)\sqrt{d}\\
&-\int_{-y_1^*}^{y_1^*} g(A^{-1}y,\theta)\exp\{ -U(A^{-1}y) \} \sum_{i=1}^d \left( - \partial_iU(A^{-1}y) \theta_i s(A^{-1}y)  +  \partial_is(A^{-1}y)\theta_i \right)dy_1.
\end{align*}
Overall 
\begin{align}\label{int.by.parts.ch5.4:1}
&  \sum_{\theta \in \{ \pm 1 \}^d}\int_{O_m}D g(x,\theta)\exp\{ -U(x) \}dx  \\
&=\sum_{\theta \in \{ \pm 1 \}^d}\int_{-m}^m\int_{-\sqrt{m^2-y_2^2}}^{\sqrt{m^2-y_2^2}}... \int_{-\sqrt{m^2-y_2^2-...-y_{d-1}^2}}^{\sqrt{m^2-y_2^2-...-y_{d-1}^2}} g(A^{-1}z_2,\theta)\exp\{ -U(A^{-1}z_2) \} 
 s(A^{-1}z_2)\sqrt{d}  \nonumber \\
 & \ \ \ \ \ \ \ \ \ \ \ \ \ \ \ \ \ \ \ \  dy_{d} \ ... \ dy_3 \ dy_2  \nonumber \\
&-\sum_{\theta \in \{ \pm 1 \}^d}\int_{-m}^m\int_{-\sqrt{m^2-y_2^2}}^{\sqrt{m^2-y_2^2}}... \int_{-\sqrt{m^2-y_2^2-...-y_{d-1}^2}}^{\sqrt{m^2-y_2^2-...-y_{d-1}^2}}  g(A^{-1}z_1,\theta)\exp\{ -U(A^{-1}z_1) \}  s(A^{-1}z_1)\sqrt{d} \nonumber \\
& \ \ \ \ \ \ \ \ \ \ \ \ \ \ \ \ \ \ \ \  dy_{d} \ ... \ dy_3 \ dy_2  \nonumber \\
&+\sum_{\theta \in \{ \pm 1 \}^d}\int_{O_m}g(x,\theta)\exp\{ -U(x) \} \sum_{i=1}^d \left( \theta_i\partial_iU(x) s(x)  -  \theta_i \partial_is(x)  \right)dx \nonumber.
\end{align}
On the other hand, rearranging the sum over $\theta$ (see also the proof of Proposition 5 of \cite{bierkens.roberts_scaling:2017}), we get
\begin{align}\label{rearrange.sum.ch.5.4:2}
&\sum_{i=1}^d\sum_{\theta \in \{ \pm 1 \}^d}\int_{O_m}\lambda_i(x,\theta) \left( g(x,F_i(\theta))-g(x,\theta) \right)\exp\{ -U(x) \}dx \nonumber \\
& = -\sum_{i=1}^d\sum_{\theta \in \{ \pm 1 \}^d}\int_{O_m}g(x,\theta)\exp \{ -U(x) \}\left( \lambda_i(x,\theta)-\lambda_i(x,F_i(\theta)) \right) dx.
\end{align}
Recall that since the rates satisfy (\ref{rates.formula:2}), we have 
\begin{equation*}
\lambda_i(x,\theta)-\lambda_i(x,F_i(\theta))=\theta_i\partial_iU(x)s(x)  -  \theta_i\partial_is(x).
\end{equation*}
When we integrate $\int_{E_m}\mathcal{A}g d\mu$, we get the sum of the RHS of equations (\ref{int.by.parts.ch5.4:1}) and (\ref{rearrange.sum.ch.5.4:2}). On this sum, only the boundary parts remain and we have
\begin{align*}
&\left| 2^d H \int_{E_m}\mathcal{A}g(x,\theta)\mu(dx,d\theta) \right|=\left| \sum_{\theta \in \{ \pm 1 \}^d}\int_{O_m}\exp\{ -U(x) \} \mathcal{A}g(x,\theta)dx \right| \\
& \leq \sum_{\theta \in \{ \pm 1 \}^d}\int_{-m}^m\int_{-\sqrt{m^2-y_2^2}}^{\sqrt{m^2-y_2^2}}... \int_{-\sqrt{m^2-y_2^2-...-y_{d-1}^2}}^{\sqrt{m^2-y_2^2-...-y_{d-1}^2}} \left|  g(A^{-1}z_2,\theta)\exp\{ -U(A^{-1}z_2) \} 
s(A^{-1}z_2)\sqrt{d} \right| \\
& \ \ \ \ \ \ \ \ \ \ \ \ \ \ \ \ \ \ \ \  dy_{d} \ ... \ dy_3 \ dy_2  \\
&+  \sum_{\theta \in \{ \pm 1 \}^d}\int_{-m}^m\int_{-\sqrt{m^2-y_2^2}}^{\sqrt{m^2-y_2^2}}... \int_{-\sqrt{m^2-y_2^2-...-y_{d-1}^2}}^{\sqrt{m^2-y_2^2-...-y_{d-1}^2}} \left| g(A^{-1}z_1,\theta)\exp\{ -U(A^{-1}z_1) \} s(A^{-1}z_1)\sqrt{d} \right| \\
&  \ \ \ \ \ \ \ \ \ \ \ \ \ \ \ \ \ \ \ \  dy_{d} \ ... \ dy_3 \ dy_2  \\
& \leq 2 \sqrt{d} \| g \|_{\infty} \sup_{x \in \partial O_m}\{ \exp \{ -U(x) \} s(x) \} \int_{x \in \partial O_m}1 dx  \\
& \leq C 2 \sqrt{d} \| g \|_{\infty} \sup_{x \in \partial O_m}\{ \exp \{ -U(x) \} s(x) \} m^{d-1} \xrightarrow{m \rightarrow \infty}0,
\end{align*}
where the convergence holds due to Assumption \ref{non.evanescence.assumption:1}. Here $\| g \|_{\infty}$ is well-defined since $g=P^{t_0}f$ is bounded since $f$ is bounded. This completes the proof.
\end{proof}

 Now, we can conclude with the proof of invariance.
\begin{proof}[Proof of Theorem \ref{inv.measure.speed.up:1}]
Let $(P^t)_{t \geq 0}$ be the transition semi-group of the process and $\mathcal{L}$ the operator defined in (\ref{generator.speed.up:1}). Let $f \in C^1_c(E)$. From Lemma \ref{lemma.C1.c.in.the.domain}, $\mathcal{L}f=\mathcal{A}f$ is the  strong generator of $f$. Because of Dynkin's formula, for any $t > 0$,
\begin{equation*}
P^tf(x,\theta)-f(x,\theta)=\int_0^t \mathcal{A}P^{t_0}f(x,\theta)dt_0.
\end{equation*}
Since $s$ and $\lambda$ are bounded on compact sets, for any $f \in C^{\infty}_c(E)$ we have that $\mathcal{L}f$ is bounded and after integrating both sides over $\mu$ and using Fubini's theorem, we get
\begin{align*}
&\int_E P^tf(x,\theta)  \mu(dx,d\theta) - \int_E f(x,\theta)  \mu (dx,d\theta)= \int_E \int_0^t \mathcal{A}P^{t_0}f(x,\theta)dt_0 \mu (dx,d\theta) \\
&=\int_0^t \int_E \mathcal{A}P^{t_0}f(x,\theta)    \mu (dx,d\theta) dt_0 = 0
\end{align*}
where the last equality follows from Lemma \ref{zero.generator:2}. Therefore, for all $f \in C^{\infty}_c$
\begin{equation}\label{inv.eq:2}
\int P^tf(x,\theta)  \mu(dx,d\theta) = \int f(x,\theta)  \mu(dx,d\theta).
\end{equation}
Since, as a simple application of Stone-Weierstrass, $ C^{\infty}_c$ is dense in $C_c$, (\ref{inv.eq:2}) holds for all $f \in C_c$. This further extends to all bounded measurable functions $f$ from Lusin's theorem \cite{folland:13}. This proves the result.
\end{proof}

\begin{remark}
It can be seen from the proof of Lemma \ref{zero.generator:2} that  Assumption \ref{non.evanescence.assumption:1} was only used in order to ensure that the boundary terms appearing in the integration by parts will decay as $\| x \|$ goes to infinity. If the deterministic dynamics are non-explosive, the path of the process until time $t_0$ has a bounded length, therefore the function $g=P^{t_0}f$ has compact support and all the boundary terms disappear as $\| x \| \rightarrow \infty$. This means that when the deterministic dynamics are non-explosive we do not need to make Assumption \ref{non.evanescence.assumption:1}, as long as we still impose Assumption \ref{non.explo.condition:0}.
\end{remark}

\section{Proof of Theorem \ref{geom.ergo.suzz:1} (Exponential Ergodicity)}\label{app.expo.ergo:00}

We first recall some stability notions of a Markov process. For more details see \cite{meyn.tweedie.2:93, meyn.tweedie:09}.

\begin{definition}
A Markov Process $X$ with state space $E$ is $\phi$-irreducible if there exists a non-trivial measure $\phi$ such that for any point $z$ and any set $A$ of positive $\phi$-measure, there exists a $t$ with $P_z\left( X_t \in A \right)>0$. We call $\phi$ an irreducibility measure.

A set $C$ is petite if there exists a probability measure $\nu$, a $c>0$ and a distribution $a$ on $\mathbb{R}_+$ such that for any $z \in C$ and $A \in \mathcal{B}(E)$
\begin{equation}\label{petite.continuous:1}
K_a(z,A)=\int_0^{+\infty}\mathbb{P}_z \left ( X_t \in A \right )a(dt) \geq c \ \nu(A).
\end{equation}

A set $C$ is called small for continuous or a discrete time Markov process/chain $X_t$, if there  exists a probability measure $\nu$, a $c>0$ and $t > 0$ such that for any $z \in C$ and $A \in \mathcal{B}(E)$
\begin{equation}\label{small.continuous:1}
\mathbb{P}_z \left ( X_t \in A \right ) \geq c \ \nu(A).
\end{equation}

Furthermore, the process is called strongly aperiodic if there exists a petite set $C$ and a $T>0$ such that for any $z \in C$ and $t \geq T$, $\mathbb{P}_z \left ( X_t \in C  \right )>0$.

The process is called a $T$-process if there exists a probability density $a$ on $[0,+\infty)$ and a kernel $K:E \times \mathcal{B}(E) \rightarrow [0,+\infty)$ such that for all $A \in \mathcal{B}(E)$ the function $z \rightarrow K(z,A)$ is lower semi-continuous and for all $z \in E$, $K(z,E)>0$ and
\begin{equation*}
\int_0^{+\infty}\mathbb{P}_z\left( X_t \in \cdot \right) a(dt) \geq K(z, \cdot).
\end{equation*}

Additionally, a $\phi$-irreducible process is Harris recurrent if for all $z \in E$ and $A$ such that $\phi(A)>0$ we have $\mathbb{P}_z \left ( \int_0^{+\infty}1_{A}(X_t)=+\infty \right )=1$. If the invariant measure of the process is finite the process is called positive Harris recurrent.

Finally the process with invariant measure $\mu$ is called ergodic if for all $z \in E$
\begin{equation*}
\| \mathbb{P}_z\left( X_t \in \cdot \right) - \mu(\cdot) \|_{TV} \xrightarrow{t \rightarrow +\infty}0.
\end{equation*}
\end{definition}

For the proof of Theorem \ref{geom.ergo.suzz:1} we will use the following result (Theorem 6.1 in \cite{meyn.tweedie.3:93}).
 \begin{theorem}[Meyn-Tweedie 1993]\label{meyn.tweedie:3.6.1}
Assume that a Markov process $(Z_t)_{t \geq 0}$ on $E$ is c\`adl\`ag, and all compact subset of $E$ are petite for some skeleton chain of $Z$. Assume further that if $\mathcal{L}^m$ is the extended generator of the process $(Z^m_t)_{t \geq 0}$, which is the process $Z$, stopped upon exiting $O_m$, then there exists a function $V:E \rightarrow [1,+\infty)$ and $c,b>0$ and a compact set $C$ such that for all $m \in \mathbb{N}$ and $z \in E$,
\begin{equation}\label{dreft.meyn.tweedie.3}
\mathcal{L}^mV(z) \leq -c V(z)+b1_C(z).
\end{equation}
Then there exists a constant $M>0$ and $\rho <1$ such that for all $z \in E$
\begin{equation*}
\| \mathbb{P}_z\left( Z_t \in \cdot \right) - \pi(\cdot) \|_{TV} \leq M V(z) \rho^t,
\end{equation*}
which means that the process $Z$ is exponentially ergodic,
\end{theorem}
From the proof of non-explosivity we have that the function $V$ introduced in (\ref{V.function:0}) satisfies the drift condition (\ref{dreft.meyn.tweedie.3}) for some compact set $C$. Therefore, in order to prove Theorem \ref{geom.ergo.suzz:1} we need to prove that the SUZZ process has all the compact sets as petite for some skeleton chain. The focus of this section is to prove this property.

In order to do this we need to establish the reachability property, introduced in \cite{bierkens.roberts.zitt:2019}.  

Given a speed function $s$, generating the family of deterministic flows $\{ \Phi_t(x,\theta),  t \geq 0\}$ for every $(x,\theta) \in E$, we define as control sequence an object $u=(t,\iota)$, where $t=(t_0,...,t_m)\in (0,+\infty)^{m+1}$, $\iota=(i_1,...,i_m) \in \{ 1,...,n \}^m$ for some $m \in \mathbb{N}$. Starting from $(x,\theta) \in E$, a control sequence $u$ gives rise to a SUZZ trajectory $(X_t,\Theta_t)$ as follows:
Start from $(x,\theta)$ and follow direction $\theta$ for $t_0$ time, i.e. set $X_t=\Phi_t(x,\theta) , \Theta_t=\theta$ for $t \in [0,t_0)$. Then, switch the $i_1$th component of $\theta$ to $F_{i_1}(\theta)$ and follow that direction for $t_1$ time, i.e. set $X_t=\Phi_{t-t_0}\left(\Phi_{t_0}(x,\theta),F_{i_1}(\theta) \right), \Theta_t=F_{i_1}(\theta)$ for $t \in [t_0,t_0+t_1)$. Continue similarly until time $t_0+...+t_m$. Write $\tau_k=\sum_{i=0}^{k-1}t_i$ for the time of the $k$th switch and denote the final position $(X_{\tau_{m+1}},\Theta_{\tau_{m+1}})$ of the path by $\Psi_u(x,\theta)$.
\begin{definition}\label{admissibility:1}
Given a starting point $(x,\theta) \in E$, a control sequence $u=(t,\iota)$ is admissible if for all $k \in \{ 1,...,m \}$ we have $\lambda_{i_k}(X_{\tau_k},\Theta_{\tau_k})>0$.\\
Given two points $(x,\theta),(y,\eta) \in E$ we say that $(y,\eta)$ is reachable from $(x,\theta)$ and write $(x,\theta)\rightarrow(y,\eta)$ if there exists a control sequence $u$ admissible from $(x,\theta)$ such that $\Psi_u(x,\theta)=(y,\eta)$. \\
We write $(x,\theta)\looparrowright(y,\eta)$ if $(x,\theta)\rightarrow(y,\eta)$ and for the admissible sequence $u=(t,\iota)$ connecting the two points, we have that every index of $\{ 1,...,d \}$ appears in $\iota$.
\end{definition}
\noindent We now focus on proving that for any two points $(x,\theta),(y,\eta) \in E$ we have $(x,\theta) \rightarrow (y,\eta)$. Note that we can assume that the SUZZ has minimal rates (i.e. $\gamma_i \equiv 0$ for all $i \in \{ 1,...,d \}$) as higher rates make admissible paths more likely. In the case of original Zig-Zag there is the following result (Theorem 4 of \cite{bierkens.roberts.zitt:2019}).

\begin{theorem}[Bierkens-Roberts-Zitt 2019]\label{normal.zz.reachability}
Assume that $U \in C^3$, $\lim_{\| x \|\rightarrow \infty}U(x)=+\infty$ and there exists an $x_0$ local minimum for $U$ such that $Hess(U)(x_0)$ is strictly positive definite. Then the original Zig-Zag process targeting the potential $U$ satisfies that for all $(x,\theta),(y,\eta)\in E$, $(x,\theta)\looparrowright (y,\eta)$.
\end{theorem}

\noindent We can then generalise these results on the SUZZ using the following Lemma.
\begin{lemma}\label{comparison.original.reachability}
Suppose $s \in C^2$ and $s(x)>0$ for $x \in \mathbb{R}^d$. Then, for any $(x,\theta),(y,\eta) \in E$, $(x,\theta) \rightarrow (y,\eta)$ in a SUZZ with speed $s$, targeting a potential $U$ with minimal rates if and only if $(x,\theta) \rightarrow (y,\eta)$ in an original Zig-Zag, targeting a potential $U-\log s$ with minimal rates.
\end{lemma}
\begin{proof}[Proof of Lemma \ref{comparison.original.reachability}]
Consider an original Zig-Zag process targeting the potential $U-\log s$ with minimal rates. The rates of this process for the $i$ coordinate are $\lambda^0_i(x',\theta')=[\theta'_i\partial_i(U(x')-\log s(x'))]^+$.\\
On the other hand, a SUZZ process with minimal rates targeting the potential $U$ has rates for the $i$ coordinate given by 
\begin{equation*}
\lambda_i(x',\theta')=[\theta'_i(s(x')\partial_iU(x')-\partial_is(x'))]^+=s(x')\lambda^0_i(x',\theta').
\end{equation*}
Therefore for any $(x',\theta') \in E$ and any $i \in \{ 1,...,d \}$ 
\begin{equation}\label{compare.rates.suzz.normalzz:1}
\lambda_i(x',\theta')>0 \iff \lambda^0_i(x',\theta')>0.
\end{equation}
Assume $(x,\theta)\rightarrow (y,\eta)$ with some admissible control sequence $u=(t,\iota)=(t_0,...,t_m,i_1,...,i_m)$ for the original Zig-Zag process, targeting the potential $U-\log s$ with minimal rates. Let $(X_t,\Theta_t)$ be the configuration of that original Zig-Zag path and let $\tau_k=\sum_{i=0}^{k-1}t_i$ be the times of the switches. We have $\lambda^0_{i_k}(X_{\tau_k},\Theta_{\tau_k})>0$ for all $k$.\\
Note that since $s$ is continuous and strictly positive, for any $(x',\theta') \in E$, $\lim_{t \rightarrow +\infty}\| \Phi_t(x',\theta') \|=+\infty$. Therefore, there exists an $s_0>0$ such that $\Phi_{s_0}(x,\theta)=x+t_0\theta=X_{\tau_1}$. Likewise, there exists an $s_1>0$ such that $\Phi_{s_1}\left(X_{\tau_1},\Theta_{\tau_1}\right)=X_{\tau_2}$ and via induction we can construct for all $k \in \{ 0,...,m \}$ an $s_k$ such that $\Phi_{s_k}\left(X_{\tau_k},\Theta_{\tau_k}\right)=X_{\tau_{k+1}}$. Then the control sequence $\tilde{u}=(s,\iota)=(s_0,...,s_m,i_1,...,i_m)$ is an admissible sequence starting from $(x,\theta)$ for the SUZZ targeting the potential $U$ with minimal rates. Furthermore, the ending point of $\tilde{u}$ starting from $(x,\theta)$ is $(y,\eta)$.

The other way around, i.e. that an admissible path for the SUZZ process targeting $U$ implies existence of an admissible path for the ZZ process targeting $U-\log s$ follows using similar arguments. \qedhere
\end{proof}

\noindent Combining Theorem \ref{normal.zz.reachability} and Lemma \ref{comparison.original.reachability} we can prove the following.
\noindent
\begin{proposition}\label{reachability.suzz:1}
Assume that $s \in C^2$ is a strictly positive function such that Assumption \ref{non.explo.condition:0} holds, $U-\log s \in C^3$ and there exists an $x_0 \in \mathbb{R}^d$ such that $U-\log s$ has a local minimum in $x_0$ with $Hess(U-\log s)(x_0)$ being strictly positive definite. Then, for every $(x,\theta),(y,\eta) \in E$, $(x,\theta)\looparrowright(y,\eta)$.
\end{proposition}

\noindent As in \cite{bierkens.roberts.zitt:2019} we can use this to prove that from any starting point and given any other point $z \in E$, the process has a positive probability of visiting a neighbourhood of $z$. The following lemma is the same as Lemma 8 in \cite{bierkens.roberts.zitt:2019}.

\begin{lemma}[Continuous Component]\label{continuous.component.lemma:1}
Assume that $(x,\theta)\looparrowright(y,\eta)$ and the rates $\lambda_i$ are continuous. Then there exist $U_x, V_x \subset \mathbb{R}^d$ open with $x \in U_x , y \in V_x$ and $\epsilon,t_0,c>0$ such that for all $x' \in U_x , t \in [t_0,t_0+\epsilon]$
\begin{equation}\label{continuous.component.lemma:2}
\mathbb{P}_{x',\theta}(X_t\in \cdot , \Theta_t=\eta)\geq c \ Leb(\cdot \cap V_x)
\end{equation}
where $Leb$ is the Lebesgue measure on $\mathbb{R}^d$.
\end{lemma}

 The proof is very similar in spirit to the proof of Lemma 8 (Continuous Component) of \cite{bierkens.roberts.zitt:2019}, therefore we do not present it here but we present it on the Supplementary material B. The main idea is that since there is an admissible path from $(x,\theta) \looparrowright (y,\eta)$, the process has a positive probability to follow some path very close to the admissible path. Therefore there is a positive probability that starting from somewhere close to $(x,\theta)$ the process ends up somewhere close to $(y,\eta)$. 

Lemma \ref{continuous.component.lemma:1} allows us to prove stability properties for the process. 

%We have the following.
\begin{proposition}\label{compacts.petite:1}
Let $(Z_t)_{t \geq 0}=(X_t,\Theta_t)_{t \geq 0}$ be a SUZZ process with strictly positive speed function $s \in C^2$. Assume that the process is non-explosive and have $\mu$ as invariant. Assume that for all $(x,\theta),(y,\eta)\in E$, $(x,\theta)\looparrowright (y,\eta) $. Then, the process is $T$, $\phi$-irreducible and strongly aperiodic. If in addition the invariant measure $\mu$ is a probability measure then all compact sets are petite, they are also small for some skeleton chain and the process is positive Harris recurrent and ergodic.
\end{proposition}

\begin{proof}[Proof of Proposition \ref{compacts.petite:1}]
The fact that the SUZZ process under the assumptions of Proposition \ref{compacts.petite:1} is $T$-process, $\phi$-irreducible and strongly aperiodic can be proven using the same proof as in Theorem 5 of \cite{bierkens.roberts.zitt:2019}.

Now, we prove that every compact set is petite. A standard argument as in \cite{bierkens.roberts.zitt:2019} shows that for $\mu$-almost all starting points $(x,\theta)$ we have $\mathbb{P}_{x,\theta}(\lim_{t \rightarrow +\infty}\| X_t \|=+\infty)=0$. Indeed, for any compact set $K$, $1_{ \{ X_t \text{ eventually leaves }K \}}=\liminf_{t \rightarrow +\infty}1_{X_t \notin K}$ and by Fatou's lemma
\begin{align*}
\mathbb{P}_{\mu}(X_t \text{ eventually leaves } K ) \leq \liminf_{t \rightarrow +\infty}\mathbb{P}_{\mu}(X_t \notin K)=1-\mu(K).
\end{align*}
By exhausting $E$ with compact sets we get $\mathbb{P}_{\mu}(\lim_{t \rightarrow +\infty}\| X_t \|=+\infty)=0$. More specifically, there exists $(x,\theta) \in E$ such that $\mathbb{P}_{x,\theta}(\lim_{t \rightarrow +\infty} \| X_t \|=+\infty )<1$. From Theorem 4.1 in \cite{meyn.tweedie.2:93} all compact sets are petite if and only if the process is $T$ and $\phi$-irreducible. The result follows.

Furthermore, the process is positive Harris recurrent from an application of Theorem 4.4 of \cite{meyn.tweedie.3:93}.

The proof of the fact that some skeleton of the process is irreducible is the same as in Theorem 5 of \cite{bierkens.roberts.zitt:2019}.

From Theorem 6.1 of \cite{meyn.tweedie.2:93} we get that the process is ergodic.

Finally, all compacts are small for some skeleton chain from Proposition 6.1 in \cite{meyn.tweedie.2:93}.
\end{proof}

\begin{proof}[Proof of Theorem \ref{geom.ergo.suzz:1}]
From Theorems \ref{non.explo:1} and \ref{inv.measure.speed.up:1} we know that the process is non-explosive and $\mu$ introduced in (\ref{zz.inv:3}) is invariant for the SUZZ. By Proposition \ref{reachability.suzz:1} for all $(x,\theta), (y,\eta) \in E$ we have $(x,\theta) \looparrowright (y,\eta)$ and therefore, by Proposition \ref{compacts.petite:1} the process is $\phi$-irreducible , aperiodic and all compact sets are small for some skeleton chain. Also, from Lemma \ref{lyapunov:1}, $V$ as in (\ref{V.function:0}) satisfies the drift condition (\ref{lyapunov:2}). All the conditions of Theorem \ref{meyn.tweedie:3.6.1} are satisfied and the result follows.
\end{proof}

\section{Proof of Theorem \ref{geom.ergo.light.tails:1}}
\begin{proof}[Proof of Theorem \ref{geom.ergo.light.tails:1}]
The proof will rely on the following simple observation. Assume that $\mathcal{L}$ is as in (\ref{generator.speed.up:1}) the extended generator of the SUZZ process targeting $\mu$ with refresh rate $\gamma(x)$ and $\mathcal{L}_{ZZ}^{\tilde{U}}$ is the strong generator of the original Zig-Zag process, with refresh rate $\frac{\gamma(x)}{s(x)}$ targeting the measure  $\nu(dx,d\theta)=\frac{1}{\tilde{H}} \exp\{ -\tilde{U}(x) \}$, where $\tilde{U}(x)=U(x)- \log s(x)$, i.e. for all $f \in C^1$,
\begin{equation*}
    \mathcal{L}_{ZZ}^{\tilde{U}}f(x,\theta)=\sum_{i=1}^d\theta_i\partial_i f(x,\theta) + \left( \left[ \theta_i \partial_i\left( U(x) - \log s(x) \right) \right]^++\frac{\gamma(x)}{s(x)}\right) \left( f(x,F_i(\theta)) -f(x, \theta) \right),
\end{equation*}
then
\begin{equation*}
    \mathcal{L}f(x,\theta)=s(x)\mathcal{L}_{ZZ}^{\tilde{U}}f(x,\theta).
\end{equation*}

Let us consider the Zig-Zag process, having generator $\mathcal{L}_{ZZ}^{\tilde{U}}$ and targeting the measure $\nu$. This is a special case of a SUZZ process where the speed function is equal to $1$ everywhere and where the potential function $U$ we have used throughout the document is replaced by $\tilde{U}$. Due to Assumption \ref{light.tail.assumption:1}, the potential $\tilde{U}$ along with the constant speed function, equal to $1$, satisfy all the assumptions of Lemma \ref{lyapunov:1}. Furthermore, since by Assumption \ref{light.tail.assumption:1}, $\gamma(x)\leq \tilde{M} s(x)$, the refresh rate of the Zig-Zag process is bounded. Therefore the function $V$ introduced in (\ref{V.function:0}) satisfies that there exists a compact set $C$ and $c,b>0$ such that
\begin{equation*}
    \mathcal{L}_{ZZ}^{\tilde{U}}V(x,\theta) \leq -c V(x,\theta)+b1_{x \in C}.
\end{equation*}
Therefore,
\begin{equation}\label{again.drift.condition:1}
    \mathcal{L}V(x,\theta) \leq -s(x)c V(x,\theta)+s(x)b1_{x \in C} \leq -c' V(x,\theta)+b'1_{x \in C}
\end{equation}
since $s$ is assumed to be bounded away from $0$ and bounded on compact sets. Using Theorem 2.1 in \cite{meyn.tweedie.3:93}, in the same way as in the conclusion of the proof of Theorem \ref{non.explo:1}, we get that the process is non-explosive. 

For the other three bullet points, given that we have found a function $V$ satisfying the drift condition (\ref{again.drift.condition:1}), the proofs of Theorems \ref{inv.measure.speed.up:1} and \ref{geom.ergo.suzz:1} carry over here and we get the CLT result as a consequence of Theorem 2 of \cite{chan.geyer:94}.
\end{proof}

\section{Proof of Proposition \ref{special.case.theorem:1}}\label{proof.of.proposition.special.case:00}

\begin{proof}[Proof of Proposition \ref{special.case.theorem:1}]
We will consider the case where 
\begin{equation*}
    s(x)=1+\| x \|_2^2.
\end{equation*} 
The case $s(x)=\sqrt{1+\| x \|_2^2}$ follows using a similar argument.
We begin by noting that we can write
\begin{equation}\label{cx:1}
    A_i(x)=s(x)\partial_iU(x)-\partial_is(x)=c(x)x_i
\end{equation}
where if $\pi$ as in (\ref{pi.subexponential:1}) then
\begin{equation*}
    c(x)=a\left( 1+ \| x \|_2^2  \right)^{\frac{a}{2}}-2
\end{equation*}
and if $\pi$ as in (\ref{pi.student:1}) then
\begin{equation*}
    c(x)=\left(\nu + d  \right) \frac{1+ \| x  \|_2^2}{\nu + \| x  \|_2^2}-2.
\end{equation*}
Given that $a>0$ or $\nu$ satisfies (\ref{nu.lower.bound:1}), we see that for both targets, there exists a $K>0$ such that for all $\| x \|_2 \geq K$
\begin{equation}\label{c.lower.bound:1}
    c(x) \geq c > \frac{27}{2}%\left( 1+ \left( 1+4\sqrt{2}\frac{1}{d} \right)^{\frac{1}{2}}  \right)^2
    d^3.
\end{equation}

We set $b=\left(\frac{3}{2}\right)^{1/\pi}$ and we consider the function 
\begin{equation}
    V(x,\theta)=\sum_{i=1}^d \left( 1+x_i^2 \right) b^{\arctan \left(\theta_i x_i \right)}.
\end{equation}
 First of all, $V \in C^1$ therefore $V \in D(\mathcal{L})$. Furthermore, $\lim_{\| x \| \rightarrow \infty}V(x,\theta)=+ \infty$ and $V(x,\theta) \geq d b^{-\pi/2}$>0. Let $\| x \|_2 \geq K$. We write
\begin{align*}
    \mathcal{L}V(x,\theta)=&\sum_{i=1}^d \underbrace{2 \theta_i x_i b^{\arctan \left( \theta_ix_i \right)} \left( 1+\| x \|_2^2 \right)}_\text{term $a_i$}+\underbrace{ \left( 1+ \| x \|_2^2  \right) \log(b) b^{\arctan\left( \theta_ix_i \right)}}_\text{term $b_i$}\\
    &+\sum_{i=1}^d\underbrace{c(x)\left[ \theta_i x_i \right]^+ \left( 1+x_i^2 \right) b^{\arctan \left(  \theta_i x_i \right)} \left(b^{-2\arctan \left( \theta_ix_i \right)}-1  \right)}_\text{term $c_i$}.
\end{align*}

Our goal is to show that $V$ satisfies the drift condition (\ref{lyapunov:2}). Let $\eta>0$ be small enough (to be determined later), and having fixed $\eta$, let $\epsilon>0$ be small enough (to be determined later).
 For any $\delta>0$, consider the set $E_{\delta}=E_{\delta}(x)=\left\{ i \in \{ 1,...,d \} : \frac{|x_i|}{\| x \|_2} \geq \frac{\delta}{\sqrt{d}}\right\}$. Set 
 \begin{equation}\label{delta.definition:1}
     \delta=\left(\frac{ (2+\eta)d}{c \left(  1- b^{-\pi} \right)} \right)^{1/2}=\left( 3\frac{(2+\eta)d}{c}  \right)^{1/2},
     \end{equation}
     where $c$ as in (\ref{c.lower.bound:1}). We first note that for any $i \in \{ 1,...,d \}$, if $\theta_i x_i\leq 0$ then $c_i=0$ and $a_i \leq 0$ so $a_i+c_i \leq 0$. 
 
 If $\theta_i x_i >0$ and $i \in E_{\delta}$, then for $\epsilon>0$ small enough and assuming that $K$ is large enough (given $\epsilon$)
 , for any $\| x \|_2 \geq K$,
\begin{align*}
    &a_i+c_i \leq |x_i| 2b^{\frac{\pi}{2}}\left(1+\epsilon \right)\| x \|_2^2-c(x)|x_i||x_i|^2b^{\frac{\pi}{2}-\epsilon} \left(  1- b^{-\pi+2\epsilon} \right) \\
    &\leq |x_i|b^{\frac{\pi}{2}}\left(2 \left(1+\epsilon \right)\| x  \|_2^2-c(x)b^{-\epsilon}\left(  1- b^{-\pi+2\epsilon} \right)\frac{\| x  \|_2^2}{d} \delta^2 \right) \\
    &\leq |x_i|b^{\frac{\pi}{2}}  \| x  \|_2^2\left( 2\left(1+\epsilon \right)-c b^{-\epsilon}\left(  1- b^{-\pi+2\epsilon} \right)\frac{\delta^2}{d}  \right)<0,
\end{align*}
from the definition of $\delta$. Here we have used that if $i \in E_{\delta}$, then $|x_i| \geq \| x \|_2\frac{\delta}{\sqrt{d}}$.

On the other hand, if $\theta_ix_i>0$ and $i \notin E_d$ then $c_i \leq 0$ so
\begin{equation*}
    a_i+c_i\leq a_i \leq |x_i|2b^{\frac{\pi}{2}}(1+\epsilon)\| x \|_2^2 \leq 2(1+\epsilon)b^{\frac{\pi}{2}}\frac{\delta}{\sqrt{d}}\| x \|_2^3,
\end{equation*}
where we have used the fact that $|x_i|\leq \| x \|_2\frac{\delta}{\sqrt{d}}$ when $i \notin E_{\delta}$.

Now consider $j=\argmax\{ |x_i|, i=1,...,d \}$. Combining all the results above, we then get
\begin{equation}\label{aibi:1}
    \sum_{i\neq j}a_i+c_i \leq \sum_{i \neq j, \\ i \notin E_{\delta}, \\ \theta_ix_i>0}a_i+c_i \leq (d-1)2(1+\epsilon)b^{\frac{\pi}{2}}\frac{\delta}{\sqrt{d}}\| x \|_2^3.
\end{equation}

Now let us consider the quantity $a_j+c_j$. Since $|x_j|=\max\{ |x_i|, i=1,...,d \}$  we have $|x_j|\geq \frac{\| x \|_2}{\sqrt{\delta}}$, i.e. $j \in E_1$. We distinguish between the following cases.

{\bf Case 1: }$\theta_jx_j>0$. In this case, from the previous calculations
\begin{align*}
   & a_j+c_j \leq |x_j|\| x \|_2^2b^{\frac{\pi}{2}}\left( 2(1+\epsilon)-cb^{-\epsilon}\left(  1- b^{-\pi+2\epsilon} \right)\frac{1}{d}  \right)\\
    &\leq -\| x \|_2^3\frac{1}{\sqrt{d}} b^{\frac{\pi}{2}}\left( cb^{-\epsilon}\left(  1- b^{-\pi+2\epsilon} \right)\frac{1}{d} - 2(1+\epsilon) \right).
\end{align*}

{\bf Case 2: } $\theta_jx_j \leq 0$. In this case, $c_j=0$ so
\begin{equation*}
    a_j+c_j = -2|x_j|b^{-\arctan(|x_j|)}(1+\|  x \|_2^2) \leq -2 \| x \|_2^3\frac{1}{\sqrt{d}}b^{-\frac{\pi}{2}},
\end{equation*}
where we have used that $x_j \geq \frac{\|  x \|_2}{\sqrt{d}}$. In any case
\begin{equation}\label{ajcj:1}
    a_j+c_j\leq -\| x  \|_2^3 b^{\frac{\pi}{2}}\frac{1}{\sqrt{d}} \min\left\{  cb^{-\epsilon}\left(  1- b^{-\pi+2\epsilon} \right)\frac{1}{d} - 2(1+\epsilon) , 2   b^{-\pi}\right\}.
\end{equation}
Overall, combining (\ref{aibi:1}) and (\ref{ajcj:1}) we have
\begin{equation}\label{aibi:2}
    \sum_{i=1}^da_i+c_i \leq \| x \|_2^3 \frac{b^{\frac{\pi}{2}}}{\sqrt{d}}\left[ (d-1)2(1+\epsilon)\delta- \min\left\{  cb^{-\epsilon}\left(  1- b^{-\pi+2\epsilon} \right)\frac{1}{d} - 2(1+\epsilon) , 2   b^{-\pi}\right\} \right].
\end{equation}
where $\delta$ as in (\ref{delta.definition:1}) for a small value of $\eta$. Our goal is to show that the term in the square bracket on the RHS is negative for small values of $\epsilon$ and $\eta$. To do this, it suffices to show that
\begin{equation}\label{two.inequalities:1}
    (d-1)2\delta<c\left( 1-b^{-\pi} \right)\frac{1}{d}-2
\end{equation}
and that
\begin{equation}\label{two.inequalities:2}
    (d-1)2\delta<2b^{-\pi},
\end{equation}
where $\delta$ as in (\ref{delta.definition:1}) for $\eta$ small enough. We will first establish (\ref{two.inequalities:1}). To do this, it suffices to prove that
\begin{align*}
    (d-1)2\left(  \frac{2d}{c \left( 1-b^{-\pi} \right) } \right)^{1/2} <c\left( 1-b^{-\pi} \right) \frac{1}{d}-2   \iff 2^{3/2}(d-1)d^{1/2}c^{-1/2}3^{1/2}+2<\frac{c}{3d}
\end{align*}
since $b^{\pi}=3/2$. To prove the last inequality, it suffices to prove that 
\begin{equation*}
    2<\frac{1}{2}\frac{c}{3d} 
\end{equation*}
and that
\begin{equation*}
    2^{3/2}d^{3/2}c^{-1/2}3^{1/2}<\frac{1}{2}\frac{c}{3d}.
\end{equation*}
The first equation is equivalent to asking that $c>12d$ which holds due to (\ref{c.lower.bound:1}). The second equation is equivalent to asking that $c>2^{5/3}3d^{5/3}$ which also holds due to (\ref{c.lower.bound:1}). This verifies (\ref{two.inequalities:1}). In order to prove (\ref{two.inequalities:2}) for small enough $\eta$, it suffices to establish that
\begin{equation*}
    d\left( \frac{2d}{c(1-b^{-\pi})} \right)^{1/2}<b^{-\pi} \iff 2^{1/2}d^{3/2}c^{-1/2}3^{1/2}<\frac{2}{3} \iff c>\frac{27}{2}d^3,
\end{equation*}
which holds from (\ref{c.lower.bound:1}). This confirms that the term in the square bracket on the RHS of (\ref{aibi:2}) is negative for small values of $\epsilon$ and $\eta$. Overall, from (\ref{aibi:2}) we get that there exist $K,k>0$ and such that for any $\| x  \|_2 \geq K$,
\begin{equation}
    \sum_{i=1}^da_i+c_i \leq -k \| x \|_2^3.
\end{equation}
Furthermore, one can bound for any $\| x \|_2 \geq K$,
\begin{equation*}
    \sum_{i=1}^db_i\leq 2d\log(b)b^{\pi/2}\| x \|_2^2,
\end{equation*}
therefore, by increasing $K$ appropriately, for any $\| x \|_2 \geq K$,
\begin{equation*}
    \mathcal{L}V(x,\theta)=\sum_{i=1}^da_i+b_i+c_i \leq -\frac{k}{2}\| x \|_2^3.
\end{equation*}
Finally, one can bound $V$ for all $\|  x \|_2 \geq K$ by
\begin{equation*}
    V(x,\theta) \leq b^{\pi/2}\| x \|_2^2+ b^{\pi/2}d,
\end{equation*}
so overall, by further appropriately increasing $K$, we get that for all $\| x  \|_2 \geq K$,
\begin{equation}\label{drift.specific.s:1003}
    \mathcal{L}V(x,\theta) \leq -\frac{k}{4}\| x \|_2 V(x,\theta).
\end{equation}
while on $\{ x: \| x \|_2 < K \}$, $\mathcal{L}V$ is bounded. This proves that $V$ satisfies the drift condition (\ref{lyapunov:2}).  The rest of the proof can be concluded using the same arguments as in Appendices \ref{app.non.explo:00},  \ref{app.invariant:00} and \ref{app.expo.ergo:00}.

The case where $s(x)=\left( 1+\| x  \|_2^2 \right)^{1/2}$ can be treated in a similar way.

Finally, one can prove Remark \ref{specific.remark:1} by considering the function \begin{equation*}
    V(x,\theta)=\sum_{i=1}^d \left( 1+\left(Bx\right)_i^2 \right) b^{\arctan \left(\theta_i \left(Bx\right)_i \right)},
\end{equation*}
where $B$ as in the Remark \ref{specific.remark:1}, and prove that $V$ satisfies the drift condition (\ref{lyapunov:2}) using similar arguments.
\end{proof}

\section{Proof of results in Section \ref{relationship.suzz.zz:00}}

\begin{proof}[Proof of Proposition \ref{space.transformation:1982}]\label{proof.of.space.uniform:1}
Introduce the space transformation $f(x)=\int_0^x1/s(u)du$ and let $M^+=\lim_{x \rightarrow +\infty}f(x) \in (0,+\infty]$ and $-M^-=\lim_{x \rightarrow -\infty}f(x) \in [-\infty, 0)$. 
Consider the process $(Y_t,\Theta_t)_{t \geq 0}$, where $Y_t=f(X_t)$, defined on $(-M^-,M^+) \times \{ -1,+1 \}$. When the process starts from $(y,\theta)=(f(x),\theta)$, it follows deterministic dynamics 
\begin{equation*}
\dfrac{dY_t}{dt}=f'(X_t)\dfrac{dX_t}{dt}=\dfrac{1}{s(X_t)}\Theta_ts(X_t)=\Theta_t,
\end{equation*}
and 
\begin{equation*}
\dfrac{d\Theta_t}{dt}=0.
\end{equation*}
 Furthermore, the random time $T$ when the sign of $\Theta_t$ changes, is the same time when the direction changes for the $(X_t,\Theta_t)_{t \geq 0}$ process. Therefore,
\begin{align*}
\mathbb{P}(T\geq t)=\exp\left\{ - \int_0^t \lambda(X_u,\theta)du \right\}=\exp \left\{ -\int_0^t\lambda(f^{-1}(Y_u),\theta) du \right\}.
\end{align*}
This proves that $(Y_t,\Theta_t)_{t \geq 0}$ is a one-dimensional original Zig-Zag with unit speed function and intensity rates $\lambda_Y$ given by
\begin{align*}
&\lambda_Y(y,\theta)=\lambda(f^{-1}(y),\theta)=\left[\theta\left( s(f^{-1}(y)) U'(f^{-1}(y))-s'(f^{-1}(y)) \right)\right]^+ \\
&=\left[ \theta \left( \left( U(f^{-1}(y)) -\log s(f^{-1}(y)) \right)'  \right)\right]^+=\left[\theta \tilde{U}'(y)\right]^+
\end{align*}
where $\tilde{U}(y)=U(f^{-1}(y))-\log s(f^{-1}(y))$. The result follows from standard results for the original Zig-Zag process (see \cite{bierkens.roberts_superefficient:2019}).
\end{proof}

\begin{proof}[Proof of Theorem \ref{uniform.ergodicity:1}]

From Proposition \ref{space.transformation:1982} we know that  if $f(x)=\int_0^x1/s(u)du$, then the process $(Y_t,\Theta_t)_{t \geq 0}=(f(X_t),\Theta_t)_{t \geq 0}$ is an original Zig-Zag. From a separation of variables technique (see Appendix \ref{ode.solution:00}), the solution of the ODE
\begin{equation}\label{ode.one.dimension:1}
\left\{ \begin{array}{l}
\dfrac{dX_t}{dt}=\theta s(X_t) \vspace{1.5mm}\\
X_0=x
\end{array} \right.
\end{equation}
satisfies
\begin{equation}\label{representation.of.ode.suzz:1}
X_t=f^{-1}(f(x)+\theta t).
\end{equation} 
Let $\lim_{x \rightarrow +\infty}f(x)=M^+ \in (0,+\infty]$ and $\lim_{x \rightarrow -\infty}f(x)=-M^- \in [-\infty , 0)$.
Then, the ODE has finite explosion time when starting from some $(x,+1)$ if and only if $M^+<\infty$, whereas it has finite explosion time when starting from some $(x,-1)$ if and only if $M^- < \infty$.
This further means that if the deterministic dynamics starting from some $(x,\theta)$ explode in finite time, then the deterministic dynamics starting from any other $(x',\theta)$ explode in finite time too. We also note that the 
Zig-Zag process $(Y_t,\Theta_t)_{t \geq 0}$ is defined on the space $(-M^-,M^+)\times \{ \pm 1 \}$.

We will begin by proving that the process $(X_t,\Theta_t)$ is non-explosive.

Let us first consider the case where the deterministic flow explodes when starting both from $(x,+1)$ and $(x,-1)$. In that case the Zig-Zag process $(Y_t,\Theta_t)$ is defined on the bounded interval $(-M^-,M^+)$.
 Due to Assumption \ref{non.explo:1} and due to Proposition \ref{switch.before.boom:1} we see that the process a.s.\ cannot start from some point $(x',\theta)$ and reach infinity without the occurrence of a direction switch. This means that if $\xi$ and $\zeta$ are as in (\ref{definition.of.xi:00}) and (\ref{definition.of.zeta:00}), then $\xi \leq \zeta$. Therefore, the only way for the process to explode is for $\xi< \infty$. Let us consider the event that $\xi < \infty$ in order to show that it has probability zero. Let $T_1<T_2<...$ be the switching times and let $\lim_{n \rightarrow \infty}T_n=\xi<\infty$. The $T_n$'s are also the switching times of the Zig-Zag $(Y_t,\Theta_t)$ process, which is defined on the bounded interval $(-M^-,M^+)$. Let $Y_{T_1},Y_{T_2},...$ be the switching points of the Zig-Zag process. Since these points lie in a bounded interval, there exists a subsequence of these points that converges to some point $z \in [-M^-,M^+]$. Assume for contradiction that there exists a second accumulation point $z' \neq z$ of the elements of $\{ Y_{T_1},Y_{T_2},... \}$. Since the Zig-Zag process moves with unit speed, it takes at least $\frac{|z-z'|}{2}$ time to get from a small neighbourhood of one accumulation point to a neighbourhood of the other. On the other hand, we have assumed that the switching times $T_n$ converge to a finite time $\xi$, therefore the difference between successive switching times must converge to zero and there is not enough time for the process to travel from one accumulation neighbourhood to the other. Therefore, there exists a unique accumulation point $z$, meaning that $\lim_{n \rightarrow \infty}Y_{T_n}=z$.

We distinguish between two cases for $z$. Assume first that $z \in (-M^-,M^+)$. Since $\lambda \in C^0$, it is bounded in a neighbourhood of $z$. Using the same argument as in Proof of Lemma \ref{infinite.switches.lemma:1} we see that the probability that there are infinitely many switches in finite time in a small neighbourhood of $z$ is zero. 
On the other hand, let us consider the case where $z \in \{ -M^-,M^+ \}$ and let us assume without loss of generality that $z=M^+$. Combining Assumptions \ref{non.explo.condition:0} and \ref{ass.lyapunov:1}, we have $\liminf_{x \rightarrow   +\infty} \left( U(x)-\log s(x) \right)' \geq 0$. This means that for large values of $x$,  
\begin{equation*}
    \lambda(x,-1)=[s(x) U'(x) - s'(x)]^-+\gamma(x)=s(x)[\left( U(x)-\log s(x) \right)']^-+\gamma(x)=\gamma(x)\leq \bar{\gamma}.
\end{equation*}
Therefore, there exists an $\epsilon>0$ such that if $y \in (M^+-\epsilon,M^+)$ then
\begin{equation*}
    \lambda_Y(y,-1)=\lambda(f^{-1}(y),-1)\leq \bar{\gamma}.
\end{equation*}
 On the event that $M^+$ is the accumulation point of switching times there are two possibilities. Either the $Y$-process switches from $-1$ to $+1$ infinitely many times in finite time inside the interval $(M^+-\epsilon, M^+)$ or it does not. Since the process has a rate $\lambda_Y(y,-1)$ that is bounded above in $(M^--\epsilon,M^+)$ the first event has probability zero. The other event is that only finitely many switches from $-1$ to $+1$ occur inside $(M^+-\epsilon,M^+)$ until time $\xi$. In that case, since $M^+$ is an accumulation point for the switching times of $Y$, there are infinitely many switches from $+1$ to $-1$ inside $(M^+-\frac{\epsilon}{2},M^+)$ in finite time. Each of these switches has to be followed by a switch from $-1$ to $+1$. Since only finitely many will occur in $(M^+-\epsilon,M^+)$, this means that in finite time there will be infinitely many switches from $+1$ to $-1$ inside $(M^+-\frac{\epsilon}{2},M^+)$ that have a successor switch from $-1$ to $+1$ outside $(M^+-\epsilon,M^+)$. This means that between any of these two successive switches the process has traveled from $M^+-\frac{\epsilon}{2}$ to $M^+-\epsilon$ without switching and  since the process moves with unit speed this takes at least $\frac{\epsilon}{2}$ time. Therefore, there cannot be infinitely many switches of this form in finite time. Therefore, we reach a contradiction assuming the second possibility. Overall this proves that the process is not explosive.

The case where the flow starting from $(x,+1)$ explodes in finite time but the one starting from $(x,-1)$ does not (and vice versa) can be handled in a similar way. In the case where the flow does not explode either in direction $+1$ or $-1$, the SUZZ algorithm is by definition non-explosive. This finishes the proof of non-explosivity.

We then observe that given non-explosivity, for the proof of Theorem \ref{inv.measure.speed.up:1}, only Assumption \ref{non.evanescence.assumption:1} (which is equivalent to Assumption \ref{non.explo.condition:0} in dimension one) is needed. We therefore conclude that the process has $\mu$ as in (\ref{zz.inv:3}) as invariant. 

Let us now assume that the flow starting from any $(x,\theta)$ explodes in finite time. Following the same argument as in the proof of Lemma 15 of \cite{bierkens.roberts_scaling:2017} for the one-dimensional original Zig-Zag, we get that the entire state space $(-M^-,M^+) \times \{ -1,+1 \}$ is petite for the transformed Zig-Zag process.
Therefore, the transformed Zig-Zag is uniformly ergodic and we get 
that there exists a constant $M>0$ and $\rho<1$ and a measure $\nu$ such that for all $f(x) \in (-M^-,M^+)$ and any $\theta \in \{ \pm 1 \}$, 
\begin{equation*}
\| \mathbb{P}_{f(x),\theta}( \left( f(X_t),\Theta_t \right) \in \cdot)-\nu(\cdot) \|_{TV}\leq M \rho^t .
\end{equation*}
Since $f$ is $1-1$ we get that for all $x \in \mathbb{R}$ , $\theta \in \{ \pm 1 \}$,
 \begin{equation*}
\| \mathbb{P}_{x,\theta}(Z_t \in \cdot)-\nu(f^{-1}(\cdot), \cdot) \|_{TV}\leq M \rho^t.
\end{equation*}
Since $\mu$ is invariant for the process $Z_t$,
we get $\nu(f^{-1}(\cdot), \cdot)=\mu(\cdot)$. This means that for any $(x,\theta) \in E$,
 \begin{equation*}
\| \mathbb{P}_{x,\theta}(Z_t \in \cdot)-\mu(\cdot) \|_{TV} \leq M \rho^t,
\end{equation*}
which proves the uniform ergodicity result.

Let us now assume that both the flows starting from $(x,+1)$ and $(x,-1)$ do not explode so the Zig-Zag $Y$-process is defined on $\mathbb{R} \times \{ -1,+1 \}$. Combining Assumptions \ref{non.explo.condition:0} and \ref{ass.lyapunov:1}, as previously done in this proof we get that
\begin{equation*}
    \liminf_{x \rightarrow +\infty}\lambda(x,+1)=\liminf_{x \rightarrow +\infty}|s(x)U'(x)-s'(x)|+\gamma(x) \geq A +\gamma(x)
\end{equation*}
where $A$ as in (\ref{non.decay.of.rates:1}). At the same time 
\begin{equation*}
    \limsup_{x \rightarrow +\infty}\lambda(x,-1)=\gamma(x).
\end{equation*}
Therefore
\begin{equation*}
    \liminf_{x \rightarrow +\infty}\lambda(x,+1)>\limsup_{x \rightarrow +\infty}\lambda(x,-1),
\end{equation*}
and since $\lambda_Y(y,\theta)=\lambda(f^{-1}(y),\theta)$ we get
\begin{equation*}
    \liminf_{y \rightarrow +\infty}\lambda_Y(y,+1)>\limsup_{y \rightarrow +\infty}\lambda_Y(y,-1).
\end{equation*}
Using the same argument when $x \rightarrow -\infty$ we get that
\begin{equation*}
    \liminf_{y \rightarrow -\infty}\lambda_Y(y,-1)>\limsup_{y \rightarrow -\infty}\lambda_Y(y,+1).
\end{equation*}
 From Lemma 16 of \cite{bierkens.roberts_scaling:2017} the last two inequalities imply that the Zig-Zag process $Y$ is exponentially ergodic. Using the same argument as when we concluded that the SUZZ process with explosive dynamics is uniformly ergodic, we conclude now that the SUZZ process with non-explosive dynamics is exponentially ergodic.

The case where the flow $(x,+1)$ explodes and the one from $(x,-1)$ does not (and vice versa) can be treated using a similar argument.
\end{proof}

\section{Proof of Proposition \ref{efficiency.proposition:1}}
\begin{proof}[Proof of Proposition \ref{efficiency.proposition:1}]
The existence of $\phi$ such that $\phi(x,\theta) \leq C V(x,\theta)$ and such that $\mathcal{L}\phi(x,\theta)=-g(x,\theta)$ is guaranteed by Assumption \ref{the.not.so.good.assumption:1}.
Consider
\begin{equation*}
M_t=\phi(Z_t)-\phi(Z_0)+\int_0^t g(Z_s)ds,
\end{equation*}
which is a local martingale by Dynkin's formula. Assume that $Z$ starts from the invariant measure $\mu$. Under $\mathbb{P}_{\mu}$, $M_t$ is also a martingale since for any $t>0$ and any $s \leq t$, we have
\begin{equation*}
\mathbb{E}_{\mu}[|M_s|] \leq \mathbb{E}_{\mu}[|\phi(Z_s)|]+\mathbb{E}_{\mu}[|\phi(Z_0)|]+\int_0^s\mathbb{E}_{\mu}\left[\left|g(Z_s)\right|\right]ds \leq 2 \mathbb{E}_{\mu}[|\phi|]+t\mathbb{E}_{\mu}[|g|]<\infty.
\end{equation*}
Furthermore $M_t$ has stationary increments. From Theorem 2.1 of \cite{komorowski.landim.olla:12} we have
\begin{equation*}
\frac{M_t}{\sqrt{t}}\xrightarrow{t \rightarrow \infty}N(0,\mathbb{E}[M_1^2])
\end{equation*}
in distribution under $\mathbb{P}_{\mu}$. Also, under $\mathbb{P}_{\mu}$ 
\begin{equation*}
\frac{\phi(Z_t)-\phi(Z_0)}{\sqrt{t}}\xrightarrow{t \rightarrow +\infty}0,
\end{equation*}
since $\phi(Z_t)$ has the same law as $\phi(Z_0)$ and $\mathbb{E}_{\mu}[\phi[Z_t]]=\mathbb{E}_{\mu}[\phi[Z_0]]<\infty$. Therefore
\begin{equation*}
\frac{1}{\sqrt{T}}\int_0^Tg(Z_s)ds \xrightarrow{T \rightarrow \infty}N(0,\mathbb{E}_{Z_0 \sim \mu}[M_1^2])
\end{equation*}
under $\mathbb{P}_{\mu}$. It suffices to prove that $\mathbb{E}_{Z_0 \sim \mu}[M_1^2]$ admits the expression in (\ref{efficiency.as.var:1}). Let $K_t$ be the number of switches before time $t$ and let $T_1,T_2,...$ be the times of the switches. We write
\begin{align*}
&M_t=\phi(Z_t)-\phi(Z_0)-\int_0^t\mathcal{L}\phi(Z_s)ds= \\
&=\int_0^t\Theta_s s(X_s)\phi'(Z_s)ds+\sum_{i=1}^{K_t}\phi(Z(T_i))-\phi(Z(T_i^-)) \\
&-\int_0^t\Theta_s s(X_s)\phi'(Z_s)+\lambda(Z_s)(\phi(X_s,-\Theta_s)-\phi(X_s,\Theta_s)ds=\\
&=\sum_{i=1}^{K_t}\phi(X_{T_i},\Theta_{T_i})-\phi(X_{T_i},-\Theta_{T_i})+\int_0^t\lambda(Z_s)(\phi(X_s,\Theta_s)-\phi(X_s,-\Theta_s))ds
\end{align*}
and therefore (see \cite{kallenberg:02}, Theorem 23.6) the predictable quadratic variation of $M$ is 
\begin{equation*}
\langle  M \rangle_t=\int_0^t\lambda(Z_t) \left ( \phi(X_s,\Theta_s)-\phi(X_s,-\Theta_s) \right )^2ds =4\int_0^t\lambda(Z_t) \left ( \psi(X_s) \right )^2 \ ds
\end{equation*}
where $\psi(x)=\frac{1}{2}\left (\phi(x,+1)-\phi(x,-1) \right )$. Therefore, from the definition of predictable quadratic variation, under $\mathbb{P}_{\mu}$,
\begin{equation}\label{efficiency.as.var:2}
\gamma_g^2=\mathbb{E}_{Z_0 \sim \mu}[M_1^2]=\mathbb{E}_{Z_0 \sim \mu}\langle  M \rangle_1=4\int_E \lambda(x,\theta) \psi^2(x)d\mu(x,\theta).
\end{equation}
It remains to write $\psi$ in terms of $g$. This can be done since for all $(x,\theta) \in E$, $\mathcal{L}\phi(x,\theta)=-g(x,\theta)$ and therefore
\begin{equation*}
\theta s(x) \phi'(x,\theta)+\lambda(x,\theta)(\phi(x,-\theta)-\phi(x,\theta))=-g(x,\theta).
\end{equation*}
Writing down the two equations for $\theta=\pm 1$ and adding them up we get
\begin{equation*}
s(x)(\phi'(x,+1)-\phi'(x,-1))-(\lambda(x,+1)-\lambda(x,-1))(\phi(x,+1)-\phi(x,-1))=-(g(x,+1)+g(x,-1))
\end{equation*}
and therefore
\begin{equation*}
s(x)\psi'(x)-(s(x)U'(x)-s'(x))\psi(x)=-\frac{g(x,+1)+g(x,-1)}{2}.
\end{equation*}
Solving this first order linear ODE we get
\begin{equation*}
\psi(x)=\frac{1}{2\exp \{ -U(x) \} s(x) }\int_x^{+\infty}(g(y,+1)+g(y,-1))\exp\{ -U(y) \}dy,
\end{equation*}
which when combined with (\ref{efficiency.as.var:2}) gives (\ref{efficiency.as.var:1}).
\end{proof}

\section{Verification of Assumptions for one-dimensional simulated targets}\label{verification:000}
Here we will check that using the speed function (\ref{one.dimensional.speed:1}) leads to exponentially or uniformly ergodic algorithms when targeting some of the one-dimensional distributions in Section \ref{efficiency:00}, where we calculate the inverse algorithmic efficiency, and in Section \ref{simulations:00}, where we present numerical simulations. In all cases we used SUZZ algorithms with refresh rate $\gamma(x)=0$, therefore, the upper bound constant $\bar{\gamma}$ appearing in Assumption \ref{ass.lyapunov:1} is $\bar{\gamma}=0$. 

\begin{itemize}
    \item $d=1$, $\pi(x)=\frac{1}{H}\left( 1+\frac{1}{3}x^2 \right)^{-2}$, with speed $s(x)=\left( 1+x^2 \right)^{\frac{1+k}{2}}, k=0,1,2$. This Student($3$) target is used in Sections \ref{efficiency:00} and \ref{simulations:00}.   
     Assumption \ref{non.explo.condition:0} is easily verified since 
    \begin{equation*}
        \lim_{|x| \rightarrow \infty}\pi(x)s(x)= 0, \text{ for } k=0,1,2.
    \end{equation*}
    Furthermore,
    here $U(x)=-\log \pi(x)=2 \log(1+\frac{1}{3}x^2)+\log H$ and $U'(x)=4\frac{x}{3+x^2}$. Straightforward calculations show that for $|x|>1$ large enough such that $\frac{1+x^2}{3+x^2}>1-\frac{1}{10}$,
    \begin{align*}
        |A(x)|&=|s(x)U'(x)-s'(x)|=|x|(1+x^2)^{\frac{k-1}{2}}\left( 4\frac{1+x^2}{3+x^2}-1-k \right)\\
        &\geq |x|(1+x^2)^{\frac{k-1}{2}}\left( 4-4\frac{1}{10}-1-k \right)\geq \frac{6}{10\sqrt{2}}>0
    \end{align*}
    since $k \in [0,2]$ and $|x|\left( 1+x^2 \right)^{-1/2}\geq \frac{1}{\sqrt{2}}$. Therefore Assumption \ref{ass.lyapunov:1} is satisfied for $k=0,1,2$, with $A=\frac{6}{10\sqrt{2}}$. 
    
Therefore, the assumptions of Theorem \ref{uniform.ergodicity:1} are verified. Using similar calculations, we conclude that the assumptions are verified for any Student target with $\nu$ degrees and speed function of the form $s(x)=\left( 1+x^2 \right)^{\frac{1+k}{2}}$ as long as $k<\nu$. The SUZZ$(0)$ algorithm is then exponentially ergodic and the SUZZ($k$) algorithm with $k>0$ is uniformly ergodic, due to Theorem \ref{uniform.ergodicity:1}.

\item $d=1$, $\pi(x)=\frac{1}{H}\exp\{ - \left( 1+x^2 \right)^{1/4} \}$, with speed $s(x)=\left(  1+x^2 \right)^{\frac{1+k}{2}}, k=0, 1, 2, 3$. This sub-exponential($0.5$) target is used in Section \ref{efficiency:00}. Assumption \ref{non.explo.condition:0} is easily verified since
\begin{equation*}
    \lim_{|x| \rightarrow \infty}\pi(x)s(x)=0, \text{ for } k=0,1,2,3.
\end{equation*}
Furthermore, here $U(x)=-\log \pi(x)=\left(  1+x^2 \right)^{1/4}+\log H$ and $U'(x)=\frac{1}{2}x\left(  1+x^2 \right)^{-3/4}$. Therefore,
\begin{align*}
    |A(x)|=|s(x)U'(x)-s'(x)|=|x| \left(  1+x^2 \right)^{\frac{k-1}{2}} \left( \frac{1}{2}\left( 1+x^2 \right)^{\frac{1}{4}}-\left( 1+k \right)   \right) \xrightarrow{|x| \rightarrow \infty}+\infty,
\end{align*}
which verifies Assumption \ref{ass.lyapunov:1}. Therefore, the assumptions of Theorem \ref{uniform.ergodicity:1} are verified. The SUZZ(0) algorithm is then exponentially ergodic and the SUZZ($k$) algorithm with $k=1, 2, 3$ is uniformly ergodic.
\item The $d=1$ Exponential and Normal targets appearing in Section \ref{efficiency:00} can be shown to verify the assumptions of Theorem \ref{uniform.ergodicity:1} in  similar way.
\end{itemize}

\end{appendix}

\section*{Supplement A: Algorithmic description of Speed Up Zig-Zag process}\label{algorithm:00}

In algorithmic terms the SUZZ process is described as follows.
\begin{algorithm}[Speed Up Zig-Zag]\label{algorithm.suzz:1}
\text{ }
\begin{enumerate}
\item Set $t=0$.
\item Start from point $(X_{t},\Theta_{t})=(x,\theta) \in E$.
\item The process $(X_{t+u},\Theta_{t+u})$ moves according to the deterministic ODE system 
\begin{equation}\label{ODE:1}
\left\{\begin{array}{l}
\dfrac{d}{du}X_{t+u}=\dfrac{d\Phi_u(x,\theta)}{du}=\theta s(X_{t+u}) , u \geq 0 \vspace{1.5mm}\\ 
\dfrac{d}{du}\Theta_{t+u}=0, u \geq 0 \vspace{1.5mm}\\
X_{t}=x , \Theta_{t}=\theta.
\end{array}\right.
\end{equation}
\item For every coordinate $i \in \{ 1,...,d \}$ construct a Poisson Process with intensity $\{m_i(u)=\lambda_i(\Phi_u(x,\theta), \theta) , u \geq 0 \}$.
\item Let $\tau_i$ be the first arrival time of the $i$th Poisson Process, i.e. for all $t_0 \geq 0$, $\mathbb{P}(\tau_i \geq t_0)=\exp \left\{ -\int_0^{t_0}m_i(u)du \right\}$. Let $j=\argmin\{ \tau_i, i=1,...,d \}$ and $\tau=\tau_j$ the first arrival time of all the processes.
\item For $u \in [0,\tau)$ set $X_{t+u}=\Phi_u(x,\theta)$ and $\Theta_{t+u}=\theta$.
\item Set $t=t+\tau$, $x=\Phi_{\tau}(x,\theta)$, $X_{t}=x$ and $\Theta_{t}=F_j(\theta)$.
\item Repeat from the Step 2.
\end{enumerate}
\end{algorithm}

\section*{Supplement B: Proof of Lemma \ref{continuous.component.lemma:1}}

% \ref{continuous.component.lemma:1}

The proof is very similar in spirit to the proof of Lemma 8 (Continuous Component) of \cite{bierkens.roberts.zitt:2019}. The main idea is that since there is an admissible path from $(x,\theta) \looparrowright (y,\eta)$, the process has a positive probability to follow some path very close to the admissible path and therefore starting from somewhere close to $(x,\theta)$ to end up somewhere close to $(y,\eta)$. The difference is that this time we allow explosive deterministic dynamics. Therefore, contrary to the original Zig-Zag, the process is not guaranteed to stay inside a fixed ball for a large and fixed time horizon. This means that the Poisson thinning construction that the authors propose to get the result cannot be applied here directly. However, in order to get the result, we only need to consider paths that are close to the admissible path. These paths need to lie on a fixed ball where the hazard rates are bounded and therefore we can use Poisson thinning to construct this type of paths. This gives us the result as shown in the next proof.
\begin{proof}[Proof of Lemma \ref{continuous.component.lemma:1}]
Since $(x,\theta)\looparrowright(y,\eta)$ there exists an admissible control sequence $u=(t,\iota)=(t_0,...,t_m,i_1,...,i_m)$ starting from $(x,\theta)$ with $\Psi_u(x,\theta)=(y,\eta)$. Let $U'$ be a small neighbourhood of $x$ and $B(0,R)$ a ball large enough to contain the paths induced by $u$ starting from  $(x',\theta)$ for every $x' \in U'$. Also, pick $\epsilon$ small enough to ensure that for any $x' \in U'$, if the process starts from the end of any path induced by $u$ and $(x',\theta)$, i.e. from $\Psi_u(x',\theta)$, then the process a.s.\ stays inside the ball $B(0,R)$ until time $\epsilon$. Such an $\epsilon$ can be picked due to Lemma \ref{locality.lemma:1}. Since the rates of the process are continuous functions, they are bounded above in $B(0,R)$ say by $\bar{\lambda}$. Let $\tau_k=\sum_{i=0}^{k-1}t_i$ and let $\underline{\lambda}>0$  such that 
\begin{equation*}
\lambda_{min}(t,\iota)=\min_{k \in \{ 0,...,m-1 \}}\lambda_{i_k}(X_{\tau_k},\Theta_{\tau_k})>\underline{\lambda},
\end{equation*}
i.e. $\underline{\lambda}$ be a lower bound on the rates at the switching points of the admissible path. Then from continuity of the rates (and possibly by making $U'$ smaller) we can find $(U_k)_{k=0}^{m-1}$ non-intersecting small neighbourhoods of $\tau_k$ such that for any control sequence $(s,\iota)=(s_0,...,s_m,i_1,...,i_m)$ with the property that $\sum_{j=0}^{k}s_j \in U_{k}$ for all $k \in \{ 0,...,m-1 \}$ and for any starting point $(x',\theta)$ such that $x' \in U'$, we have
\begin{equation}\label{lower.bound.rates.again:1.5}
\lambda_{min}(s,\iota) \geq \underline{\lambda}>0,
\end{equation}
and the path induced by the starting point $(x',\theta)$ and the control sequence $(s,\iota)$ lies in $B(0,R)$. Let $T_1,T_2,...$ be the switching times of the process. On the event that there are $m$ switches before time $t$, we introduce the function 
\begin{equation*}
\Omega(x,\theta,t,T_1,...,T_m)=x+T_1\theta+(T_2-T_1)F_{i_1}(\theta)+...+(t-T_m)F_{i_1,...,i_m}(\theta),
\end{equation*}
that is the end point of the path until time $t$ when the order in which the coordinates change is $i_1,i_2,...,i_m$.

The random variable $\Omega(x,\theta,T_1,...,T_m)$ can be simulated using Poisson thinning. As a bounding process we can use the hazard rate $\Lambda(y,\eta)$, defined to be equal to $d\bar{\lambda}$ for all $x \in B(0,R)$ and defined to be equal to $\lambda(y,\eta)$ for all $y \notin B(0,R)$. Using the exponential representation of Poisson process along with Poisson thinning, we can construct the first $m$ switches of the process using i.i.d. $\tilde{E}_1,\tilde{E}_2,... \sim \exp(1)$ and i.i.d. $u_1,u_2,... \sim unif(0,1)$. We then simulate $T_1^{prop}$ from the bounding process $\Lambda$ such that 
\begin{equation*}
T_1^{prop}=\inf\{ t \geq 0 : \int_0^{t}\Lambda(X_s,\theta)ds \geq \tilde{E}_1 \}.
\end{equation*} 
We then accept this time on the event that $Acc_{1}=\{ u_1 \leq \lambda(X_{T_1^{prop}},\theta)/\Lambda(X_{T_1^{prop}},\theta) \}$. If we accept, we set $T_1=T_1^{prop}$. Else, we start again the process from $(X_{T_1^{prop}},\theta)$, we pick a new $T_2^{prop}$ according to $\tilde{E}_2$ and we add it to the previous $T_1^{prop}$. We keep doing that until we accept. Suppose that the $j_1$th proposal was the one that was accepted. Then we decide which coordinate of the velocity to change according to the value of $u_{j_1}$. More specifically, we change the $i_1$ coordinate (which is the first coordinate to be switched according to the control sequence $u$) when the event $B_{j_1}$ occurs, where $B_{j_1}=\{ u_{j_1} \leq \lambda_{i_1}(X_{T_1},\theta)/\Lambda(X_{T_1},\theta) \} \subset Acc_{j_1}$.

Using the same construction, we construct the second switching time $T_2$ and if $j_2$ is
the time to accept the switch, then we decide which coordinate to switch according to the value $u_{j_2}$ took and we change the $i_2$ coordinate (which is the second coordinate to be switched according to control sequence $u$) when the event $B_{j_2}=\{ u_{j_2} \leq \lambda_{i_2}(X_{T_2},F_{i_1}(\theta))/\Lambda(X_{T_2},F_{i_1}(\theta)) \} \subset Acc_{j_2}$ occurs.

Now, let us condition on the following event $B$: {\it $u_1,...,u_m$ took values less than $\underline{\lambda}/d\bar{\lambda}$ and $\tilde{E}_1,...,\tilde{E}_m$ took values such that for all $k \in \{ 1,...,m \}$, $\left(d\bar{\lambda}\right)^{-1}\sum_{i=1}^k\tilde{E}_i \in U_{k-1}$, where $U_k$ are the open subsets of $\mathbb{R}^+$ defined just before equation (\ref{lower.bound.rates.again:1.5}). Furthermore, if $\tau_m$ is the ending time of the admissible path given by the control sequence $u$, then $\tilde{E}_{m+1}$ took a value larger than $\left( \tau_m+\epsilon-\inf U_{m-1} \right) d\bar{\lambda}$, where $\epsilon$ was introduced in the beginning of the proof.
} 

Note that $B$ has a positive probability $c$ to occur that does not depend on the starting point of the path. Let us consider what happens if event $B$ occurs. In the beginning, the process moves in a straight line in direction $\theta$ and stays inside the ball of radius $B(0,R)$. This means that the bounding process is equal to $d\bar{\lambda}$ by construction. Since $a_1=\frac{1}{d\bar{\lambda}}\tilde{E}_1 \in U_1$, this means that $\tilde{E}_1=\int_0^{a_1}d\bar{\lambda}\ ds=\int_0^{a_1}\Lambda(X_s,\theta)ds$ and therefore $T_1^{prop}=a_1 \in U_1$. At the same time $u_1 \leq \underline{\lambda}/\bar{\lambda}$ therefore the first switch was accepted and the coordinate to be switched was $i_1$.

Using the same line of argument and since $\left(d\bar{\lambda}\right)^{-1}\left( \tilde{E}_1+ \tilde{E}_2 \right) \in U_2$ we can guarantee that until the second proposal time $T_2^{prop}$ the process will remain inside the ball $B(0,R)$ so the bounding rate will be $d\bar{\lambda}$ and therefore the proposed switching time will occur inside $U_2$. Since $u_2 \leq \underline{\lambda}/\bar{\lambda}$ the switch is accepted and the coordinate to switch is $i_2$.

Using induction we see that for all $k \in \{ 1,...,m \}$ the $k$th switching time $T_k$ occurred inside $U_k$ and the $i_k$ coordinate was the one to switch.

 Furthermore, from time $T_m$ until $\tau_m+\epsilon$, the process is guaranteed to not leave the ball $B(0,R)$ by construction of $\epsilon$. Therefore, the bounding process until $\tau_m+\epsilon$ is $d\bar{\lambda}$ and since $\tilde{E}_{m+1}>\left( \tau_m+\epsilon-\inf U_{m-1} \right) d\bar{\lambda}$ the process is guaranteed to not switch the velocity until time $\tau_{m+1}+\epsilon$. 

This means that given event $B$ occurs and if $T_1,T_2,...$ are the switching events, then for any $t \in [\tau_m,\tau_m+\epsilon]$, $X_t=\Omega(x,\theta,t,T_1,...,T_m)$. Furthermore, conditioning on $B$, all times $T_1,...,T_m$ are distributed according to first arrival times of the homogeneous bounding Poisson process, conditioned on taking values on the sets $U_0,U_1,...,U_{m-1}$. Therefore, $T_k \sim unif(U_{k-1})$ for all $k=1,...,m$. Conditioning on event $B$ occurring, we write
\begin{equation*}
\mathbb{P}_{x,\theta}\left( X_t\in \cdot , \Theta_t=\eta \right) \geq  c \ \mathbb{P}_{x,\theta} \left( \Omega(x,\theta,t,T_1,...,T_m) \in \cdot  , \Theta_t =\eta \right),
\end{equation*}
where  $T_k\sim unif(U_{k-1})$. 
Recall, that $c$ does not depend on the starting position $(x,\theta)$. We can use the same argument for every starting point $(x',\theta)$ for any $x' \in U'$ and get that for any $t \in [\tau_m,\tau_m+\epsilon]$,
\begin{equation}\label{continuity.lemma.equation:1.5}
\mathbb{P}_{x',\theta}\left( X_t\in \cdot , \Theta_t=\eta \right) \geq  c \ \mathbb{P}_{x',\theta} \left( \Omega(x',\theta,t,T_1,...,T_m) \in \cdot  , \Theta_t =\eta \right).
\end{equation}
Now, since $(x,\theta)\looparrowright (y,\eta)$ we can assume that $\{ 1,...,d \} \subset \{ i_1,...,i_m \}$. Therefore, for all $t \in [\tau_m,\tau_m+\epsilon]$, the (up to translation) linear map  $(u_1,...,u_m)\rightarrow \Omega(x,\theta,t,u_1,...,u_m)$ is of full rank since its matrix has column vectors
\begin{equation*}
\left \{ \theta-F_{i_1}(\theta),...,F_{i_1,...,i_{m-1}}(\theta)-F_{i_1,...,i_m}(\theta) \right \}= \left \{ \pm 2 e_{i_1},...,\pm 2 e_{i_m} \right \}= \left \{ \pm 2 e_1,..., \pm 2 e_d  \right \}.
\end{equation*}
From Lemma 6.3 of \cite{benaim.leborgne.malrieu.zitt:15} we get that there exists a neighbourhood $U_t$ of $x$, a neighbourhood $V_t$ of $y$ and  a constant $c'>0$, such that for all $x' \in U_t$,
\begin{equation*}
\mathbb{P}_{x',\theta} \left( \Omega(x',\theta,t,T_1,...,T_m) \in \cdot , \Theta_t =\eta\right) \geq c' \lambda(\cdot \cap V_t).
\end{equation*}
Now, since $\Omega(x',\theta,t,T_1,...,T_m)=x'+T_1\theta+(T_2-T_1)F_{i_1}(\theta)+...+(t-T_m)F_{i_1,...,i_m}(\theta)$, a change in $t$ effects on $\Omega$ as a translation in direction $F_{i_1,...,i_m}(\theta)$. This means that for every $t \in [\tau_m,\tau_m+\epsilon]$ if we pick a starting point $(x',\theta)$ with $x' \in U_{\tau_m}$ we get for all $A \in \mathcal{B}(\mathbb{R})$,
\begin{align*}
&\mathbb{P}_{x',\theta} \left( \Omega(x',\theta,t,T_1,...,T_m) \in A , \Theta_t =\eta\right)= \\
&=\mathbb{P}_{x',\theta} \left( \Omega(x',\theta,\tau_m,T_1,...,T_m) \in A-(t-\tau_m)F_{i_1,...,i_m}(\theta) , \Theta_t =\eta\right)\\
& \geq  c' \lambda(\left( A-(t-\tau_m)F_{i_1,...,i_m}(\theta) \right) \cap V_{\tau_m}) \geq c' \lambda(A \cap \left( V_{\tau_m}+(t-\tau_m)F_{i_1,...,i_m}(\theta) \right)).
\end{align*}
If $\epsilon>0$ is picked small enough then $\cap_{t\in [\tau_m,\tau_m+\epsilon]}\left( V_{\tau_m}+(t-\tau_m)F_{i_1,...,i_m}(\theta) \right)$ is not empty and contains an open set $V_x$. Then, for all $x' \in U_{\tau_m}$, and all $A \in \mathcal{B}(\mathbb{R})$,
\begin{equation*}
\mathbb{P}_{x',\theta} \left( \Omega(x',\theta,t,T_1,...,T_m) \in A , \Theta_t =\eta\right) \geq c' \lambda(A \cap V_x).
\end{equation*}
Overall, using (\ref{continuity.lemma.equation:1.5}), for all $x' \in U_{\tau_m} \cap U'$, for all $t \in [\tau_m,\tau_m+\epsilon]$ and all $A \in \mathcal{B}(\mathbb{R})$,
\begin{equation*}
\mathbb{P}_{x',\theta}\left( X_t \in A , \Theta_t=\eta \right) \geq c \ \mathbb{P}_{x',\theta} \left( \Omega(x',\theta,t,T_1,...,T_m) \in A , \Theta_t =\eta \right) \geq c \ c'\lambda(A \cap V_x),
\end{equation*}
which proves the result.
\end{proof}

\end{document}